\documentclass{elsart}
\usepackage{amssymb}
\usepackage{color}
\usepackage{graphicx}\usepackage{subfigure}%\usepackage{psfrag}
\usepackage{amsmath}
\usepackage{amsfonts}

\usepackage{epstopdf}
\usepackage[title]{appendix}
\allowdisplaybreaks
\usepackage{psfrag}

\usepackage{paralist}
\usepackage[colorlinks=true]{hyperref}

\usepackage{multirow}
\newtheorem{theorem}{Theorem}[section]
\numberwithin{equation}{section}
\numberwithin{figure}{section}
\numberwithin{table}{section}

\renewcommand{\vec}[1]{\mbox{\boldmath \small $#1$}}

\newtheorem{example}{Example}[section]
\newtheorem{lemma}{Lemma}[section]
\newtheorem{remark}{Remark}[section]
\numberwithin{equation}{section}
\numberwithin{figure}{section}
\numberwithin{table}{section}
\numberwithin{thm}{section}

\newenvironment{proof}[1][Proof]{\begin{trivlist}
\item[\hskip \labelsep {\bfseries #1}]}{\end{trivlist}}
\renewcommand{\qed}{\hfill \nobreak \ifvmode \relax \else
      \ifdim\lastskip<1.5em \hskip-\lastskip
      \hskip1.5em plus0em minus0.5em \fi \nobreak
      \vrule height0.75em width0.5em depth0.25em\fi}
\begin{document}

\begin{frontmatter}

\title{Runge-Kutta discontinuous local evolution Galerkin methods for the shallow water equations on the cubed-sphere}

\author{Yangyu Kuang},
\ead{kyy@pku.edu.cn}
\author{Kailiang Wu},
\ead{wukl@pku.edu.cn}
\author[label2]{Huazhong Tang}
\thanks[label2]{Corresponding author. Tel:~+86-10-62757018;
Fax:~+86-10-62751801.}
\ead{hztang@math.pku.edu.cn}
\address{HEDPS, CAPT \& LMAM, School of Mathematical Sciences, Peking University,
Beijing 100871, P.R. China}
 \date{\today{}}

\maketitle

\begin{abstract}
The paper develops high order accurate Runge-Kutta discontinuous local evolution Galerkin (RKDLEG) methods
on the cubed-sphere grid for the shallow water equations (SWEs). Instead of using the dimensional splitting method or solving  one-dimensional Riemann problem in the direction normal to the cell interface, the RKDLEG methods are built on genuinely multi-dimensional approximate local evolution operator of the locally linearized SWEs on a sphere by considering all bicharacteristic directions. Several numerical experiments  are conducted to demonstrate the accuracy and performance of our RKDLEG methods, in comparison to the { Runge-Kutta discontinuous Galerkin} method with Godunov's flux etc.
\end{abstract}

\end{frontmatter}

%\thispagestyle{plain} \markboth{ } {}
%%%%%%%%%%%%%%%%%%%%%%%%%%%%%%%%%%%%%%%%%%%%%%%%%%%% 5

\section{Introduction}
\label{sec:intro}

The shallow water equations (SWEs)
describe the motion of a thin layer of fluid held down by gravity.
%Atmospheric model is essential and important for numerical weather predictions and climatic simulations \cite{Chen:2014}. The governing equations are the fundamental atmospheric conservation laws for mass, momentum, and energy etc \cite{Lauter:2007}. We simplified them to the {\color{red}{SWEs}} on sphere.
%The SWE{\color{red}{s}} {\color{red}{constitute}} a hyperbolic system of conservation laws. It is one of the simplest nonlinear hyperbolic systems, covering the main features of the atmospheric model \cite{Lauter:2007}.
%
%Williamson summarized many forms in which the SWE{\color{red}{s}} on sphere can be written \cite{Williamson:1992}.
The SWEs on the sphere exhibit the major difficulties associated with the horizontal dynamical aspects of atmospheric modeling on the spherical earth and thus are important in
studying the dynamics of large-scale atmospheric flows
and developing numerical methods of more complex atmospheric models.
In comparison with the planar case, the main difficulties in solving the SWEs on the sphere
 {come} from the spherical geometry,
the choice of coordinates,   nonlinearity,
and the large scale difference between the horizontal  and  vertical motions of the fluids. %, with large horizontal scale and small vertical scale \cite{Ullrich:2010}. Finally, in climate simulation, we need numerical methods with good numerical stability to satisfy with the long term calculation.
High-order accurate numerical methods
are becoming increasingly popular in atmospheric modeling, but the numerical methods should be competent for long time simulation.
In order to evaluate numerical methods for the solutions of SWEs in spherical geometry, Williamson et al.
  proposed a suite of seven test cases and offered reference solutions to those tests obtained by using  a pseudo-spectral method \cite{Williamson:1992}.

%SWE{s} on sphere {require} solving nonlinear equations, and the analytic treatment of practical problems {\color{red}{is}} more difficult. Hence studying them numerically is the primary approach.

Representation of the spherical geometry plays an important role in solving SWEs on the sphere.
The latitude-longitude (LAT/LON) coordinates or grids are naturally and popularly chosen in the  early stage \cite{Bates: 1990,Lin:1997,McDonald:1989},  but the singularity at the poles leads to big numerical difficulty.
Overcoming such pole singularity  needs special numerical technique and boundary conditions \cite{Sadourny:1972}.
 To avoid  the pole singularity in the LAT/LON coordinates,
 other choices are the icosahedral hexagonal or triangular grids \cite{Lauter:2007,Lee:2008,Pudykiewicz: 2011,Tometa 2001}, Yin-Yang grid \cite{Kageyama:2004,Li:2008,Li:2012}, and cubed-sphere grid \cite{Chen:2008,Nair:2005-1,Putman: 2007,Ronchi:1996,Sadourny:1972,Ullrich:2010}.
 Comparisons of those frequently-used grids are given in \cite{Chen:2014,StCyr: 2008}.
An icosahedral-hexagonal grid on the sphere is created by dividing the faces of an icosahedron and projecting the vertices onto the sphere, thus it is non-quadrilateral and  unstructured.
%which brings about substantial difficulties in developing high-order global models,
The Ying-Yang grid is  overset  in spherical geometry
and composes of two identical component grids  combined in a complementary way
to cover a spherical surface with partial overlap on their boundaries %\cite{Kageyama:2004}
so that the interpolation should be used between two  component grids.  %\cite{Chen:2014,Peng:2006}.
%does not automatically guarantee
%the numerical conservativeness for the interpolation for communicating data
%across the Yin-Yang border
%or requires complicated numerical treatment to complete conservativeness
%across the Yin-Yang overlapping region.
The cubed-sphere grid is quasi-uniform  and easily generated by  dividing  the sphere into six identical regions
with the aid of projection of the sides of a circumscribed cube onto a spherical surface
and  choosing the coordinate lines on each region to be arcs of great circles.
%is popular now for its grid uniformity, convenient for numerical conservation, easiness to implement different numerical methods, and excellent parallel efficiency. % \cite{Putman: 2007}.
%So in this paper we employ the cubed-sphere grid for RKDLEG method.
The mainly existing  numerical methods  for  the SWE{{s}} on the sphere are as follows:
 %, reached high accuracy but resulted non-physical numerical oscillations and a large computational cost of Legendre transforms.
 finite-difference \cite{Bates: 1990,Sadourny:1972,Thuburn:1997,Tometa 2001},
  finite-volume   \cite{Lee:2008,Lin:1997,Yang:2010},
  %, but all of these rely on low-order classical finite-volume methods. To improve the order,
  multi-moment finite volume \cite{Chen:2014,Chen:2008,Li:2008,Li:2012},
    spectral transform   \cite{Jakob: 1995},
  %spectral-element and discontinuous-Galerkin method to get higher order.
  spectral element \cite{Giraldo:2005-1,Taylor: 1997,Thomas:2005},
  %and the main disadvantage of spectral-element method is lack of conservation which affects the accuracy in long time simulations \cite{Giraldo:2008,Nair:2005-2}.
   and discontinuous Galerkin (DG) methods \cite{Giraldo: 2002,Giraldo:2008,Lauter:2007,Nair:2005-1,Nair:2005-2}
  etc.
%Many of the models are based on the exact or approximate Riemann solvers.
 Most of them are built on the one-dimensional exact or approximate Riemann solver.
 %Due to its briefness, local Lax-Friedrich(or Rusanov) is a popular one-dimensional approximate Riemann solver used by \cite{Nair:2005-2,Giraldo:2008,Lauter:2007} et al, piecewise upwind approximate Riemann solver \cite{vanLeer1977} used by \cite{Lee:2008}, Osher's approximate Riemann solver \cite{Osher: 1983} used by \cite{Yang:2010}, and Ullrich compared Rusanov \cite{Rusanov:1961}, Roe \cite{Roe: 1981} and the new $AUSM^{+}-up$ \cite{Liou: 2006} approximate Riemann solver \cite{Ullrich:2010}. Unfortunately, such multi-{\color{red}{dimensional}} methods, based on one-dimensional (exact or approximate) Riemann solver, may lead to a misinterpretation of local wave structure in the solutions of the quasi-linear hyperbolic system \cite{Roe1986}. Hence, it is useful to develop genuinely multi-dimensional methods for SWE{\color{red}{s}} on sphere. Meanwhile, as long time period calculation requirement in climate simulation, we need numerical methods with good numerical stability to satisfy with the long term calculation. By Shu and Cheng, it was shown that the error of the {\color{red}{DG}} scheme will not grow for fine grids over a long time period in \cite{Cheng:2008,Cheng:2010,Zhang:2004} for linear hyperbolic equations. In our experiments, we can see this property of DG scheme is preserved in our model.

The aim of the paper is to develop {Runge-Kutta discontinuous local evolution Galerkin (RKDLEG)} methods for the SWEs on the cubed sphere. They are  the (genuinely) multi-dimensional   and  combining  the {Runge-Kutta discontinuous Galerkin (RKDG)} methods on the cubed-sphere  with the local evolution Galerkin (LEG) method, which
 is a modification and simplification  of the original  finite volume  evolution Galerkin (EG) method for  {multi-dimensional} nonlinear hyperbolic system \cite{LukacovaSaibertovaWarnecke2002,SunRen2009}.
The EG method generalizes  the Godunov method
by using an evolution operator coupling the flux formulation of each direction for the multi-dimensional hyperbolic system.
 The basic idea of the {{EG}} method {was} introduced in \cite{Morton1998},
and then it {was} developed for the linear hyperbolic system in \cite{LukacovaMortonWarnecke2000}
and  nonlinear hyperbolic systems in \cite{LukacovaMortonWarnecke2002,LukacovaSaibertovaWarnecke2002}.
The EG method is constructed by using the theory of bicharacteristics
in order to take all infinitely many directions of wave propagation
into account and give the exact and approximate evolution operators of the linearized hyperbolic system,
in other words, integrating the linearized hyperbolic system along its bicharacteristics
to obtain an equivalent integral system,
then making a suitable approximations of the integral system.
Similar bicharacteristic-type methods for hyperbolic system,
can be found in the early literature such as \cite{Butler1960,Huang1995,Johnston1972}.
The EG method may be considered as a genuinely multi-dimensional Godunov-type scheme,
in which the so-called approximate evolution operator is used to
get the explicit approximate solutions at each cell interface along all infinite bicharacteristics
of the linearized system,
instead of solving the local one-dimensional Riemann problem
in direction normal to each cell interface by any Riemann solver.
It has been used successfully for various physical applications,
e.g. the wave propagation in heterogeneous media \cite{Lukacova:2009},
the Euler equations of gas dynamics \cite{LukacovaMortonWarnecke2004}, and the SWEs \cite{Lukacova:2011,Lukacova:2013,LukacovaHundertmark:2011,LukacovaNoelleKraft2004} with well-balanced property with or without dry beds.
A survey of finite volume {EG} method {was} presented in \cite{LukacovaMortonWarnecke2010}.
The LEG method {was} proposed in \cite{SunRen2009} to simplify the evaluation of the EG
numerical fluxes
by taking the limit of the approximate evolution operator  at time level $t_{n}$
as the time $t_{n}+\tau$ approaches $t_{n}$. It has been  successfully extended  to the relativistic hydrodynamics \cite{Wu:2014}.
%Although all above-mentioned EG methods are in the finite volume frame,
{Few attensions were} paid to a combination of the {RKDG} methods
with the {{EG}} operator for the highly accurate simulation of  planar compressible flows, see e.g. \cite{Lukacova:2012,Lukacova:2014}.
%Although the {EG} method or {{LEG}} method
%is applied for various equations with finite volume method or {\color{red}{DG}} method,
%its application on the system on sphere is untouched.
%
It is challenging to extend the EG or LEG method to the SWE{{s}} in the spherical
geometry.
%we want to derive the evolution operator {\color{red}{of SWEs}} on sphere.
Due to the complex geometry and the choice of coordinates,
the derivation of evolution operator is much more complicate
than the planar case.
%and due to the discontinuity of the transformation
 %rom the cube face  to the sphere surface across the edge of cube face,
%we need special treatment on the edges of cube face.
%To simplify the operator,
%we choose to use the {\color{red}{LEG}} method as \cite{SunRen2009,Wu:2014}.
%And for the reason proposed in the last paragraph,
%we use the {\color{red}{LEG}} method for the numerical fluxes calculation
%in the {\color{red}{DG}} method in this paper.

The paper is organized as follows.
 Section \ref{sec:Preliminaries}
 introduces the cubed-sphere grid and spherical SWEs and derives the exact evolution
 operator of the locally linearized spherical SWEs in the reference coordinates.
 %
 %\ref{subsec:cube} introduces the cubed-sphere grid
 %and Section \ref{subsec:equations} gives several form of SWE{\color{red}{s}} on sphere.
 %Section \ref{subsec:exact operator} derives the exact evolution operator of the linearized {\color{red}{SWEs}} on sphere.
 Section \ref{sec:method} presents our RKDLEG method in the reference coordinates,
 including  the DG spatial discretization in Section \ref{subsubsec:DG},
  Runge-Kutta time discretization in Section \ref{subsubsec:RK},
  and the approximate evolution operators  in the cubed sphere face in Section \ref{subsec:Approximate},
where a special treatment is given  for  the points on the
 edges of cubed sphere  face in order to   % in Section \ref{subsubsection:edges}
 preserve the conservation of numerical fluxes there.
 The approximate local evolution operator  is equal to
 the limit of the approximate evolution operator as the time parameter tends to zero,
 %  A proof of conservation of the fluxes is in Section \ref{subsec:conservation}.
 Section \ref{sec:numerical-results} conducts several numerical experiments
 to demonstrate the accuracy and effectiveness
 as well as the multi-dimensional behavior
 of the proposed RKDLEG method.
 Section \ref{sec:conclud} concludes the paper. % with several remarks.

% % % % % % % % % % % % % % % % % % % % % % % % % % % % % % % % % % % % % % % % % % % %
% % % % % % % % % % % % % % % % % % % % % % % % % % % % % % % % % % % % % % % % % % % %
\section{Preliminaries and notations}
\label{sec:Preliminaries}
This section introduces the cubed-sphere grid and SWE{s},
and derives  the exact evolution operator of
the locally linearized SWE{s}  in spherical geometry.

\subsection{The cubed-sphere grid}
\label{subsec:cube}
%[Sadourny:1972}] A class of conservative finite-difference approximations of the primitive equations
%is given for quasi-uniform spherical grids derived from regular polyhedrons.
%The earth is split into several contiguous regions. Within
%each region, a coordinate system derived from central
%projections is used, instead of the spherical coordinate
%system, to avoid the use of inconsistent boundary conditions at the poles.
%The presence of artificial internal boundaries has no effect
%on the conservation properties of the approximations.
%Examples of conservative schemes, up to the second order in the case of a cube, are given.
%A selective damping operator is needed to remove the
%two-grid interval waves generated by the existence of internal boundaries.

%[M. Taylor et al.JCP1997]

%Each of  6 identical subregions $P_i$ of the sphere
%is divided into $N\times N$ uniform cells in $\left(x,y\right)$ plane.

Since the cubed-sphere grid {was} proposed,
it  has widely been  used in the literature because the boundary conditions
at  the pole in spherical coordinates are seen to vanish in the finite-difference formulation \cite{Sadourny:1972}.
%and the simplest schemes are mass- and energy- conserving.
%Consistency as well as computational efficiency
%are ewy to obtain in the absence of singularities.
It {was} numerically demonstrated that the cubed-sphere grid with gnomonic (equiangular central) projection
{was} an excellent choice for high-order accurate numerical methods in global
modeling applications, see e.g. \cite{Nair:2005-1,Nair:2005-2} etc.

The sphere with radius of $R$ is decomposed into 6 identical cubed-sphere faces
$\{\rm P_{i}, i=1,2,\cdots,6\}$ by
using the central (gnomonic) projection of an inscribed cube
with side $\frac{2\sqrt{3}}{3}R$ \cite{Sadourny:1972},
see Fig. \ref{fig:sphere} (a),
where the thick line denotes the edge of cubed-sphere face.
There are  two different central projections:
employing the local Cartesian coordinates \cite{Sadourny:1972} and the
equiangular (central) coordinates \cite{Ronchi:1996,Thomas-Loft2002}.
%(Thomas and Loft 2002).
%
In the equiangular projection,  each  face of cube or cubed-sphere may be mapped to a
reference region $\tilde{\Omega}=[-\frac{\pi}{4}, \frac{\pi}{4}]\times [-\frac{\pi}{4}, \frac{\pi}{4}]$,
in which the equiangular (central)  coordinates are denoted by  $x$ and $y$ here.
As an example,
%We use equiangular projection \cite{Sadourny:1972}.
%the point on  the cubed-sphere face $\rm P_{1}$
%can be determined by the local coordinate $\left(x,y\right)$ \cite{Sadourny:1972}.
%For each face of the cube, $x,y\in\left[-\frac{\pi}{4},\frac{\pi}{4}\right]$.
the mapping relation between the cubed-sphere face $\rm P_{1}$
and the reference region $\tilde{\Omega}$
%are shown in (\cite{Chen:2008} Fig. 2 and \cite{Nair:2005-1} Fig. 1).
%The transformation between the LAT/LON coordinate $\xi \in \left[-\pi ,\pi \right),\eta \in \left[-\frac{\pi }{2},\frac{\pi }{2}\right]$, with , and
is given by %can be derived from the geometrical relations
\begin{equation}
\label{tranform}
x=\xi,\quad y=\arctan \left(\frac{\tan\eta }{\cos \xi }\right),
\end{equation}
where
$\eta \in \left[-\arctan\left(|\cos\xi|\right),\arctan\left(|\cos\xi|\right)\right]\subset  \left[-\frac{\pi }{2},\frac{\pi }{2}\right]$,
$\xi \in \left[-\frac{\pi}{4} ,\frac{\pi}{4} \right]\subset  \left[-\pi ,\pi \right)$, and
 $\xi$ and $\eta$ denote the longitude and   latitude  (LAT/LON) coordinates.
In fact, the transformation rules in \eqref{tranform} are also satisfied for any face
of the cubed-sphere
in rotated  LAT/LON coordinates $(\xi,\eta)$
with the origin located at the center of corresponding face.
With the application of the transformations between
the rotated and  original LAT/LON coordinates,
the transformation rules for other faces of the cubed-sphere can be obtained,
see Appendix A of \cite{Nair:2005-1} for a detailed description.

Divide $\tilde{\Omega}$ into a { square} grid, see Fig. \ref{fig:sphere} (b),
and then such grid is inversely mapped to the cubed-sphere face to get
the cubed-sphere grid, see its schematic diagram in Fig. \ref{fig:sphere} (c).
%
%On the sphere, we denote longitude by $\eta$, latitude
%by $\xi$, and the associated unit vectors by  and u. We will
%denote the derivative of the mapping from (x, y) R (l, u)
%by the matrix D,  %[M. Taylor et al.JCP1997]
It is worth noting that  the equiangular projection generates more uniform grid on the sphere as
opposed to the equidistant projection \cite{Nair:2005-1},
the reference coordinate system is free of pole singularities, and all grid lines on the sphere are
great-circle arcs. However,  the transformation from $\tilde{\Omega}$ to the sphere
is not conformal  and such central mapping creates identical non-orthogonal curvilinear coordinates on
each face of the cubed-sphere.
%Edges of the cube-faces are discontinuous.
% in Equiangular projection, $(x,y)$ denotes the reference coordinates;
%while in equidistant projection, $(x,y)$ denotes the local coordinates;

\begin{figure}[htbp]
  \centering
  \subfigure[{\small Cubic subdivision of  sphere}]{
  \begin{minipage}{4cm}
  \includegraphics[width=4cm,height=4cm]{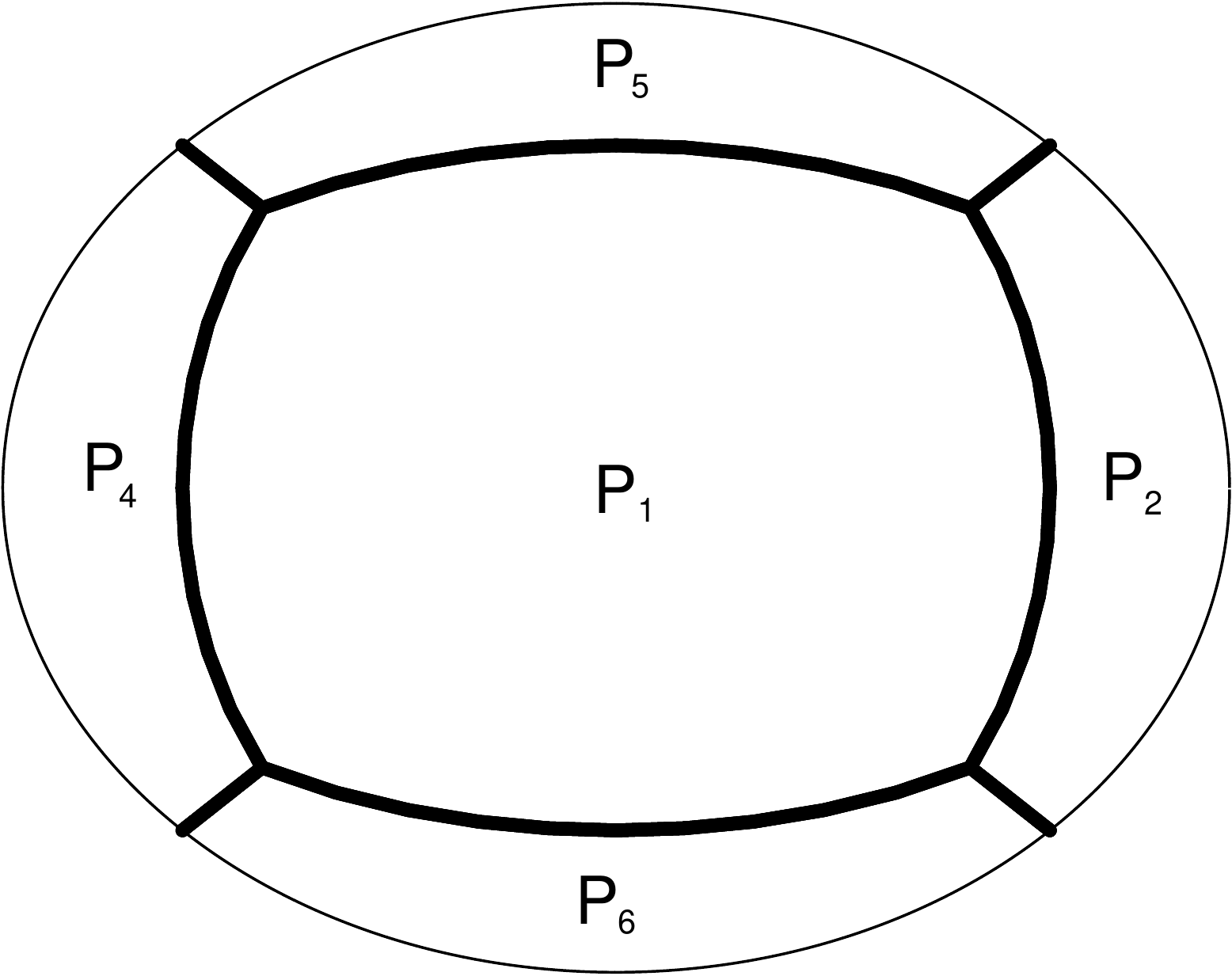}
 %   \psfrag{P}{$P_1$}
  \end{minipage}}\quad \quad
    \subfigure[{\small Square mesh in $\tilde{\Omega}$}]{
    	\begin{minipage}[c]{4cm}
    		\includegraphics[width=4cm,height=4cm]{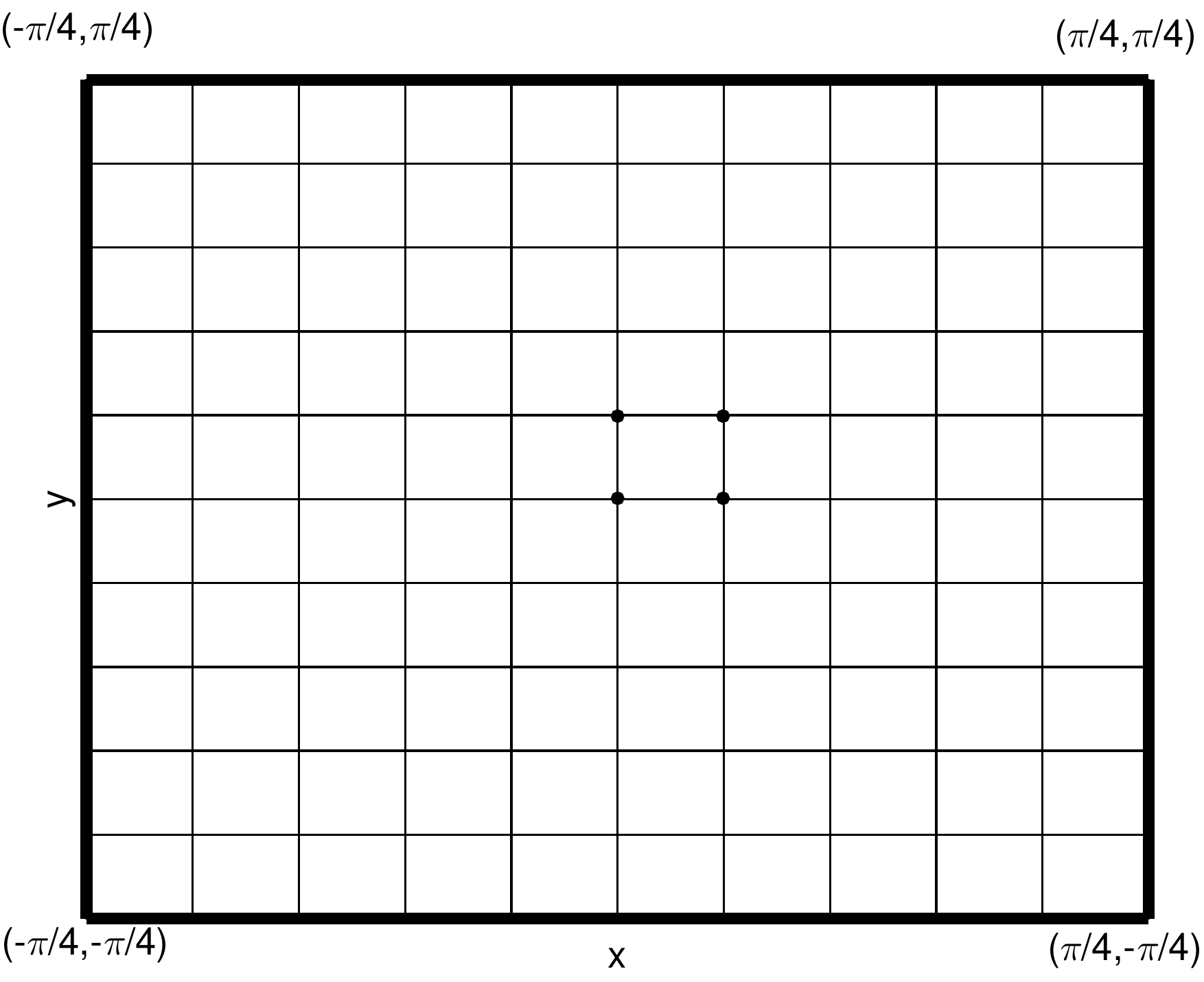}
    	\end{minipage}}\qquad
  \subfigure[{\small Cubed-sphere grid}]{
  \begin{minipage}{4cm}
  \includegraphics[width=4cm,height=4cm]{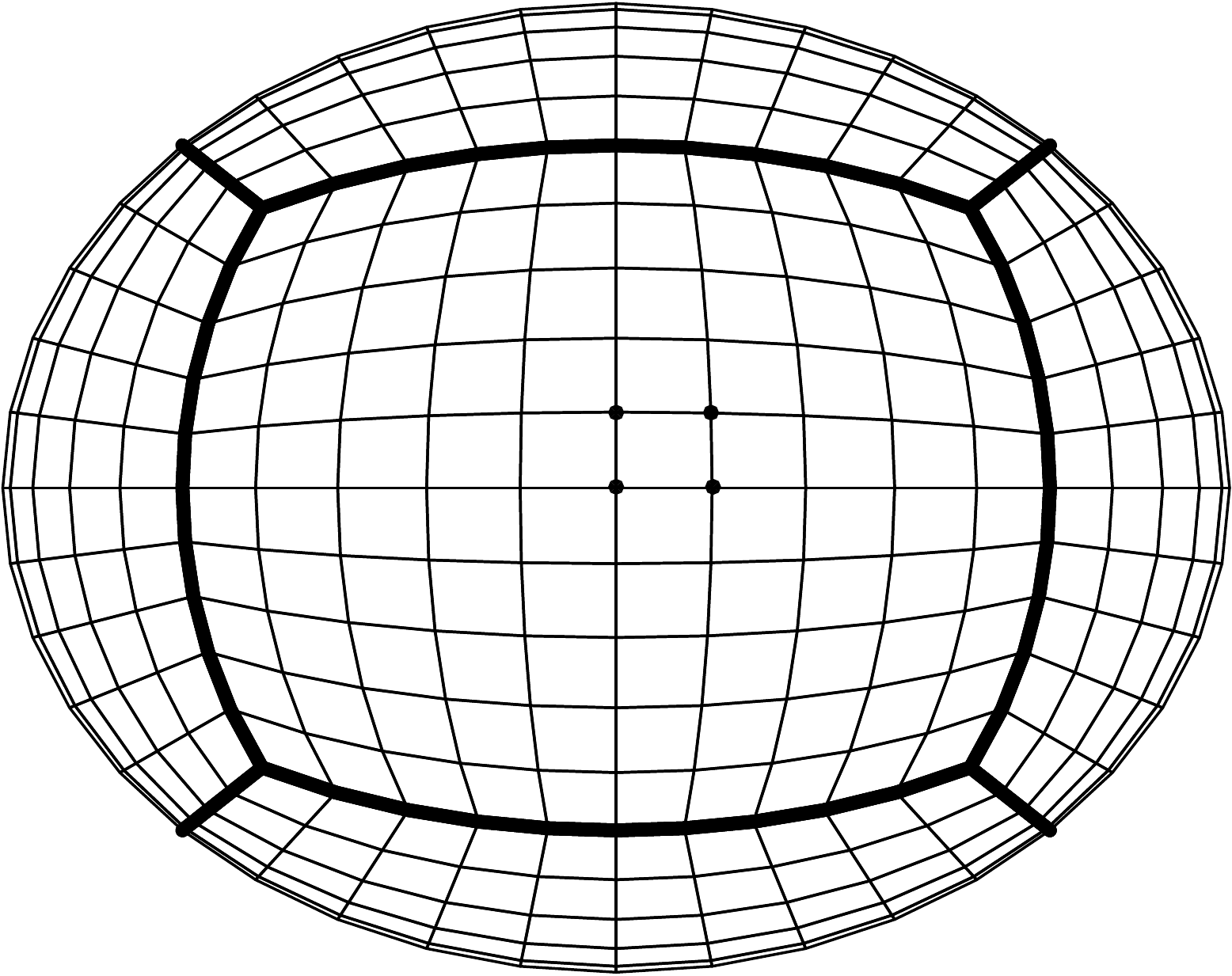}
  \end{minipage}}
  \caption{Schematic diagram of  the
  cubed-sphere and  cubed-sphere grid.}
  \label{fig:sphere}
\end{figure}

%General curvilinear treatment (tensor form) of the coordinates necessitates two
%sets of variables (vectors) viz., covariant ($x_i$ ) and contravariant ($x^i$ ).

It is now possible to compute the metric tensor and corresponding quantities using \eqref{tranform}.
Let $\vec r$ be the position vector of the point on the spherical surface,
and  $\vec{v}$ denote the velocity vector of the fluid on the sphere satisfying
 $\vec{v}\cdot\vec{k}=0$, where $\vec{k}$ is the outward unit normal vector of the spherical surface.
The covariant velocity $\left(\hat{u},\hat{v}\right)$ may be calculated by $\hat{u}=\vec{v}\cdot\frac{\partial \vec{r}}{\partial x},\hat{v}=\vec{v}\cdot\frac{\partial \vec{r}}{\partial y}$,
while the contravariant velocity $\left(u,v\right)$ may be given
by solving $\vec{v}=u\frac{\partial \vec{r}}{\partial x}+v\frac{\partial \vec{r}}{\partial y}$ and  $\vec{v}\cdot\vec{k}=0$,
where $\frac{\partial \vec{r}}{\partial x}$ and $\frac{\partial \vec{r}}{\partial y}$ denote the covariant base vectors of the transformation between the reference region $\tilde{\Omega}$
and spherical surface.
If using $\left(u_{s},v_{s}\right)$ to denote the velocity in LAT/LON coordinates $(\xi,\eta)$, that is,
 $u_{s}$ and $v_{s}$ are the longitude and latitude components of the velocity respectively,
then the relations among $\left(u_{s},v_{s}\right)$, $\left(u,v\right)$ and $\left(\hat{u},\hat{v}\right)$ can be given by
\begin{equation}
\label{veclocity}
\vec{A}\begin{pmatrix}
u\\
v
\end{pmatrix}=
\begin{pmatrix}
u_{s}\\
v_{s}
\end{pmatrix},
\quad
\vec{A}^{T}\begin{pmatrix}
u_{s}\\
v_{s}
\end{pmatrix}=
\begin{pmatrix}
\hat{u}\\
\hat{v}
\end{pmatrix},
\quad
\vec{A}=\begin{pmatrix}
   R\xi_{x}\cos\eta  & R\xi_{y}\cos\eta   \\
   R\eta_{x} & R\eta_{y}  \\
\end{pmatrix}
,\end{equation}
where $R$ is the radius of the sphere.
The metric tensor $\vec{G}$ for the above transformation
can be obtained by
\begin{equation}
\label{metric}
   \vec{G}=
   \begin{pmatrix}
   {g}_{11} & {g}_{12}  \\
   {g}_{21} & {g}_{22}  \\
   \end{pmatrix}    =\vec{A}^{T}\vec{A}=
    \frac{R^{2}}{\rho^{4}\cos^{2}x\cos^{2}y}
   \begin{pmatrix}
    1+\tan^{2}x & -\tan x\tan y  \\
   -\tan x\tan y & 1+\tan^{2}y  \\
    \end{pmatrix},
\end{equation}
and its inverse is
\begin{equation}
\vec{G}^{-1}=
\begin{pmatrix}
   {g}^{11} & {g}^{12}  \\
   {g}^{21} & {g}^{22}  \\
\end{pmatrix}
=\frac{1}{\det { (\vec{G})}}
\begin{pmatrix}
   {g}_{22} & -{g}_{12}  \\
   -{g}_{21} & {g}_{11}  \\
\end{pmatrix},
\end{equation}
where ${{\rho }^{2}}=1+{{\tan }^{2}}x+{{\tan }^{2}}y$.
It is worth noting that  the metric tensor  has the same form for each face of the cubed-sphere \cite{Nair:2005-2},
and   some  special numerical treatments are required around the edges  of the cubed-sphere face,
see Section \ref{sec:method},
because the coordinate transformation across the edges of the cubed-sphere face is not continuous.

%which may degrade the accuracy of numerical solution or cause non-physical oscillation.

%\begin{remark}
%The cubed-sphere grid is generated by mapping the sphere onto an spherically inscribed cube using the equiangular projection.
%The six same cubed-sphere faces  connected to each other by 12 edges called
%the edges of the cubed-sphere face.
%We obtain more uniform spacing by the projection, but the mesh on each cubed-sphere face is not orthogonal.
%\end{remark}

% % % % % % % % % % % % % % % % % % % % % % % % % % % %
\subsection{Governing equations}
\label{subsec:equations}

The spherical shallow water equations in the LAT/LON coordinates
$(\xi,\eta)$ may be written as follows \cite{Williamson:1992}
%
%We consider SWE{\color{red}{s}} on sphere. Considering the continuous equation and the momentum equations on Cartesian coordinates in , on , the {\color{red}{SWEs}} in spherical component form {\color{red}{are}} given as follows \cite{Williamson:1992},
\begin{equation}
\label{eq:latlon}
 \begin{aligned}
  & \frac{\partial h}{\partial t}+\frac{1}{R\cos \eta }\left[\frac{\partial }{\partial \xi }\left(h{u}_{s}\right)+\frac{\partial }{\partial \eta }\left(h{v}_{s}\cos \eta\right)\right]=0, \\
 & \frac{\partial {u}_{s}}{\partial t}+\frac{1}{R\cos \eta }\left[\frac{\partial \left(g\left(h+b\right)\right)}{\partial \xi }+\left({u}_{s}\frac{\partial {u}_{s}}{\partial \xi }+{{v}_{s}}\cos \eta \frac{\partial {u}_{s}}{\partial \eta }\right)\right]-f{v}_{s}-\frac{{u}_{s}\tan \eta }{R}{v}_{s}=0, \\
 & \frac{\partial {v}_{s}}{\partial t}+\frac{1}{R\cos \eta }\left[\cos \eta \frac{\partial \left(g\left(h+b\right)\right)}{\partial \eta }+\left({u}_{s}\frac{\partial {v}_{s}}{\partial \xi }+{{v}_{s}}\cos \eta \frac{\partial {v}_{s}}{\partial \eta }\right)\right]+f{u}_{s}+\frac{{u}_{s}\tan \eta }{R}{u}_{s}=0, \\
\end{aligned}
\end{equation}
where $b$ denotes the height of the bottom mountain, $h$  is the height of the fluid over the bottom mountain, $u_{s}$ and $v_{s}$ are two velocity components in the longitude $\xi$ and   latitude $\eta$ directions respectively, $g$ is the gravitational constant, $f$ is the Coriolis parameter defined by $f=2\Omega \sin\theta$, and $\Omega=7.292\times 10^{-5} \mbox{ s}^{-1}$ is the angular speed of the Earth's rotation.

%The same reference coordinates $(x,y)$ are used for the six faces of the cubed-spahere.
%and the governing equations can be written into the same expressions.
Under the transformation between the LAT/LON coordinates $(\xi,\eta)$
and the reference coordinates $(x,y)$, given in Section \ref{subsec:cube},
the   SWE{s} \eqref{eq:latlon} may be  transformed into the following divergence form \cite{Yang:2010}
\begin{equation}
\label{eq:cube}
\frac{\partial \vec{U} }{\partial t} + \frac{\partial \vec{F}_{1} }{\partial x} +\frac{\partial \vec{F}_{2} }{\partial y} = \vec{S}_0,
\end{equation}
where
\[
\vec{U}=
\begin{pmatrix} \Lambda h \\ \Lambda hu\\ \Lambda hv\end{pmatrix},
\
\vec{F}_{1} =
\begin{pmatrix}
\Lambda hu \\ \Lambda\left(hu^{2}+\frac{1}{2}gg^{11}h^{2}\right)\\ \Lambda \left(huv+\frac{1}{2}gg^{12}h^{2}\right)
\end{pmatrix},
\
\vec{F}_{2}=
\begin{pmatrix} \Lambda hv \\ \Lambda\left(huv+\frac{1}{2}gg^{12}h^{2}\right)\\ \Lambda \left(hv^{2}+\frac{1}{2}gg^{22}h^{2}\right)\end{pmatrix},
\
\vec{S}_0=
\begin{pmatrix} 0\\ \Lambda S_{0}^{(1)}\\ \Lambda S_{0}^{(2)}\end{pmatrix},
\]
Here  $\left(u,v\right)$ denotes the contravariant velocity vector,
$\Lambda =\sqrt{\det {(\vec{G})}}$ is the Jacobian of the  transformation, $\vec{G}$ is defined in \eqref{metric}, and
\begin{align*}
S_{0}^{(1)}&=-\Gamma_{11}^{1}hu^{2}-2\Gamma_{12}^{1}huv-f\Lambda\left(g^{12}hu-g^{11}hv\right)-gh\left(g^{11}b_{x}+g^{12}b_{y}\right),\\
S_{0}^{(2)}&=-{ \Gamma_{22}^{2}}hv^{2}-2\Gamma_{12}^{2}huv-f\Lambda\left(g^{22}hu-g^{12}hv\right)-gh\left(g^{12}b_{x}+g^{22}b_{y}\right),\end{align*}
in which  the Christoffel symbols  are given by \cite{Yang:2010}
\[{
\Gamma _{11}^{1}=\frac{2\tan x~{{\tan }^{2}}y}{{{\rho }^{2}}},\quad \Gamma _{12}^{1}=-\frac{\tan y}{{{\rho^2}} {{\cos }^{2}}y}, \quad \Gamma _{12}^{2}=-\frac{\tan x}{{{\rho^2}} {{\cos }^{2}}x}, \quad
 \Gamma _{22}^{2}=\frac{2{{\tan }^{2}}x\tan y}{{{\rho^2}} }.}
\]

%We consider equations \eqref{eq:cube}.

When  the solutions are smooth,   \eqref{eq:cube} is equivalent
to the following primitive variable form
\begin{equation}
\label{eq:primitive}
\frac{\partial \vec{V} }{\partial t} + \vec{A}_{1}(\vec{V},\vec{x}) \frac{\partial \vec{V} }{\partial x} + \vec{A}_{2}(\vec{V},\vec {x}) \frac{\partial \vec{V} }{\partial y} = \vec{S}_{1},
\end{equation}
where $\vec{V}=\left(h,u,v\right)^{T}$, $\vec{x}=\left(x,y\right)$,
 and
\[
\vec{A}_{1}(\vec{V},\vec{x})= \begin{pmatrix}
u &h &0\\
gg^{11} & u & 0\\
gg^{12} & 0 & u
\end{pmatrix},
\quad
\vec{A}_{2}(\vec{V},\vec{x})= \begin{pmatrix}
v &0 &h\\
gg^{12} & v & 0\\
gg^{22} & 0 & v
\end{pmatrix}
. \]
Here the source term { $\vec{S}_{1}= \vec{S}_{1}^{B}-\vec{S}_{1}^{S}$ with
 $\vec{S}_{1}^{B}:=(0, S_{1}^{B(1)}, S_{1}^{B(2)})^{T} $ and $\vec{S}_{1}^{S}:=\left(0, S_{1}^{S(1)}, S_{1}^{S(2)}\right)^{T}$},
in which
\begin{align*}
{S_{1}^{B(1)}}=&g\left(g^{11}\frac{\partial b}{\partial x}+g^{12}\frac{\partial b}{\partial y}\right), \quad
{S_{1}^{B(2)}}=g\left(g^{12}\frac{\partial b}{\partial x}+g^{22}\frac{\partial b}{\partial y}\right),\\
S_{1}^{S(1)}=&\Gamma_{11}^{1}u^{2}+2\Gamma_{12}^{1}uv+f\Lambda\left(g^{12}u-g^{11}v\right)\\
&+\frac{1}{2}gh\left(\frac{\partial g^{11}}{\partial x}+\frac{\partial g^{12}}{\partial y}\right)+\frac{1}{2\Lambda}gh\left(\Lambda_{x}g^{11}+\Lambda_{y}g^{12}\right),\\
S_{1}^{S(2)}=&\Gamma_{22}^{2}v^{2}+2\Gamma_{12}^{2}uv-f\Lambda\left(g^{22}u-g^{12}v\right)\\
&+\frac{1}{2}gh\left(\frac{\partial g^{12}}{\partial x}+\frac{\partial g^{22}}{\partial y}\right)+\frac{1}{2\Lambda}gh\left(\Lambda_{x}g^{12}+\Lambda_{y}g^{22}\right).
\end{align*}

The system \eqref{eq:primitive}
or \eqref{eq:cube} is hyperbolic in time, and
the linearized version of   \eqref{eq:primitive}
becomes the start point of the approximate local evolution operator in our RKDLEG methods for the SWEs \eqref{eq:cube}, see Section \ref{subsec:exact operator}.

\begin{lemma}[Hyperbolicity in time]\label{lemma:hyper}

For all admissible states $\vec{V}$ and any real angle $\theta$,
the matrix $\vec{A}(\vec{V},\vec{x}{;\theta}):= \vec{A}_{1}(\vec{V},\vec{x})\cos\theta+\vec{A}_{2}(\vec{V},\vec{x})\sin\theta$ may be  diagonalized as
\[
\vec{A}(\vec{V},\vec{x}{;\theta})= \vec{R}(\vec{V},\vec{x}{;\theta}) \vec{\Lambda}(\vec{V},\vec{x}{;\theta}) \vec{L}\left(\vec{V},\vec{x}{;\theta}\right),
%\quad %\vec{L}\left(\vec{V},\vec{x}{\color{red};\theta}\right)\vec{R}\left(\vec{V},\vec{x}{\color{red};\theta}\right)=\vec{I},
\]
where $\vec{\Lambda}(\vec{V},\vec{x}{;\theta})$ is a diagonal matrix with three real entries
\begin{equation}
\label{lambda}
\lambda^{(1)}(\vec{V},\vec{x}{;\theta})=v_{\theta} -cK_{\theta},
\quad
\lambda^{(2)}(\vec{V},\vec{x}{;\theta})=v_{\theta},
\quad
\lambda^{(3)}(\vec{V},\vec{x}{;\theta})=v_{\theta} +cK_{\theta},
\end{equation}
and the matrix $\vec{L}(\vec{V},\vec{x}{;\theta})$ and its inverse $\vec{R}(\vec{V},\vec{x}{;\theta})$ are given by
\begin{align}
%\label{lvector}
\vec{L}(\vec{V},\vec{x}{;\theta})=
\begin{pmatrix}
-\frac{1}{2} & \frac{c\cos\theta}{2gK_{\theta}} & \frac{c\sin\theta}{2gK_{\theta}} \\
0 & \frac{G_{s}(\theta)}{K_{\theta}} & -\frac{G_{c}(\theta)}{K_{\theta}} \\
\frac{1}{2} & \frac{c\cos\theta}{2gK_{\theta}} & \frac{c\sin\theta}{2gK_{\theta}}
\end{pmatrix},
%=: \begin{pmatrix}
%\vec{L}^{(1)}\left(\vec{V},\vec{x}{\color{red};\theta}\right)\\
%\vec{L}^{(2)}\left(\vec{V},\vec{x}{\color{red};\theta}\right)\\
%\vec{L}^{(3)}\left(\vec{V},\vec{x}{\color{red};\theta}\right)
%\end{pmatrix},
\
  \vec{R}(\vec{V},\vec{x}{;\theta})=
%&=\vec{L}^{-1}\left(\vec{V},\vec{x}{\color{red};\theta}\right)=
\begin{pmatrix}
-1 &0 &1\\
\frac{g}{c}G_{c}(\theta) & \sin\theta & \frac{g}{c}G_{c}(\theta)\\
\frac{g}{c}G_{s}(\theta) & -\cos\theta & \frac{g}{c}G_{s}(\theta)
\end{pmatrix}.
%\\
%&=:\left(\vec{R}^{(1)}\left(\vec{V},\vec{x}{\color{red};\theta}\right),
%\vec{R}^{(2)}\left(\vec{V},\vec{x}{\color{red};\theta}\right),\vec{R}^{(3)}\left(\vec{V},\vec{x}{\color{red};\theta}\right)\right).
\label{rvector}
\end{align}
Here
\begin{equation}
\label{kth,vth,c}
K_{\theta}=\sqrt{g^{11}\cos^{2}\theta+g^{12}\sin2\theta+g^{22}\sin^{2}\theta},\quad
v_{\theta}=u\cos\theta+v\sin\theta,\quad
c=\sqrt{gh},
\end{equation}
and
\[
G_{c}(\theta)=\frac{g^{11}\cos\theta+g^{12}\sin\theta}{K_{\theta}},\quad
G_{s}(\theta)=\frac{g^{12}\cos\theta+g^{22}\sin\theta}{K_{\theta}}
.\]
\end{lemma}
%\begin{proof}
The proof of this lemma is trivial and omitted here.
%\qed\end{proof}

%The scalar $\lambda^{(\ell)}\left(\vec{V},\vec{x}{\color{red};\theta}\right)$, the row vector
%$ \vec{L}^{(\ell)}\left(\vec{V},\vec{x}{\color{red};\theta}\right)$, and the column %vector
%$ \vec{R}^{(\ell)}\left(\vec{V},\vec{x}{\color{red};\theta}\right)$ are the $\ell$th eigenvalue and corresponding left and right eigenvectors of $\vec{A}\left(\vec{V},\vec{x}{\color{red};\theta}\right) $, $\ell = 1,2,3$.

\subsection{Exact evolution operator}
\label{subsec:exact operator}
This section derives the exact evolution operator of the locally linearized SWE{{s}}, or equivalently, integrates the locally linearized SWE{{s}} along their bicharacteristics to give an equivalent integral system.

%We take SWE{{s}} on the cubed-sphere \eqref{eq:primitive} as an example.

Use $\vec{\tilde{x}}=\left(\tilde{x},\tilde{y}\right)$ and $\vec{\tilde{V}} =\left(\tilde{h},\tilde{u},\tilde{v}\right)^{T}$  to denote the reference position and state of the vector  $\vec{V}$ in \eqref{eq:primitive}, and linearize the system
\eqref{eq:primitive} as follows
\begin{equation}
\label{eq:linear}
\frac{\partial \vec{ V } }{\partial t} + \vec{A}_{1}(\vec{ \tilde{V}  },\vec{\tilde{x}}) \frac{\partial \vec{ V } }{\partial x} + \vec{A}_{2}(\vec{ \tilde{V}  },\vec{\tilde{x}}) \frac{\partial \vec{  V } }{\partial y} = \vec{{S}}_{1}(\vec{V},\vec{x}).
\end{equation}
It is obvious that \eqref{eq:linear} is still hyperbolic in time thanks to {{Lemma}} \ref{lemma:hyper}. For the sake of convenience, we will shorten notations in the following such as $\vec{\tilde{A}}_{i}:=\vec{A}_{i}(\vec{\tilde{V}},\vec{\tilde{x}}) $,
$\vec{\tilde{L}}(\theta):=\vec{L}(\vec{\tilde{V}},\vec{\tilde{x}}{;\theta}) $,
$\vec{\tilde{R}}(\theta):=\vec{R}(\vec{\tilde{V}},\vec{\tilde{x}}{;\theta}) $,
$\tilde{\lambda}^{(\ell)}(\theta):=\lambda^{(\ell)}(\vec{\tilde{V}},\vec{\tilde{x}}{;\theta}) $, and so on.

Multiplying the system \eqref{eq:linear} from the left by $ \vec{\tilde{L}}(\theta)$
gives its characteristic form
\begin{equation}
\label{eq:charact}
\frac{\partial \vec{ W } }{\partial t} + \vec{\tilde{B}}_{1}(\theta) \frac{\partial \vec{W} }{\partial x} + \vec{\tilde{B}}_{2}(\theta)  \frac{\partial \vec{ W} }{\partial y} = \vec{\tilde{L}}(\theta)\vec{S}_{1}
, \end{equation}
or the quasi-diagonalized form
\begin{equation}
\label{eq:diag}
\frac{\partial \vec{ W } }{\partial t} + \vec{\tilde{D}}_{1}(\theta) \frac{\partial \vec{ W } }{\partial x} + \vec{\tilde{D}}_{2}(\theta)  \frac{\partial \vec{  W} }{\partial y} = \vec{S}(\vec{W}{;\theta})+ \vec{S}^{(a)}
, \end{equation}
where  $\vec{W} %=\left(w_{1},w_{2},w_{3}\right)^{T}
= \vec{\tilde{L}}(\theta)\vec{V} $ is the characteristic variable vector with three components
\begin{equation}
\label{w1w2w3}
w_{1}= -\frac{h}{2}+\frac{\tilde{c}}{2g\tilde{K}_{\theta}}v_{\theta},\quad
w_{2}=\frac{1}{\tilde{K}_{\theta}}\left[\tilde{G}_{s}(\theta)u-\tilde{G}_{c}(\theta)v\right],\quad
w_{3}= \frac{h}{2}+\frac{\tilde{c}}{2g\tilde{K}_{\theta}}v_{\theta},
\end{equation}
and the matrix $\vec{\tilde{D}}_{i}(\theta)$ denotes the diagonal component of $\vec{\tilde{B}}_{i}(\theta)
=\vec{\tilde{L}}(\theta)  \vec{\tilde{A}}_{i}
\vec{\tilde{R}}(\theta)$, $i=1,2$. Moreover, the ``source'' terms
are expressed by
\begin{equation}
\label{swtheta}
\vec{S}(\vec{W}{;\theta})
:=\left(\vec{\tilde{D}}_{1}(\theta)-\vec{\tilde{B}}_{1}(\theta)\right)\frac{\partial \vec{W}}{\partial x}+\left(\vec{\tilde{D}}_{2}(\theta)-\vec{\tilde{B}}_{2}(\theta)\right)\frac{\partial \vec{W}}{\partial y},\\
\end{equation}
and
\[
\vec{S}^{(a)} %=\left(s_{1}^{(a)},s_{2}^{(a)},s_{3}^{(a)}\right)^{T}
=\vec{\tilde{L}}(\theta)\vec{S}_{1}.
\]
%and the matrix $\vec{\tilde{D}_{i}}(\theta)$ denotes the diagonal component of $\vec{\tilde{B}_{i} }(\theta)$.
Because the entry in the $\imath$th row and $\jmath$th column of a matrix $\vec{\tilde{B}}_{i}(\theta)$ is $\vec{\tilde{L}}^{(\imath)}(\theta)\vec{\tilde{A}}_{i}\vec{\tilde{R}}^{(\jmath)}(\theta)$,
where $\vec{\tilde{L}}^{(\imath)}(\theta)$ denotes
the $\imath$th row vector of the matrix  $\vec{\tilde{L}}(\theta)$ and
 $\vec{\tilde{R}}^{(\jmath)}(\theta)$ denotes
the $\jmath$th column vector of the matrix  $\vec{\tilde{R}}(\theta)$,
the diagonal entries of $\vec{\tilde{D}}_{i}(\theta)
={\rm diag} \left\{ d_{i}^{(1)}(\theta) ,d_{i}^{(2)}(\theta), d_{i}^{(3)}(\theta) \right\}$
may be expressed as
\begin{align*}
d_{i}^{(\ell)}(\theta)=
\vec{\tilde{L}}^{(\ell)}(\theta)\vec{\tilde{A}}_{i}\vec{\tilde{R}}^{(\ell )}(\theta),
%& =
%{\rm diag} %\left\{\vec{\tilde{L}^{(1)}}(\theta)\vec{\tilde{A}_{i}}\ %vec{\tilde{R}^{(1)}}(\theta),
%\vec{\tilde{L}^{(2)}}(\theta)\vec{\tilde{A}_{i}}\vec{\wi %detilde{R}^{(2)}}(\theta) ,\vec{\tilde{L}^{(3)}}(\theta)\vec{\tilde{A}_{i}}\vec{\tilde{R}^{(3)}}(\theta) \right\}\\
%&,
\ \ \ell=1,2,3,\ i=1,2.
\end{align*}
Those diagonal entries { determine the bicharacteristics of \eqref{eq:linear}} by
\begin{equation}
\label{eq:bich}
\frac{dx}{dt}=d_{1}^{(\ell)}(\theta),
\ \frac{dy}{dt}=d_{2}^{(\ell)}(\theta),\ \ell=1,2,3.
\end{equation}
\begin{figure}[htbp]
  \centering
  \includegraphics[width=0.4\textwidth]{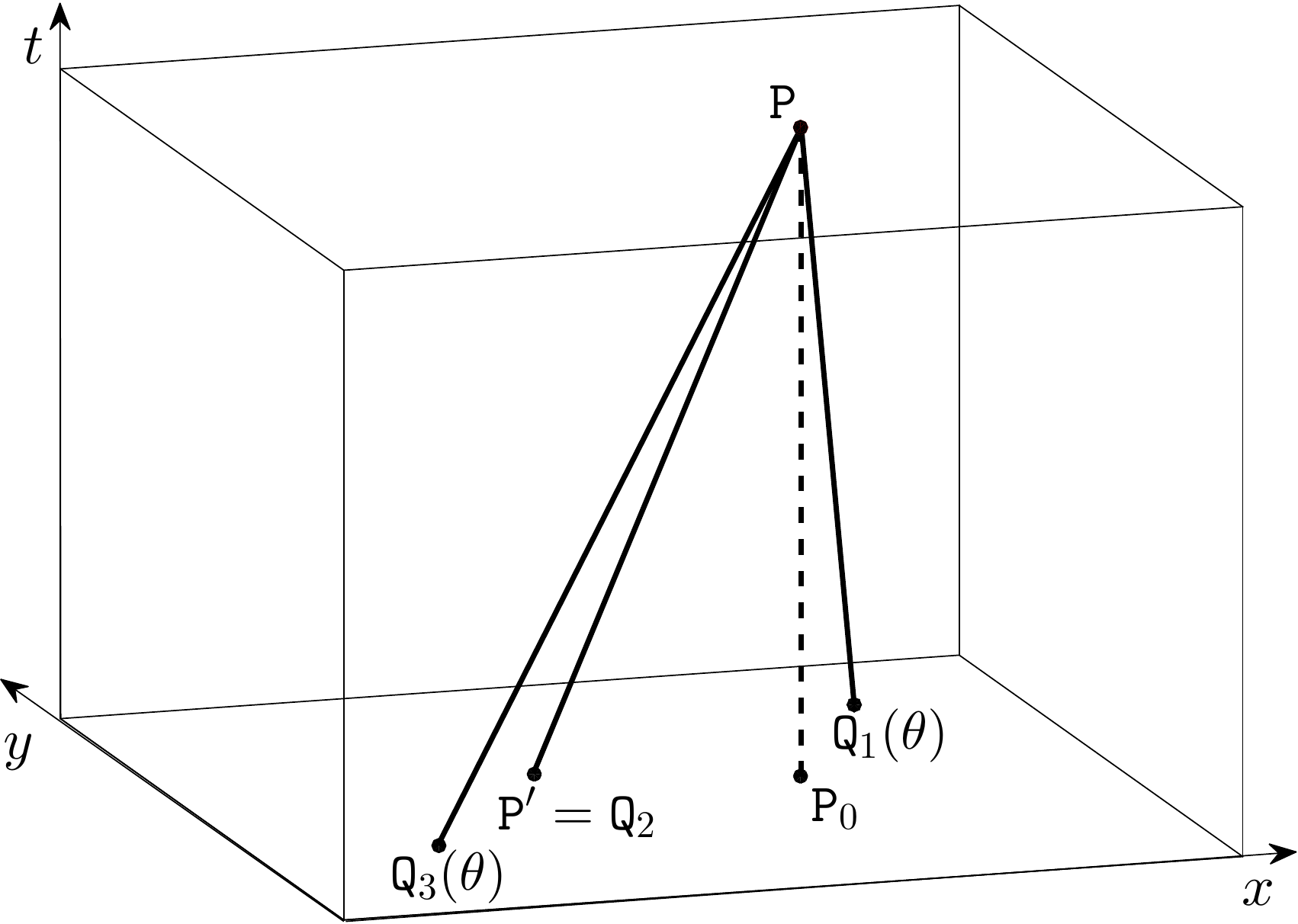}
  \includegraphics[width=0.4\textwidth]{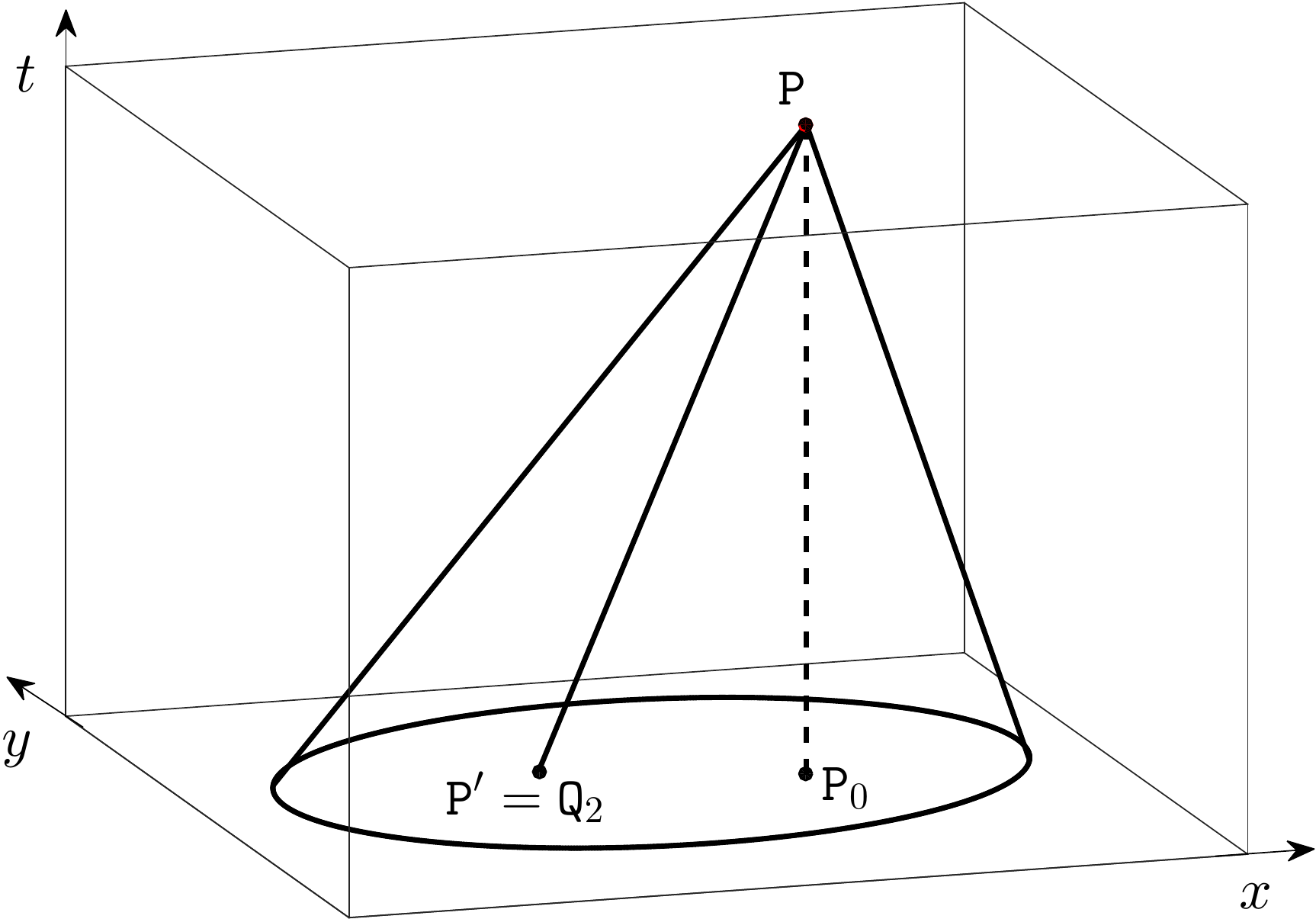}
  \caption{The bicharacteristic curves with the fixed angle $\theta$ (left) and the bicharacteristic cone (right) of the system \eqref{eq:linear}. }
  \label{fig:bicharacteristic}
\end{figure}
It is worth nothing that the second bicharacteristic direction does not depend on the real angle $\theta$. The left figure in Fig. \ref{fig:bicharacteristic} shows three bicharacteristics in the $\left(x,y,t\right)$ space for
a fixed angle $\theta$, i.e.  straight (solid) line segments ${\tt PQ_{\ell}},\ell=1,2,3 $, where ${\tt PQ_{2}}$ is
also denoted by ${\tt PP'}$, and the dotted line ${\tt PP_{0}}$ is only a line segment perpendicular to the horizontal $(x,y)$ plane. If assuming that the coordinate of the point ${\tt P}$ is $(x,y,t+\tau)$, then ones of
${\tt P_{0}}$ and ${\tt Q_{\ell}}(\theta)$ are $(x,y,t)$ and $\left(x-d_{1}^{(\ell)}(\theta)\tau,y-d_{2}^{(\ell)}
(\theta)\tau, t\right)$, respectively, $\ell=1,2,3$,
%see the left figure in Fig. \ref{fig:bicharacteristic}.
The right plot in Fig. \ref{fig:bicharacteristic} displays corresponding bicharacteristic cone past the point ${\tt P}$, which may be described by the set
\begin{equation}
\label{bichline}
\left\{\left(x-d_{1}^{(\ell)}(\theta)\zeta,y-d_{2}^{(\ell)}(\theta)\zeta, t+\tau-\zeta
\right),\ {{\ell}}=1,3,\ \theta\in\left[0,2\pi\right),\ \zeta\in\left[0,\tau\right]\right\},
\end{equation}
which is an elliptic cone in general.
\begin{lemma}
The diagonal entries $ d_{i}^{(\ell)}(\theta)$ and the ``source'' term $\vec{S}(\vec{W}{;\theta})$ in the
quasi-diagonalized system \eqref{eq:diag} {{have}} the following explicit form
\[
d_{1}^{(1)}(\theta) =\tilde{u}-\tilde{c}\tilde{G}_{c}(\theta){,\quad} d_{1}^{(2)}(\theta) =\tilde{u}{,\quad}
d_{1}^{(3)}(\theta) =\tilde{u}+\tilde{c}\tilde{G}_{c}(\theta),
\]
\[
d_{2}^{(1)}(\theta) =\tilde{v}-\tilde{c}\tilde{G}_{s}(\theta){,\quad}d_{2}^{(2)}(\theta) =\tilde{v}{,\quad}
d_{2}^{(3)}(\theta) =\tilde{v}+\tilde{c}\tilde{G}_{s}(\theta),
\]
and
\begin{equation}
\label{sourceadd}
 \vec{S}(\vec{W}{;\theta}) =\left(s_{1},s_{2},s_{3}\right)^{T}
=\vec{M}(\theta)\left(\frac{ \partial \vec{W}}{\partial y}\cos\theta- \frac{\partial \vec{W}}{\partial x}\sin\theta\right),
\end{equation}
where $\vec{M}(\theta) $ is a $3\times 3$ matrix defined by
\begin{equation}
\label{mtheta}
\vec{M}(\theta)=
\begin{pmatrix}
0 & -M_{1}(\theta) & 0\\
-M_{2}(\theta) & 0 & M_{2}(\theta) \\
0 & M_{1}(\theta)  & 0
\end{pmatrix},
%\begin{pmatrix}
%0 & -\frac{1}{2}\tilde{h} & 0\\
%- g\frac{ 1 }{\tilde{K_{\theta}^{2}} \tilde{\Lambda} } & 0 & %g\frac{1}{\tilde{K_{\theta}^{2}} \tilde{\Lambda} } \\
%0 & \frac{1}{2}\tilde{h}  & 0
%\end{pmatrix}.
\end{equation}
here $  M_{1}(\theta)=\frac{\tilde{h}}{2}$ and $  M_{2}(\theta)=\frac{g}{\tilde{K}_{\theta}^{2} \tilde{\Lambda} }$.
\end{lemma}

\begin{proof}
Because
\[
\vec{\tilde{B}}_{i}(\theta) = \vec{\tilde{L}}(\theta)\vec{\tilde{A}}_{i}\vec{\tilde{R}}(\theta),\ i=1,2,
\]
one has
\begin{equation*}
\begin{aligned}
\vec{\tilde{B}}_{1} (\theta)&= \begin{pmatrix}
\tilde{u}- \tilde{c}\tilde{G}_{c}(\theta) & - \frac{\tilde{h}}{2}\sin\theta &0\\
- \frac{1}{\tilde{K}_{\theta}^{2} \tilde{\Lambda} }g\sin\theta & \tilde{u} & \frac{1}{\tilde{K}_{\theta}^{2} \tilde{\Lambda} }g\sin\theta \\
0 & \frac{\tilde{h}}{2}\sin\theta  & \tilde{u}+ \tilde{c}\tilde{G}_{c}(\theta)
\end{pmatrix},\\
\vec{\tilde{B}}_{2}(\theta)&=
\begin{pmatrix}
\tilde{v}- \tilde{c}\tilde{G}_{s}(\theta) & \frac{\tilde{h}}{2}\cos\theta &0\\
\frac{1}{\tilde{K}_{\theta}^{2} \tilde{\Lambda} }g\cos\theta & \tilde{v} & -\frac{1}{\tilde{K}_{\theta}^{2} \tilde{\Lambda} }g\cos\theta \\
0 & -\frac{\tilde{h}}{2}\cos\theta  & \tilde{v}+ \tilde{c}\tilde{G}_{s}(\theta)
\end{pmatrix}.
\end{aligned}
\end{equation*}
Thus the identities
\begin{equation}
\label{Biethetam}
\vec{\tilde{B}}_{1}(\theta)=\vec{\tilde{D}}_{1}(\theta)+\vec{M}(\theta)\sin\theta,\quad
\vec{\tilde{B}}_{2}(\theta) =\vec{\tilde{D}}_{2}(\theta)-\vec{M}(\theta)\cos\theta,
\end{equation}
holds for all $\theta \in \mathbb{R}$.
 {{Using}} \eqref{swtheta} and \eqref{Biethetam} may {{complete the proof}}.
\qed\end{proof}

\begin{remark}
The left-hand side of the quasi-diagonalized system \eqref{eq:diag} does fully decouple the components of the
characteristic variable vector $\vec{W}= (w_{1},w_{2},w_{3})$,
but the right-hand side of \eqref{eq:diag} weakly couples
 three characteristic variables $w_{1},w_{2}$, and $w_{3}$, that is,
 the $\ell$th component of $\vec{S}(\vec{W}{;\theta}) $ in \eqref{eq:diag} does not depend on
the characteristic variable $w_{\ell}$.
\end{remark}

Along the bicharacteristics \eqref{eq:bich}, the system \eqref{eq:diag} reduces to the following system of ordinary differential equations
\begin{equation}
\label{eq:differential}
\frac{\mbox{\rm D}^{(\ell)}w_{\ell}}{\mbox{\rm D}t}=s_{\ell}+s_{\ell}^{(a)},\ell=1,2,3,
\end{equation}
where the differential operator $\frac{\mbox{\rm D}^{(\ell)}}{\mbox{\rm D}t}:=\frac{\partial}{\partial t}+d_{1}^{(\ell)}(\theta)\frac{\partial}{\partial x}+d_{2}^{(\ell)}(\theta)\frac{\partial}{\partial y}
$ denotes the total derivative operator along the $\ell$th bicharacteristic in \eqref{eq:bich}. Along the $\ell$th bicharacteristic in \eqref{eq:bich}, integrating the $\ell$th equation in \eqref{eq:differential} {in terms of} the time from $t$ to $t+\tau$ with $\tau>0$ gives the following equivalent integral system of \eqref{eq:linear}
\begin{align}
\nonumber w_{\ell}\left(x,y,t+\tau;\theta\right)&
=w_{\ell}\left(x-d_{1}^{(\ell)}(\theta)\tau,y-d_{2}^{(\ell)}(\theta)\tau, t;\theta\right)\\
&+s_{\ell}^{\tau}\left(x,y,t;\theta\right)+s_{\ell}^{\tau,(a)}\left(x,y,t;\theta\right),\label{eq:intergralw}
\end{align}
or
\begin{equation}
\label{eq:sinteralw}
\begin{pmatrix}
w_{1}({\tt P};\theta)\\
w_{2}({\tt P};\theta)\\
w_{3}({\tt P};\theta)
\end{pmatrix}
=\begin{pmatrix}
w_{1}\left({\tt Q_{1}}(\theta);\theta\right)
+
s_{1}^{\tau}\left({\tt P_{0}};\theta\right) + s_{1}^{\tau,(a)}\left({\tt P_{0}};\theta\right)\\
w_{2}\left({\tt Q_{2}}(\theta);\theta\right)
+
s_{2}^{\tau}\left({\tt P_{0}};\theta\right) + s_{2}^{\tau,(a)}\left({\tt P_{0}};\theta\right)\\
w_{3}\left({\tt Q_{3}}(\theta);\theta\right)
+
s_{3}^{\tau}\left({\tt P_{0}};\theta\right) + s_{3}^{\tau,(a)}\left({\tt P_{0}};\theta\right)\\
\end{pmatrix},
\end{equation}
%,\quad
where
\begin{equation}
\label{stau}
\begin{aligned}
s_{\ell}^{\tau}\left(x,y,t;\theta\right)&=\int_{t}^{t+\tau}s_{\ell}\left(x-d_{1}^{(\ell)}(\theta)\left(t+\tau-\zeta\right),y-d_{2}^{(\ell)}(\theta)\left(t+\tau-\zeta \right), \zeta;\theta\right)d\zeta,\\
s_{\ell}^{\tau,(a)}\left(x,y,t;\theta\right)&=\int_{t}^{t+\tau}s_{\ell}^{(a)}\left(x-d_{1}^{(\ell)}(\theta)\left(t+\tau-\zeta\right),y-d_{2}^{(\ell)}(\theta)\left(t+\tau-\zeta\right), \zeta;\theta\right)d\zeta,
\end{aligned}
\end{equation}
and ${\tt P},{\tt P_{0}}$, and ${\tt Q_{\ell}}(\theta)$ denote the points $\left(x,y, t+\tau\right),\left(x,y,t\right)$, and $\left(x-d_{1}^{(\ell)}(\theta)\tau,y-d_{2}^{(\ell)}(\theta)\tau, t \right)$, respectively.
The integral equation \eqref{eq:intergralw} or \eqref{eq:sinteralw} gives the time evolution of the variable $w_{\ell}$ in the quasi-diagonalized system \eqref{eq:diag} along its bicharacteristics \eqref{eq:charact},  $\ell=1,2,3$.

Multiplying \eqref{eq:sinteralw} by $\vec{\tilde{R}}(\theta)$ from the left and integrating it with respect to $\theta$ from $0$ to $2\pi$ (i.e. superposition of all the waves together) yield the exact evolution operator  $\mathcal{E}(\tau) $  of \eqref{eq:linear}    as follows
\begin{equation}
\label{eq:exactinteral}
\mathcal{E}(\tau)      \vec{V}({\tt P_{0}}):=\vec{V}({\tt P})=\frac{1}{2\pi}\int_{0}^{2\pi}\sum_{\ell=1}^{3}\vec{\tilde{R}}^{(\ell)}(\theta)\left(w_{\ell}\left({\tt Q_{\ell}}(\theta);\theta\right)+s_{\ell}^{\tau}\left({\tt P_{0}};\theta\right) + s_{\ell}^{\tau,(a)}({\tt P_{0}};\theta)\right).
\end{equation}

\begin{theorem}\label{theor:interal}
The exact integral equations \eqref{eq:exactinteral} are equivalent to
\begin{align}
h\left({\tt P}\right)=&\frac{1}{2\pi}\int_{0}^{2\pi}Jd\theta
          -\frac{\tilde{h}}{2\pi}\int_{0}^{2\pi}\int_{t}^{t+\tau} \mathcal{S}({\tt Q_{\zeta}}(\theta);\theta) d\zeta d\theta,\label{eq:interalh}\\
\nonumber
u\left({\tt P}\right)=&-\frac{g}{2\pi\tilde{c}} \int_{0}^{2\pi} J\tilde{G}_{c}(\theta)d\theta
+\left[u({\tt P'})J_{6}-v({\tt P'})J_{4}\right]\\
	  \nonumber +&\frac{g}{\tilde{\Lambda}  } \int_{t}^{t+\tau}\left(h_{y}({\tt P}_{\zeta}')J_{1}-h_{x}({\tt P}_{\zeta}')J_{2}\right)d\zeta
         +\frac{\tilde{c}}{2\pi} \int_{0}^{2\pi}\int_{t}^{t+\tau} \tilde{G}_{c}(\theta)\mathcal{S}({\tt Q_{\zeta}}(\theta);\theta)d\zeta d\theta\\
         +&\frac{1}{2\pi} \int_{0}^{2\pi}\int_{t}^{t+\tau} S_{1}^{(1)}({\tt Q_{\zeta}}(\theta);\theta)d\zeta d\theta,\label{eq:interalu}\\
\nonumber
v\left({\tt P}\right)=&-\frac{g}{2\pi\tilde{c}} \int_{0}^{2\pi} J\tilde{G}_{s}(\theta) d\theta
-\left[u({\tt P'})J_{7}-v({\tt P'})J_{5}\right]\\
	   \nonumber+&\frac{g}{\tilde{\Lambda} } \int_{t}^{t+\tau}\left(h_{x}({\tt P}_{\zeta}')J_{1}-h_{y}({\tt P}_{\zeta}')J_{3}\right) d\zeta
         +\frac{\tilde{c}}{2\pi} \int_{0}^{2\pi}\int_{t}^{t+\tau}\tilde{G}_{s}(\theta)\mathcal{S}({\tt Q_{\zeta}}(\theta);\theta) d\zeta d\theta\\
          +&\frac{1}{2\pi} \int_{0}^{2\pi}\int_{t}^{t+\tau} S_{1}^{(2)}({\tt Q_{\zeta}}(\theta);\theta)d\zeta d\theta,\label{eq:interalv}
\end{align}
where  $ \mathcal{S}\left(x,y,t;\theta\right)= \frac{1}{\tilde{K}_{\theta}}\left(\Phi_{x}\sin\theta-\Phi_{y}\cos\theta\right)$,
$\Phi_{x}$ and $\Phi_{y}$ are the partial derivatives of
$\Phi(x,y,t;\theta):= \tilde{G}_{s}(\theta)u-\tilde{G}_{c}(\theta)v$
with respect to $x$ and $y$ respectively,
%\begin{equation}
%\label{sinteral}
%\mathcal{S}\left(x,y,t;\theta\right):= \frac{1}{\tilde{K_{\theta}}}\left(\Phi_{x}\sin\theta-\Phi_{y}\cos\theta\right),\quad
%\Phi\left(x,y,t;\theta\right):= \tilde{G_{s}}(\theta)u-\tilde{G_{c}}(\theta)v,
%\end{equation}
and the shortened notations ${\tt P'}:={\tt Q_{2}}$ and ${\tt Q}(\theta):={\tt Q_{1}}(\theta)$.
Moreover, ${\tt Q_{\zeta}}(\theta)$ and ${\tt P}_{\zeta}'$ denote the points
\[
\left(x-d_{1}^{(1)}(\theta)\left(t+\tau-\zeta\right),y-d_{2}^{(1)}(\theta)\left(t+\tau-\zeta\right), \zeta\right)
,\]
and
\[
\left(x-\tilde{u}\left(t+\tau-\zeta\right),y-\tilde{v}\left(t+\tau-\zeta\right),\zeta\right)
,\]
respectively.
Here, J and $J_{i},i=1,2,\cdots,7$, are defined by
\begin{align}
J&:=h({\tt Q(\theta)})-\frac{\tilde{c}}{g\tilde{K}_{\theta}}\left(u({\tt Q(\theta)})\cos\theta+v({\tt Q(\theta)})\sin\theta\right)\\
J_{1}&:=\frac{1}{2\pi} \int_{0}^{2\pi} \frac{\sin\theta\cos\theta}{\tilde{K}_{\theta}^2}d\theta =
\begin{cases}
0,                        &\vec{\tilde{x}}=\vec{0},\\
\frac{\tilde{g}^{12}}{\tilde{H}}\left(2-\left(\tilde{g}^{11}+\tilde{g}^{22}\right)\tilde{\Lambda}\right),                         & \vec{\tilde{x}}\neq \vec{0},
\end{cases}\label{j1}\\
J_{2}&:=\frac{1}{2\pi} \int_{0}^{2\pi} \frac{\sin^{2}\theta}{\tilde{K}_{\theta}^2}d\theta=
\begin{cases}
\frac{1}{2},    &\vec{\tilde{x}}=\vec{0},\\
\frac{1}{\tilde{H}}\left(\left(\left(\tilde{g}^{11}\right)^{2}-\tilde{g}^{11}\tilde{g}^{22}
+2\left(\tilde{g}^{12}\right)^{2}\right)\tilde{\Lambda}+\tilde{g}^{22}-\tilde{g}^{11}\right),
&\vec{\tilde{x}}\neq \vec{0},
\end{cases}\label{j2}\\
J_{3}&:=\frac{1}{2\pi} \int_{0}^{2\pi} \frac{\cos^{2}\theta}{\tilde{K}_{\theta}^2}d\theta=
{
\begin{cases}
\frac{1}{2},    &\vec{\tilde{x}}=\vec{0},\\
\frac{1}{\tilde{H}}\left(\left(\left(\tilde{g}^{22}\right)^{2}-\tilde{g}^{11}\tilde{g}^{22}
+2\left(\tilde{g}^{12}\right)^{2}\right)\tilde{\Lambda}+\tilde{g}^{11}-\tilde{g}^{22}\right),
&\vec{\tilde{x}}\neq \vec{0},
\end{cases}
}
\label{j3}\\
J_{4}&:=\frac{1}{2\pi} \int_{0}^{2\pi} \frac{\tilde{G}_{c}(\theta)\sin\theta}{\tilde{K}_{\theta}}d\theta=\tilde{g}^{11}J_{1}+\tilde{g}^{12}J_{2}\label{j4},\\
J_{5}&:=\frac{1}{2\pi} \int_{0}^{2\pi} \frac{\tilde{G}_{c}(\theta)\cos\theta}{\tilde{K}_{\theta}}d\theta=\tilde{g}^{11}J_{3}+\tilde{g}^{12}J_{1}\label{j5},\\
J_{6}&:=\frac{1}{2\pi} \int_{0}^{2\pi} \frac{\tilde{G}_{s}(\theta)\sin\theta}{\tilde{K}_{\theta}}d\theta=\tilde{g}^{12}J_{1}+\tilde{g}^{22}J_{2}\label{j6},\\
J_{7}&:=\frac{1}{2\pi} \int_{0}^{2\pi} \frac{{ \tilde{G}_{s}}(\theta)\cos\theta}{\tilde{K_{\theta}}}d\theta=\tilde{g}^{12}J_{3}+\tilde{g}^{22}J_{1}\label{j7},
\end{align}
with
\[
\tilde{H}=\left(\tilde{g}^{11}\right)^{2}+\left(\tilde{g}^{22}\right)^{2}-2\tilde{g}^{11}\tilde{g}^{22}
+4\left(\tilde{g}^{12}\right)^2.
\]
\end{theorem}

\begin{proof}
Because the variables $ w_{\ell}({\tt Q_{\ell}}(\theta);\theta)$ and $\vec{\tilde{R}}^{{ (\ell)}}(\theta) $ are $2\pi$-periodic with respect to $\theta$, $\ell=1,3$, and
\[
 \vec{\tilde{R}}^{(1)}(\theta+\pi) = - \vec{\tilde{R}}^{(3)}(\theta),\  d_{i}^{(1)}(\theta+\pi)=d_{i}^{(3)}(\theta),\  i=1,2
,\
\vec{M}(\theta+\pi)=\vec{M}(\theta),
\]
which imply
\[
 {\tt Q_{1}}(\theta+\pi)
 ={\tt Q_{3}}(\theta),
 \ w_{1}({\tt Q_{1}}(\theta+\pi);\theta+\pi)
  = -w_{3}({\tt Q_{3}}(\theta); \theta)
,\
w_{2}({\tt P_{0}};\theta+\pi)=-w_{2}({\tt P_{0}};\theta),
\]
 one has
\begin{equation}
\label{rwinteral}
\int_{0}^{2\pi}\vec{\tilde{R}}^{(1)}(\theta)w_{1}({\tt Q_{1}}(\theta);\theta)d\theta=\int_{0}^{2\pi}\vec{\tilde{R}}^{(3)}(\theta)w_{3}({\tt Q_{3}}(\theta);\theta)d\theta,
\end{equation}
and
\begin{align*}
s_{1}(\vec{x},t;\theta+\pi) =& M_{1}(\theta+\pi)\left( \frac{\partial w_{2}(\vec{x},t;\theta+\pi)}{\partial x}\sin(\theta+\pi)- \frac{\partial w_{2}(\vec{x},t;\theta+\pi)}{\partial y}\cos(\theta+\pi)\right)\\
                               =&M_{1}(\theta)\left( \frac{\partial w_{2}(\vec{x},t;\theta)}{\partial x}\sin\theta- \frac{\partial w_{2}(\vec{x},t;\theta)}{\partial y}\cos\theta \right)
                               = -s_{3}(\vec{x},t;\theta).
\end{align*}
Substituting the last equation into \eqref{stau} gives
\begin{align*}
s_{1}^{\tau}(\vec{x},t;\theta+\pi)=& \int_{t}^{t+\tau}s_{1}\left(x-d_{1}^{(1)}(\theta+\pi)(t+\tau-\zeta),y-d_{2}^{(1)}
(\theta+\pi)(t+\tau-\zeta),\zeta;         \theta+\pi\right) d\zeta\\
                          =& \int_{t}^{t+\tau}s_{1}\left(x-d_{1}^{(3)}(\theta)
                          (t+\tau-\zeta),y-d_{2}^{(3)}(\theta)
                          (t+\tau-\zeta),\zeta;   \theta+\pi\right)        d\zeta\\
                          =&- \int_{t}^{t+\tau}s_{3}\left(x-d_{1}^{(3)}(\theta)
                          (t+\tau-\zeta),y-d_{2}^{(3)}(\theta)
                          (t+\tau-\zeta),\zeta;    \theta+\pi\right) d\zeta\\
                          =& -s_{3}^{\tau}(\vec{x},t;\theta),
\end{align*}
and thus the identity
\begin{equation}
\label{rsinteral}
\int_{0}^{2\pi}\vec{\tilde{R}}^{(1)}(\theta)
s_{1}^{\tau}(\vec x,t;\theta)d\theta
=\int_{0}^{2\pi}\vec{\tilde{R}}^{(3)}(\theta)s_{3}^{\tau}(\vec x,t;\theta)d\theta,
\end{equation}
holds.
With the definition of $\vec{\tilde{R}}^{(\ell)}(\theta) $ and the identities \eqref{rwinteral} and \eqref{rsinteral}, the exact integral equations \eqref{eq:exactinteral} can be rewritten as follows
\begin{align}
\nonumber \vec{V}({\tt P})=&\frac{1}{2\pi}\int_{0}^{2\pi}
\begin{pmatrix}
-2\left(w_{1}({\tt Q}(\theta);\theta)+s_{1}^{\tau}(\vec x,t;\theta)\right)\\
2\frac{g}{\tilde{c}}\tilde{G}_{c}(\theta)\left(w_{1}({\tt Q}(\theta);\theta)+s_{1}^{\tau}(\vec x,t;\theta)\right)+\left(w_{2}({\tt P'};\theta)+s_{2}^{\tau}(\vec x,t;\theta)\right)\sin\theta\\
2\frac{g}{\tilde{c}}\tilde{G}_{s}(\theta)\left(w_{1}({\tt Q}(\theta);\theta)+s_{1}^{\tau}(\vec x,t;\theta)\right)-\left(w_{2}({\tt P}';\theta)+s_{2}^{\tau}(\vec x,t;\theta)\right)\cos\theta\\
\end{pmatrix}
d\theta\\
+&\frac{1}{2\pi} \int_{0}^{2\pi}\int_{t}^{t+\tau} \vec{S}_{1}({\tt Q}_{\zeta}(\theta);\theta)d\zeta d\theta.\label{eq:exactinteral rewritten}
\end{align}
Noting the relations in \eqref{w1w2w3} { gives}
\begin{align}
\nonumber
\frac{1}{2\pi}\int_{0}^{2\pi}w_{2}({\tt P'};\theta)\sin\theta d\theta &=  \frac{1}{2\pi}\int_{0}^{2\pi} \frac{1}{\tilde{K}_{\theta}}\left[ \tilde{G}_{s}(\theta)u({\tt P'}) - \tilde{G}_{c}(\theta)v({\tt P'})\right]\sin\theta d\theta\\
&=u({\tt P'})J_{6}-v({\tt P'})J_{4}\label{w1integral},\\
\nonumber
\frac{1}{2\pi}\int_{0}^{2\pi}w_{2}({\tt P'};\theta)\cos\theta d\theta &= \frac{1}{2\pi}\int_{0}^{2\pi} \frac{1}{\tilde{K}_{\theta}}\left[ \tilde{G}_{s}(\theta)u({\tt P'}) - \tilde{G}_{c}(\theta)v({\tt P'})\right]{ \cos\theta}d\theta\\
&= u({\tt P'})J_{7}-v({\tt P'})J_{5}\label{w2integral}.
\end{align}
On the other hand, one has
\begin{align}
\nonumber
s_{1}(\vec x,t;\theta)=&-M_{1}(\theta)\left( \frac{\partial w_{2}(\vec x,t;\theta)}{\partial y}\cos\theta - \frac{\partial w_{2}(\vec x,t;\theta)}{\partial x}\sin\theta \right)\\
\nonumber    =&-\frac{1}{2}\tilde{h}\left( \frac{\partial w_{2}(\vec x,t;\theta)}{\partial y}\cos\theta - \frac{\partial w_{2}(\vec x,t;\theta)}{\partial x}\sin\theta \right)\\ =&\frac{1}{2}\tilde{h}\frac{\sin\theta}{\tilde{K}_{\theta}}\left[\tilde{G}_{s}(\theta)u_{x}-\tilde{G}_{c}(\theta)v_{x}\right]-\frac{1}{2}\tilde{h}\frac{\cos\theta}{\tilde{K_{\theta}}}\left[\tilde{G_{s}}(\theta)u_{y}-\tilde{G_{c}}(\theta)v_{y}\right]
\label{s1intergral},\\
\nonumber
s_{2}(\vec x,t;\theta)=&M_{2}(\theta)\left( \frac{\partial (w_{3}-w_{1})(\vec x,t;\theta)}{\partial y}\cos\theta - \frac{\partial (w_{3}-w_{1})(\vec x,t;\theta)}{\partial x}\sin\theta \right)\\
\nonumber		     =&\frac{g }{\tilde{K}_{\theta}^{2}\tilde{\Lambda}}\left(-(w_{3}-w_{1})_{x}\sin\theta+(w_{3}-w_{1})_{y}\cos\theta\right)\\
		     =&\frac{g }{\tilde{K}_{\theta}^{2}\tilde{\Lambda} }\left(-h_{x}\sin\theta+h_{y}\cos\theta\right)
\label{s2intergral}.
\end{align}
Substituting \eqref{w1integral}-\eqref{s2intergral} into \eqref{eq:exactinteral rewritten} may give
the integral equations in \eqref{eq:interalh}-\eqref{eq:interalv}. The proof is completed.
\qed\end{proof}

The integral equations  \eqref{eq:interalh}-\eqref{eq:interalv} are the base of  our RKDLEG methods
presented in Section \ref{sec:method}.
In order to derive the RKDLEG method, the exact evolution operator $\mathcal{E}(\tau)$ or the integrals in \eqref{eq:interalh}-\eqref{eq:interalv} has to be further numerically approximated, see
Section \ref{subsubsec:nonlocal}.

% % % % % % % % % % % % % % % % % % % % % % % % % % % % % % % % % % % % % % % % % % % % % % % % % % % % %
\section{Numerical method}
\label{sec:method}
This section is devoted to present the RKDLEG method for the {{SWEs}}
 \eqref{eq:cube} on the cubed-sphere.
 Let $x$ and $y$ be the Cartesian coordinates in a face of the cube and
 restrict our attention   to the following square mesh  in the $(x,y)$ plane:
\[
  x_{j}= -\frac{\pi}{4} + \frac{j\pi}{2N},\ y_{k}=-\frac{\pi}{4}+  \frac{k\pi}{2N},\ \ j,k=0,1,\cdots, N,
  %\ \Delta x= \Delta y= \frac{\pi}{2N}
\]
where $N$ is the grid number in $x$ or $y$ direction. Moreover, the time interval $\left[0,T\right]$ is
{ assumed} to be partitioned into $\left\{t_{n}|t_{0}=0,t_{n+1}=t_{n}+\Delta t_{n},n\geq 0\right\}$, where $\Delta t_{n}$ is the time step size determined by
\begin{equation}
\label{delta t}
\Delta t_{n}=\frac{    {\pi} C_{cfl} }{2N \max \limits_{j,k,1\leq\ell\leq3}\left\{|\lambda^{(\ell)}
(\vec{\bar{V}}^n_{j+\frac{1}{2},k+\frac{1}{2}},\vec{x}_{j+\frac{1}{2},k+\frac{1}{2}}, 0)|+|\lambda^{(\ell)}
(\vec{\bar{V}}^n_{j+\frac{1}{2},k+\frac{1}{2}},\vec{x}_{j+\frac{1}{2},k+\frac{1}{2}}, \frac{\pi}{2})|\right\}},
\end{equation}
where $\lambda^{(\ell)}\left(\vec{V},\vec{x}{;\theta}\right)$ is given in \eqref{lambda}, $\ell=1,2,3$,
 $C_{cfl}$ denotes the CFL number, and
\[
\vec{\bar{V}}_{j+\frac{1}{2},k+\frac{1}{2}}^n \approx  { \frac{4N^2}{\pi^2}}
 \iint_{C_{j+\frac{1}{2},k+\frac{1}{2} }}\vec{V}(\vec x,t_{n} )d\vec x
.\]
here the cell $ C_{j+\frac{1}{2},k+\frac{1}{2}} $ defined by
\[
 C_{j+\frac{1}{2},k+\frac{1}{2}} =\left\{\left(x,y\right)|x_{j}\leq x\leq x_{j+1},y_{k}\leq y\leq y_{k+1},0\leq j,k\leq N-1\right\}.
\]

%\subsection{RKDG method}
%\label{subsec:RKDG}
%
\subsection{DG spatial discretization}
\label{subsubsec:DG}
This section gives the DG spatial discretizations of the SWE{{s}} \eqref{eq:cube}. % \cite{Cockburn: 1998}.
The purpose is to seek an approximation $\vec{U}_{h}$ to $\vec{U}$ such that for each time $t\in \left(0,T\right]$,
each component of $\vec{U}_{h}$ belongs to the finite dimensional space
\[
V_{h}:=\left\{v(\vec{x})\in L^{2}(\tilde{\Omega})  :
v(\vec{x})\big|_{C_{j+\frac{1}{2},k+\frac{1}{2}}}\in { \mathbb{P}^{K}}\left( C_{j+\frac{1}{2},k+\frac{1}{2}}  \right)\right\}
,\]
where $\tilde{\Omega}=[ -\frac{\pi}{4}, \frac{\pi}{4}  ] \times [ -\frac{\pi}{4}, \frac{\pi}{4}  ]$,
${ \mathbb{P}^{K}}\left( C_{j+\frac{1}{2},k+\frac{1}{2}} \right)$ is the space of polynomials in the cell $ C_{j+\frac{1}{2},k+\frac{1}{2}} $ of degree at most $K$, with the dimension at most $\left(K+1\right)\left(K+2\right)/2$.

Multiplying \eqref{eq:cube} with a test function $v(\vec{x})\in { \mathbb{P}^{K}}\left( C_{j+\frac{1}{2},k+\frac{1}{2}} \right)$, integrating by parts over the cell $ C_{j+\frac{1}{2},k+\frac{1}{2}} $, and {{replacing}} the exact solution $\vec{U}$ with the approximate solution $\vec{U}_{h}$ give
\begin{align}
\nonumber
 \frac{d}{dt}  & \iint_{ C_{j+\frac{1}{2},k+\frac{1}{2} }}\vec{U}_{h}(\vec{x},t)  v(\vec{x})
 d\vec{x}
 =
- \int _{\partial  C_{j+\frac{1}{2},k+\frac{1}{2} }}
 \vec{F}_{n}\left(\mathcal{E}_{h,0} \vec{V}_h (\vec{{x}},t)\right)   v(\vec{x})
 ds\\
 &+
 \iint_{C_{j+\frac{1}{2},k+\frac{1}{2} }}
\Big( \vec{S}_{0}\left(\vec{U}_{h}   (\vec{x},t\right))     v(\vec{x})+
 \vec{F}\left( \vec{U}_{h}  (\vec{x},t)   \right) \cdot\nabla v(\vec{x})\Big)   d\vec{x},
 \label{eq:DG}
\end{align}
where $\vec{F}=\left(\vec{F}_{1},\vec{F}_{2}\right)$,
$\vec{F}_{n}(\vec U)=\vec{F}(\vec U)\cdot \vec n$,
$\nabla v (\vec{x})=(\partial x,\partial y)v (\vec{x})$,
$ \vec{n}=\left(n_{1},n_{2}\right)$ is the outward unit normal vector of the cell boundary $\partial C_{j+\frac{1}{2},k+\frac{1}{2}}$,
$\mathcal{E}_{h,0}$ is the approximate local evolution operator and will be discussed in Section \ref{subsec:Approximate},
and $ \vec{V}_h\left(\vec{x},t\right)$ is  the primitive variable vector corresponding to
$ \vec{U}_{h}\left( \vec{x} ,t\right)$. %, so we can get $ \vec{V}\left(\vec{x},t\right)$ by \eqref{Uhphi}.
It is worth noting that  in the traditional RKDG method for hyperbolic conservation laws,
see e.g. \cite{Cockburn: 1998},
  the first term at the right-hand side of  \eqref{eq:DG} is replaced with
  $$- \int _{\partial  C_{j+\frac{1}{2},k+\frac{1}{2} }}
  \vec{\widehat{F}}_{n} \left( \vec{U}_{h}\left(\vec{x}-0,t\right),  \vec{U}_{h}\left(\vec{x}+0,t\right)    \right)
  v(\vec{x})
  ds,
  $$
  where   $\vec{\widehat{ F}}_{n}( \cdot, \cdot)$ is the two-point numerical flux vector satisfying
  the consistency condition  $\vec{\widehat{ F}}_{n} ( \vec{U},  \vec{U})=\vec{ F}_{n}(\vec U)$.

If using $\left\{ \phi_{j,k}^{(\ell)}(\vec{x}),\ell=0,1,\cdots,K\left(K+3\right)/2\right\} $ to denote a basis of the space ${ \mathbb{P}^{K}}\left( C_{j+\frac{1}{2},k+\frac{1}{2}} \right)$, then the DG approximate solution $\vec{U}_{h}$ may be
expressed by
\begin{equation}
\label{Uhphi}
\vec{U}_{h}(\vec{x},t)= \sum_{\ell=0}^{K(K+3)/2}\vec{U}_{j,k}^{(\ell)}(t) \phi_{j,k}^{(\ell)}(\vec x), %=:\vec{U}_{j,k}\left(\vec{x},t\right),
\ \mbox{if}\ \vec{x}\in C_{j+\frac{1}{2},k+\frac{1}{2}},
\end{equation}
and the scaled Legendre polynomials are taken
as the basis in this paper, see Section 3.3.2 in \cite{Zhao:2013}.

The  first and second terms at the right-hand side of \eqref{eq:DG}  are further respectively discretized
by using Gaussian quadratures of  high order accuracy as follows
\begin{align}
%\iint_{ C_{j+\frac{1}{2},k+\frac{1}{2} } } \vec{F}\left( \vec{U}_{h}\left(\vec{x},t\right) \right) \cdot \nabla v(\vec{x})d\vec{x}\approx &
& -| \partial  C_{j+\frac{1}{2},k+\frac{1}{2} }|\sum_{m=1}^{K+2}\tilde{\omega}_{m}
 %\widehat{ \vec{F}}_n \left( \vec{U}_{h}\left( \vec{\tilde{x}}_{m}^{G}                -0,t\right),  \vec{U}_{h}\left(\vec{\tilde{x}}_{m}^{G}  +0,t\right)    \right)
 % v\left(  \vec{\tilde{x}}_{m}^{G}\right)
 %
 \vec{F}_n\left(\mathcal{E}_{h,0} \vec{V}_h (\vec{\tilde{x}}_{m}^{G},     t)\right)
 v(  \vec{\tilde{x}}_{m}^{G})
 ds,
 \label{fluxboundary}
 \\
 & | C_{j+\frac{1}{2},k+\frac{1}{2}}|
 \sum_{m=1}^{(K+2)^2}\omega_{m}
 \Big(
 \vec{F}\left(\vec{U}_{h} (\vec{x}_{m}^{G},t)\right)\cdot \nabla v (\vec{x}_{m}^{G})
 +\vec{S}_{0}\left(\vec{U}_{h} (\vec{x}_{m}^{G},t) \right)    v(\vec{x}_{m}^{G})
 \Big),
 \label{fluxinner}
%\label{sourceinner}
\end{align}
where
$\{ \tilde{\omega}_{m} , \vec{\tilde{x}}_{m}^{G}\}$, $m=1,\cdots, K+2$,
and $\{{\omega}_{m}, \vec{x}_{m}^{G}\}$, $m=1,\cdots,  (K+2)^2$,
denote the Gauss-Lobatto quadrature weights and nodes
in $ \partial  C_{j+\frac{1}{2},k+\frac{1}{2} } $ and
$ C_{j+\frac{1}{2},k+\frac{1}{2}} $,  respectively.
%while $\tilde{\omega}_{m}$ and $\omega_{m} $ are  corresponding Gaussian quadrature weights.
%The flux $\vec{F}\left(\vec{U}_{h}\left(\vec{\tilde{x}}_{m}^{G},t\right)\right) \cdot\vec{n}$ in \eqref{fluxboundary} is further replaced with a monotone numerical flux.

In conclusion, our semi-discrete ${ \mathbb{P}^{K}}$-based DG methods for \eqref{eq:cube} may be given as
\begin{align}
\sum_{\ell=0}^{K\left(K+3\right)/2}&
 \iint_{ C_{j+\frac{1}{2},k+\frac{1}{2} }}\phi_{j,k}^{(\ell)}v(\vec{x})d\vec{x}  \frac{d\vec{U}_{j,k}^{(\ell)}\left(t\right)}{dt}
\nonumber
=  -| \partial  C_{j+\frac{1}{2},k+\frac{1}{2} }|\sum_{m=1}^{K+2}\tilde{\omega}_{m} %\vec{\hat{F}}\left(\vec{\tilde{x}}_{m}^{G},t\right) \cdot\vec{n}v\left(\vec{\tilde{x}}_{m}^{G}\right)
 \vec{F}_n\left(\mathcal{E}_{h,0} \vec{V}_h (\vec{\tilde{x}}_{m}^{G},     t)\right)
 v(  \vec{\tilde{x}}_{m}^{G})
\\
&+| C_{j+\frac{1}{2},k+\frac{1}{2} }|\sum_{m=1}^{(K+2)^2}\omega_{m}
\Big( \vec{F}\left(\vec{U}_{h} (\vec{x}_{m}^{G},t)\right)\cdot \nabla v (\vec{x}_{m}^{G})+
\vec{S}_{0}\left(\vec{U}_{h}
(\vec{x}_{m}^{G},t) \right)
v(\vec{x}_{m}^{G})   \Big),
%\
\label{semiDG}
\end{align}
for $v(\vec{x})=\phi_{j,k}^{(\ell')}(\vec{x}),\ell'=0,1,\cdots,K\left(K+3\right)/2$.
It  forms a nonlinear system of
ordinary differential equations evolving the degrees of freedom or moments $\vec{U}_{j,k}^{(\ell)}\left(t\right),\ell=0,1,\cdots,K\left(K+3\right)/2$.
%\begin{remark}
%For the above $2D$ ${ \mathbb{P}^{K}}$-based DG methods, the boundary integration in \eqref{fluxboundary} uses $K+2$ Gauss-Lobatto quadrature points, i.e. $\tilde{q}=K+2$, while the element integration in \eqref{fluxinner},\eqref{sourceinner} uses $(K+2)^{2}$ Gauss-Lobatto quadrature points, i.e. $q=\left(K+2\right)^{2}$.
%\end{remark}

%\begin{remark}
%The scaled Legendre polynomials $\left\{ \phi_{j,k}^{(\ell)}(\vec{x}),\ell=0,1,\cdots,K\left(K+3\right)/2\right\} $
 %is a basis of the polynomial space ${ \mathbb{P}^{K}}\left( C_{j+\frac{1}{2},k+\frac{1}{2}} \right)$,
% the details of the scaled Legendre polynomials see .
%\end{remark}

\subsection{Time discretization}
\label{subsubsec:RK}
The semi-discrete schemes \eqref{semiDG} may be  rewritten into {an} abstract  form
\begin{equation}\label{RK-01}
\frac{d\vec{U}}{dt}=\vec{L}\left(\vec{U},t\right),
\end{equation}
which is a nonlinear system of ordinary differential equations {{with respect to}} $\vec{U}$.
Following the traditional RKDG methods, the system \eqref{RK-01} may be approximated by some
strong stability-preserving high-order time discretization.  % explicit Runge-Kutta method
 For example, the
explicit third order Runge-Kutta discretization \cite{Shu: 1988} for \eqref{RK-01} may be given by
\begin{align*}
\vec{U}^{(1)}=&\vec{U}^{n}+\Delta t_{n}\vec{L}\left(\vec{U}^{n},t_{n}\right),\\
\vec{U}^{(2)}=&\frac{3}{4}\vec{U}^{n}+\frac{1}{4}\left(\vec{U}^{(1)}+\Delta t_{n}\vec{L}\left(\vec{U}^{(1)},t_{n}+\Delta t_{n}\right)
\right),\\
\vec{U}^{n+1}=&\frac{1}{3}\vec{U}^{n}+\frac{2}{3}\left(  \vec{U}^{(2)}+\Delta t_{n}\vec{L}\left(\vec{U}^{(2)},t_{n}+\frac{1}{2}\Delta t_{n}\right) \right).
\end{align*}

In our practical computations, in order to match  the accuracy of DG spatial discretization,
the $(K+1)${th} order strong stability-preserving  explicit Runge-Kutta method
 is used for the ${ \mathbb{P}^{K}}$-based { RKDLEG} methods,
$K=1,2$, but
a general explicit { fourth}-order explicit Runge-Kutta method is employed
for the  ${ \mathbb{P}^{3}}$-based {RKDLEG} methods.

\subsection{Approximate   evolution operators}
\label{subsec:Approximate}
This section will derives the approximate local evolution operator
$\mathcal{E}_{h,0}$ used in our RKDLEG methods, see \eqref{semiDG}.
The operator $\mathcal{E}_{h,0} $ is the limit
of the approximate evolution operator $\mathcal{E}_{h}(\tau) $ as $\tau$ approaches to zero, i.e. $\mathcal{E}_{h,0}=\lim\limits_{\tau\to 0}\mathcal{E}_{h}(\tau)$, where $\mathcal{E}_{h}(\tau) $ is
an appropriate approximation of the exact evolution operator $\mathcal{E}(\tau) $ defined in
\eqref{eq:exactinteral} by numerically approximating the ``source'' terms in   \eqref{eq:interalh}-\eqref{eq:interalv}, specifically, the integrands of the integral terms depending on $ \mathcal{S}({\tt Q_{\zeta}}(\theta);\theta) $,  $ S_{1}^{(1)}({\tt Q_{\zeta}}(\theta);\theta)$, and $S_{1}^{(2)}({\tt Q_{\zeta}}(\theta);\theta)$.

\subsubsection{Approximate evolution operator $\mathcal{E}_{h}(\tau) $}
\label{subsubsec:nonlocal}
%The exact integral equations in \eqref{eq:exactinteral} or \eqref{eq:interalh}-\eqref{eq:interalv} only give a nonlinear   system of the integral equations with respect to the unknowns
%$\vec{V}({\tt P}) $ implicitly due to the time integrals dependent on  $ \mathcal{S}({\tt Q_{\zeta}}
%(\theta);\theta) $,
%$ S_{1}^{(1)}({\tt Q_{\zeta}}(\theta);\theta)$,
%and $S_{1}^{(2)}({\tt Q_{\zeta}}(\theta);\theta)$,
%thus they have to be further accurately approximated to get the approximate evolution operator $\mathcal{E}_{h}(\tau) $ as well as its limit $\mathcal{E}_{h,0}$
%for the numerical computations.

Our RKDLEG methods only require the approximate local evolution operator $\mathcal{E}_{h,0}$
at the Gauss-Lobatto quadrature nodes.
Without loss of generality, we will only discuss the
(approximate) evolution operator at the grid point $(x_{j},y_{k})$
(not on the edges of the cube face).
 The inner points on the cell edge $ \partial C_{j+\frac{1}{2},k+\frac{1}{2}}$
 will be similarly discussed and simpler than  those grid {points}.
Assume that the coordinates points ${\tt P_{0}}$ and ${\tt P}$ in Fig. \ref{fig:bicharacteristic}
are  $(x_{j},y_{k}, t_{n})$ and
$(x_{j},y_{k},t_{n}+\tau)$ with  $0<\tau\leq\Delta t_{n}$. Such constraint on $\tau$ guarantees
that the bicharacteristic cones past  the Gauss-Lobatto quadrature nodes
do not interact { with} each other. Use $\mathcal{C}_{{\tt P}}^{n}$ to denote the close curve $\left\{\left(x_{j}-d_{1}^{(\ell)}(\theta)\tau,y_{k}-d_{2}^{(\ell)}(\theta)\tau,
t_{n}\right)=:{\tt Q(\theta)}|\ell=1,3,   \theta\in[0,2\pi)\right\} $, which are the intersection of  bicharacteristic cone past the point ${\tt P}$ defined by
\[
\left\{ \left(  x_{j}-d_{1}^{(\ell)}(\theta)\zeta,
y_{k}-d_{2}^{(\ell)}(\theta)\zeta,
t_{n}+\tau-\zeta
\right),\ell=1,3, \theta\in\left[0,2\pi\right),\zeta\in\left[0,\tau\right]\right\} ,\]
with the  ${(\vec{x},t_{n})} $ plane. Under the assumption of $  0<\tau\leq\Delta t_{n} $, the closed curve $\mathcal{C}_{{\tt P}}^{n}$ possibly intersects with following four cell edges past the grid point ${\tt P_{0}}$
\begin{align*}
\mathcal{L}_{P_{0}}^{(1)}&=\left\{(x,y_{k})|x_{j-1}\leq x\leq x_{j}\right\}, \ \mathcal{L}_{P_{0}}^{(2)}=\left\{(x_{j},y)|y_{k-1}\leq y\leq y_{k}\right\},\\
\mathcal{L}_{P_{0}}^{(3)}&=\left\{(x,y_{k})|x_{j}\leq x\leq x_{j+1}\right\}, \ \mathcal{L}_{P_{0}}^{(4)}=\left\{(x_{j},y)|y_{k}\leq y\leq y_{k+1}\right\}
.\end{align*}
In the following, the notation $\hat{N}$ will be used to denote the number of the arc segments of the closed curve $\mathcal{C}_{{\tt P}}^{n}$. Use $\theta_{i} $ to denote the angle corresponding to the $i$th intersection point between $\mathcal{C}_{{\tt P}}^{n}$ and $\left\{\mathcal{L}_{{\tt P}_{0}}^{(\ell)},\ell=1,2,3,4\right\}$ so that the $i$th intersection point is ${\tt Q}
\left(\theta_{i}\right),i=1,2,\cdots,{{\hat{N}}} $, and the closed curve $\mathcal{C}_{{\tt P}}^{n}$ is divided into $\hat{N}$ arc segments, i.e. ``${\rm arc}\ {\tt Q}(\theta_{i}){\tt Q}(\theta_{i+1})$'', $i=1,2,\cdots,\hat{N}$, with $\theta_{\hat{N}+1}=\theta_{1}+2\pi $. Calculation of
$\theta_{i} $ is presented in Appendix \ref{sec:AppendixA}, where the case of inner points on the edge
$ \partial C_{j+\frac{1}{2},k+\frac{1}{2}}$
is also included.

From the exact integral equations \eqref{eq:interalh}-\eqref{eq:interalv}, one may derive the approximate integral equations or the approximate evolution operator $\mathcal{E}_{h}(\tau) $ for the linearized system \eqref{eq:linear} as follows.

\begin{theorem}\label{theor:interalapproximiate}

The linearized system \eqref{eq:linear} has the approximate evolution operator $\mathcal{E}_{h}(\tau)$
defined by
\[
\mathcal{E}_{h}(\tau)\vec{V}({\tt P_{0}})  %:=\vec{V}_{EG}({\tt P})
=\left(h_{EG}({\tt P}),u_{EG}({\tt P}),v_{EG}({\tt P})\right)^T,
\]
%defined by the following approximate integral equations
with
\begin{align}
h_{EG}({\tt P})=&\frac{1}{2\pi} \sum_{i=1}^{{{\hat{N}}}}\int_{\theta_{i}}^{\theta_{i+1}}Jd\theta
          -\frac{\tilde{c}\tilde{\Lambda}}{2\pi g} \sum_{i=1}^{{{\hat{N}}}} \Pi_{0}^{i}(\tau),\label{eq:inapprh}\\
\nonumber
u_{EG}({\tt P})=&-\frac{g}{2\pi\tilde{c}} \sum_{i=1}^{{{\hat{N}}}}\int_{\theta_{i}}^{\theta_{i+1}}  J\tilde{G}_{c}(\theta)d\theta+\left[u({\tt P'})J_{6}-v({\tt P'})J_{4}\right]\\
\nonumber&+{ \frac{\tilde{\Lambda}}{2\pi}} \sum_{i=1}^{{{\hat{N}}}} \Pi_{c}^{i}(\tau)+\frac{1}{2\pi}   \sum_{i=1}^{{{\hat{N}}}} \Pi_{1}^{i}(\tau)+\frac{1}{\tilde{\Lambda} }\left[\left({\tilde{g}_{12}}J_{1}-{\tilde{g}_{11}}J_{2}\right)\left(u({\tt P'})-u({\tt P}) +\tau S_{1}^{(1)}({\tt P'})\right)\right.\\
	    &+\left.\left({\tilde{g}_{22}}J_{1}-{\tilde{g}_{12}}J_{2}\right)\left(v({\tt P'})-v({\tt P})+\tau S_{1}^{(2)}({\tt P'})\right)\right],\label{eq:inappru}\\
\nonumber
v_{EG}({\tt P})=&-\frac{g}{2\pi\tilde{c}}  \sum_{i=1}^{{{\hat{N}}}} \int_{\theta_{i}}^{\theta_{i+1}}  J\tilde{G}_{s}(\theta)d\theta-\left[u({\tt P'})J_{7}-v({\tt P'})J_{5}\right]\\
\nonumber&+{\frac{\tilde{\Lambda}}{2\pi}} \sum_{i=1}^{{{\hat{N}}}} \Pi_{s}^{i}(\tau)+\frac{1}{2\pi}
               \sum_{i=1}^{{{\hat{N}}}} \Pi_{2}^{i}(\tau)+\frac{1}{\tilde{\Lambda} }\left[\left(\tilde{g}_{11}J_{1}-\tilde{g}_{12}J_{3}\right)\left(u({\tt P'})-u({\tt P}) +\tau S_{1}^{(1)}({\tt P'})\right)\right.\\
&+\left.\left(\tilde{g}_{12}J_{1}-\tilde{g}_{22}J_{3}\right)\left(v({\tt P'})-v({\tt P})+\tau S_{1}^{(2)}({\tt P'})\right)\right],\label{eq:inapprv}
\end{align}
where
\begin{equation}
\begin{aligned}
\Pi_{0}^{i}(\tau)=& \int_{\theta_{i}}^{\theta_{i+1}}\left(\phi_{1}'(\theta)u({\tt Q}(\theta))-\phi_{2}'(\theta)v({\tt Q}(\theta))\right)d\theta,\\
\Pi_{c}^{i}(\tau)=&\int_{\theta_{i}}^{\theta_{i+1}}\left(\phi_{3}'(\theta)u({\tt Q}(\theta))-\phi_{4}'(\theta)v({\tt Q}(\theta))\right)d\theta,\\
\Pi_{s}^{i}(\tau)=& \int_{\theta_{i}}^{\theta_{i+1}} \left(\phi_{5}'(\theta)u({\tt Q}(\theta))-\phi_{6}'(\theta)v({\tt Q}(\theta))\right)d\theta,\\
\Pi_{1}^{i}(\tau)=& \tau\int_{\theta_{i}}^{\theta_{i+1}}S_{1}^{(1)}({\tt Q}(\theta);\theta)d\theta,\quad \Pi_{2}^{i}(\tau)= \tau\int_{\theta_{i}}^{\theta_{i+1}}S_{1}^{(2)}({\tt Q}(\theta);\theta)d\theta.
\end{aligned}
\label{intervalgamma}
\end{equation}
Here
\begin{align*}
\phi_{1}:=& \tilde{K}_{\theta}^{2}\tilde{G}_{s}(\theta),
\ \ \phi_{2}:= \tilde{K}_{\theta}^{2}\tilde{G}_{c}(\theta),
\  \  \phi_{3}:=\tilde{G}_{c}(\theta)\phi_{1}(\theta)
,\\
\phi_{4}:=&\tilde{G}_{c}(\theta)\phi_{2}(\theta),
\ \ \phi_{5}:=\tilde{G}_{s}(\theta)\phi_{1}(\theta),
\ \ \phi_{6}:=\tilde{G}_{s}(\theta)\phi_{2}(\theta)
.\end{align*}
\end{theorem}

Before proving {{Theorem}} \ref{theor:interalapproximiate}, the following lemma is first introduced.

\begin{lemma}
\label{lemma:phipsi}
If  $\phi(\theta)\in C^{1}(\mathbb{R})$
and {$\psi(\vec{x},t)$} is continuous and differentiable along the arc segment ``${\rm arc}\ {\tt Q}(\theta_{i}){\tt Q}(\theta_{i+1})$'',
$i=1,2,\cdots,\hat{N}$,  $\theta_{\hat{N}+1}=\theta_{1}+2\pi$,
then the integral relation
\begin{equation}
\label{lemmaphipsi}
\frac{\tilde{c}\tau}{\tilde{\Lambda}}\int_{0}^{2\pi}\frac{\phi(\theta)}{\tilde{K}_{\theta}^{3}}
\left(\sin\theta\psi_{x}({{\tt Q}}(\theta))-\cos\theta\psi_{y}({{\tt Q}}(\theta))\right)d\theta= \int_{0}^{2\pi}\phi'(\theta)\psi({\tt Q}(\theta))d\theta,
\end{equation}
{ holds, where $\psi_{x}$ and $\psi_{y}$ denote  generalized derivatives of $\psi$.}
\end{lemma}

\begin{proof}
Integrating $\frac{d}{d\theta}\left(\phi(\theta)\psi({\tt Q}(\theta))\right)$ along the { the closed curve $\mathcal{C}_{{\tt P}}^{n}$} and using the relations
\[
\frac{d}{d\theta}d_{1}^{(1)}(\theta)=\frac{\tilde{c}}{\tilde{\Lambda}\tilde{K}_{\theta}^{3} }\sin\theta,\quad \frac{d}{d\theta}d_{2}^{(1)}(\theta)=-\frac{\tilde{c}}{\tilde{\Lambda}\tilde{K}_{\theta}^{3} }\cos\theta,
\]
gives
\begin{align*}
\frac{\tilde{c}\tau}{\tilde{\Lambda}}\int_{0}^{2\pi}\frac{\phi(\theta)}
{\tilde{K}_{\theta}^{3}}&\left(\sin\theta\psi_{x}({{\tt Q}}(\theta))
-\cos\theta\psi_{y}({{\tt Q}}(\theta))\right)d\theta- \int_{0}^{2\pi}\phi'(\theta)\psi({{\tt Q}}(\theta))d\theta\\
=&-\int_{0}^{2\pi}\left(\phi(\theta)\psi({{\tt Q}}(\theta))\right)'d\theta
=- \phi(\theta)\psi({{\tt Q}}(\theta)) |_{0}^{2\pi}=0.
\end{align*}
The proof is completed.
\qed \end{proof}

\begin{proof} {\bf of {Theorem}} \ref{theor:interalapproximiate}\
It is divided into two steps.

{\bf Step 1}. Let us consider both integrals depending on the height gradient $\left(\partial_{x}h,\partial_{y}h\right) $ in \eqref{eq:interalu} and \eqref{eq:interalv}.
The linearized system \eqref{eq:linear}  may give
\begin{equation}
\label{hxhy}
\begin{pmatrix}
\partial_{x}h\\
\partial_{y}h
\end{pmatrix}
=\vec{\tilde{L}}\left(\frac{ \tilde{\rm D}\vec{u}}{ \tilde{\rm D}t}-
\begin{pmatrix}
S_{1}^{(1)}\\
S_{1}^{(2)}
\end{pmatrix}\right),
\end{equation}
where
\[
\vec{u}=\left(u,v\right)^{T},\quad \frac{ \tilde{\rm D}}{ \tilde{\rm D}t}=\partial_{t}+\tilde{u}\partial_{x}+\tilde{v}\partial_{y},\quad
\vec{\tilde{L}}=
-\frac{1}{g}
\begin{pmatrix}
\tilde{g}_{11} &\tilde{g}_{12}\\
\tilde{g}_{12} &\tilde{g}_{22}
\end{pmatrix}.
\]
Hence
\begin{align}
\nonumber &\int_{t_{n}}^{t_{n}+\tau}
\begin{pmatrix}
\partial_{x}h({\tt P}_{\zeta}')\\
\partial_{y}h({\tt P}_{\zeta}')
\end{pmatrix}d\zeta
= \vec{\tilde{L}}\left(\frac{ \tilde{\rm D}\vec{u}}{ \tilde{\rm D}t}|{\tt P}_{\zeta}'d\zeta-
\int_{t_{n}}^{t_{n}+\tau}
\begin{pmatrix}
S_{1}^{(1)}({\tt P}_{\zeta}')\\
S_{1}^{(2)}({\tt P}_{\zeta}')
\end{pmatrix}d\zeta\right)\\
\nonumber = &\vec{\tilde{L}}\left(\frac{d  }{ d\zeta}\vec{u}({\tt P}_{\zeta}')d\zeta-
\int_{t_{n}}^{t_{n}+\tau}
\begin{pmatrix}
S_{1}^{(1)}({\tt P}_{\zeta}')\\
S_{1}^{(2)}({\tt P}_{\zeta}')
\end{pmatrix}d\zeta\right) =\vec{\tilde{L}}\left(\left(\vec{u}({\tt P})-\vec{u}({\tt P'})\right)-
\int_{t_{n}}^{t_{n}+\tau}
\begin{pmatrix}
S_{1}^{(1)}({\tt P}_{\zeta}')\\
S_{1}^{(2)}({\tt P}_{\zeta}')
\end{pmatrix}d\zeta\right)\\
 =&\begin{pmatrix}
\frac{1}{g}\left[\tilde{g}_{11}\left(u({\tt P'})-u({\tt P}) + \int_{t_{n}}^{t_{n}+\tau}S_{1}^{(1)}({\tt P}_{\zeta}')d\zeta\right)+\tilde{g}_{12}\left(v({\tt P'})-v({\tt P})+\int_{t_{n}}^{t_{n}+\tau}S_{1}^{(2)}({\tt P}_{\zeta}')d\zeta\right)\right]\\
\frac{1}{g}\left[\tilde{g}_{12}\left(u({\tt P'})-u({\tt P}) +\int_{t_{n}}^{t_{n}+\tau}S_{1}^{(1)}({\tt P}_{\zeta}')d\zeta\right)+\tilde{g}_{22}\left(v({\tt P'})-v({\tt P})+\int_{t_{n}}^{t_{n}+\tau}S_{1}^{(2)}({\tt P}_{\zeta}')d\zeta\right)\right]
\end{pmatrix},
\label{hxhyfinal}
\end{align}
here we have used the fact that
\[
\frac{d}{d\zeta}\vec{u}({\tt P}_{\zeta}')=\frac{d}{d\zeta}\vec{u}\left(  (x-\tilde{u}(t_{n}+\tau-\zeta),
y-\tilde{v}(t_{n}+\tau-\zeta)), \zeta\right)
=\frac{\tilde{\rm D}\vec{u}}{\tilde{\rm D}t}\Big|_{{\tt P}_{\zeta}'}.
\]

{\bf Step 2}. First, approximate three  double integrals containing $ S({\tt Q}_{\zeta}(\theta);\theta) $ in \eqref{eq:interalh}-\eqref{eq:interalv},
two double integrals dependent on $ S_{1}^{(1)}({\tt Q}_{\zeta}(\theta);\theta)$ and $S_{1}^{(2)}({\tt Q}_{\zeta}(\theta);\theta)$,
two single integrals relying on $ S_{1}^{(1)}({\tt P}_{\zeta}')$ and
$S_{1}^{(2)}({\tt P}_{\zeta}')$ in \eqref{hxhyfinal} with the left rectangle rule in the $\zeta $-direction as follows
\begin{align}
& \int_{0}^{2\pi} \int_{t_{n}}^{t_{n}+\tau}\mathcal{S}({\tt Q}_{\zeta}(\theta);\theta)d\zeta d\theta \approx \tau \int_{0}^{2\pi}  \mathcal{S}({\tt Q}(\theta);\theta) d\theta\label{s0q},\\
\nonumber &\\
& \int_{0}^{2\pi} \int_{t_{n}}^{t_{n}+\tau} \tilde{G}_{c}(\theta) \mathcal{S}({\tt Q}_{\zeta}(\theta);\theta)d\zeta d\theta \approx \tau \int_{0}^{2\pi}  \tilde{G}_{c}(\theta) \mathcal{S}({\tt Q}(\theta);\theta) d\theta\label{sqgc},\\
\nonumber &\\
& \int_{0}^{2\pi} \int_{t_{n}}^{t_{n}+\tau} \tilde{G}_{s}(\theta) \mathcal{S}({\tt Q}_{\zeta}(\theta);\theta)d\zeta d\theta \approx \tau \int_{0}^{2\pi}  \tilde{G}_{s}(\theta) \mathcal{S}({\tt Q}(\theta);\theta) d\theta\label{sqgs},\\
\nonumber &\\
& \int_{0}^{2\pi} \int_{t_{n}}^{t_{n}+\tau} S_{1}^{(\ell)}({\tt Q}_{\zeta}(\theta);\theta) \approx \tau \int_{0}^{2\pi}  S_{1}^{(\ell)}({\tt Q}(\theta);\theta),\quad\ell=1,2 \label{sq},\\
\nonumber &\\
& \int_{t_{n}}^{t_{n}+\tau} S_{1}^{(\ell)}({\tt P}_{\zeta}') d\zeta \approx\tau S_{1}^{(\ell)}({\tt P}'),\quad\ell=1,2\label{sp}.
\end{align}

Next, use {{Lemma}} \ref{lemma:phipsi} to handle three integrals
depending on the spatial derivatives of  $\vec{V}$
at the right-hand sides of \eqref{eq:interalh}-\eqref{eq:interalv}.
Taking $\psi=u$ and $\phi(\theta)=\phi_{1}(\theta)  $ in  \eqref{lemmaphipsi} gives
\[
\frac{\tilde{c}\tau}{\tilde{\Lambda}}
\int_{0}^{2\pi}\frac{\tilde{G}_{s}(\theta)}{\tilde{K}_{\theta}}\left(\sin\theta u_{x}({\tt Q}(\theta))-\cos\theta u_{y}({\tt Q}(\theta))\right)d\theta= \int_{0}^{2\pi}\phi_{1}'(\theta)u({\tt Q}(\theta))d\theta.
\]
Again taking $\psi=v$, $\phi(\theta)=
\phi_{2}(\theta) $ in \eqref{lemmaphipsi} leads to
 \[
\frac{\tilde{c}\tau}{\tilde{\Lambda}}\int_{0}^{2\pi}\frac{\tilde{G}_{c}(\theta)}
{\tilde{K}_{\theta}}\left(\sin\theta v_{x}({\tt Q}(\theta))-\cos\theta { v}_{y}({\tt Q}(\theta))\right)d\theta= \int_{0}^{2\pi}\phi_{2}'(\theta)v({\tt Q}(\theta))d\theta.
\]
Subtracting those two equations gives
\[
\tau \int_{0}^{2\pi}\mathcal{S}({\tt Q}(\theta);\theta)d\theta ={\frac{\tilde{\Lambda}}{\tilde{c}}}\int_{0}^{2\pi}\left(\phi_{1}'(\theta)u({\tt Q}(\theta))-\phi_{2}'(\theta)v({\tt Q}(\theta))\right)d\theta=\frac{\tilde{\Lambda}}{\tilde{c}}\sum_{i=1}^{{{\hat{N}}}} \Pi_{0}^{i}(\tau).
\]
Similarly using {{Lemma}} \ref{lemma:phipsi} may get
\begin{align*}
\tau \int_{0}^{2\pi}\tilde{G}_{c}(\theta)\mathcal{S}({\tt Q}(\theta);\theta)d\theta &={\frac{\tilde{\Lambda}}{\tilde{c}}}\int_{0}^{2\pi}\left(\phi_{3}'(\theta)u({\tt Q}(\theta))-\phi_{4}'(\theta)v({\tt Q}(\theta))\right)d\theta=\frac{\tilde{\Lambda}}{\tilde{c}}\sum_{i=1}^{{{\hat{N}}}} \Pi_{c}^{i}(\tau),\\
\tau \int_{0}^{2\pi}\tilde{G}_{s}(\theta)\mathcal{S}({\tt Q}(\theta);\theta)d\theta &={\frac{\tilde{\Lambda}}{\tilde{c}}}\int_{0}^{2\pi}\left(\phi_{5}'(\theta)u({\tt Q}(\theta))
-\phi_{6}'(\theta)v({\tt Q}(\theta))\right)d\theta=\frac{\tilde{\Lambda}}{\tilde{c}}\sum_{i=1}^{{{\hat{N}}}} \Pi_{s}^{i}(\tau).
\end{align*}
Because the closed curve $\mathcal{C}_{{\tt P}}^{n}$ is divided into $\hat{N}$ arc segments and the approximate solution $\vec{V}_{h}(\vec{x},t_{n})$ smooth along each arc segment of   $\mathcal{C}_{{\tt P}}^{n}$,  combining the above three relations with \eqref{hxhyfinal}-\eqref{sp} as well as \eqref{eq:interalh}-\eqref{eq:interalv} may completes the proof of {{Theorem}} \ref{theor:interalapproximiate}.
\qed \end{proof}

%\begin{remark}
%From the exact integral equations \eqref{eq:interalh}-\eqref{eq:interalv} to the approximate integral equations \eqref{eq:inapprh}-\eqref{eq:inapprv}, numerical approximation(i.e. the left rectangle rule for the integral with regard to $zeta$) is only employed in \eqref{sq} and \eqref{sp}.
%\end{remark}

\begin{remark}
The approximate integral equations \eqref{eq:inappru}-\eqref{eq:inapprv} form a $2\times2 $ system of the linear algebraic equations with respect to the unknowns $\left(u_{h}({\tt P}),v_{h}({\tt P})\right) $.
Solving this linear system may give the explicit expression of the approximate evolution operator $\mathcal{E}_{h}(\tau) $.
However, it still { contains} the complicate integrals with respect to $\theta$
so that the calculation of  flux integral
at the right-hand side of \eqref{semiDG} is very time-consuming and technical.
In order to avoid such difficulty,   the approximate local evolution operator $\mathcal{E}_{h,0} $
is introduced to replace $\mathcal{E}_{h}(\tau) $, see \cite{SunRen2009,Wu:2014}.
% , the limit of $\mathcal{E}_{h}(\tau) $ as $\tau$ approaches to zero,
\end{remark}

% % % % % % % % % % % % % % % % % % % % % % % % % % % % % % % % % % % % % % %5
\subsubsection{Approximate local evolution operator $\mathcal{E}_{h,0} $}
\label{subsubsec:local}
This section derives the approximate local evolution operator
$\mathcal{E}_{h,0} $ defined by
\[
\mathcal{E}_{h,0}:=\lim_{\tau\to {0^+}}\mathcal{E}_{h}(\tau),
\]
which only requires to evolve the solutions to the time $t_n+\tau $  from the ``initial'' time $t_{n} $,
where $0<\tau \ll 1 $.
%The RKDLEG methods may maintain the genuinely multi-dimensional nature in numerical flux evaluation of the EG method.
%
Since the coordinates of ${\tt P'}$ and ${\tt Q}(\theta)$ are
$
\left(x_{j}-\tilde{u}\tau,y_{k}-\tilde{v}\tau,t_{n} \right)
$
and $
\left(x_{j}-d_{1}^{(1)}(\theta)\tau,y_{k}-d_{2}^{(1)}(\theta)\tau, t_{n}\right)
$, respectively,
both ${\tt P'}$ and ${\tt Q}(\theta)$ will tend to the point ${\tt P}_{0}$,
 and the length of the arc segment ``${\rm arc}\ {\tt Q}\left(\theta_{i}\right){\tt Q}\left(\theta_{i+1}\right)$'' will also approach to zero, as $\tau\to {0^{+}} $. Hence, one has
\begin{align*}
\lim_{\tau\to 0^{+}} \Pi_{0}^{i} (\tau)=& \lim_{\tau\to 0^{+}} \int_{\theta_{i}}^{\theta_{i+1}}\left(\phi_{1}'(\theta)u({\tt Q}(\theta))-\phi_{2}'(\theta)v({\tt Q}(\theta))\right)d\theta\\
				           =& u_{i}^{*}\left(\phi_{1}(\theta_{i+1})-\phi_{1}(\theta_{i})\right)- v_{i}^{*}\left(\phi_{2}(\theta_{i+1})-\phi_{2}(\theta_{i})\right)=:\Pi_{0,0}^{i},
\end{align*}
where
\[
\vec{V}_{i}^{*}{ {=}}\left(h_{i}^{*},u_{i}^{*},v_{i}^{*}\right)^{T}:=\lim_{\tau\to 0^{+}}\vec{V}({\tt Q}(\theta)),\theta\in\left(\theta_{i},\theta_{i+1}\right),i=1,2,\cdots,\hat{N}.
\]
Similarly,  one may also get
\begin{align*}
\lim_{\tau\to 0^{+}} \Pi_{c}^{i} (\tau)=& u_{i}^{*}\left(\phi_{3}(\theta_{i+1})-\phi_{3}(\theta_{i})\right)- v_{i}^{*}\left(\phi_{4}(\theta_{i+1})-\phi_{4}(\theta_{i})\right)=:\Pi_{c,0}^{i},\\
\lim_{\tau\to 0^{+}} \Pi_{s}^{i} (\tau)=& u_{i}^{*}\left(\phi_{5}(\theta_{i+1})-\phi_{5}(\theta_{i})\right)- v_{i}^{*}\left(\phi_{6}(\theta_{i+1})-\phi_{6}(\theta_{i})\right)=:\Pi_{s,0}^{i}
,\end{align*}
and
\[
\lim_{\tau\to 0} \Pi_{1}^{i}(\tau)= \lim_{\tau\to 0}\tau\int_{\theta_{i}}^{\theta_{i+1}}S_{1}^{(\ell)}({\tt Q}(\theta);\theta)d\theta=0,\ell=1,2.
\]
%due to the bounded integrand,
In view of the above facts,  taking the limit of the approximate integral equations \eqref{eq:inapprh}-\eqref{eq:inapprv} as $\tau\to 0^{+} $  leads to the following approximate local
integral equations
\begin{align}
h_{LEG}({\tt P})=&\frac{1}{2\pi} \sum_{i=1}^{{{\hat{N}}}} \left[h_{i}^{*}- \frac{\tilde{c}}{g}\left(u_{i}^{*} \int_{\theta_{i}}^{\theta_{i+1}} \frac{\cos\theta}{\tilde{K}_{\theta}}d\theta +v_{i}^{*} \int_{\theta_{i}}^{\theta_{i+1}}\frac{\sin\theta}{\tilde{K}_{\theta}} d\theta\right)\right]
          -\frac{\tilde{c}\tilde{\Lambda}}{2\pi g} \sum_{i=1}^{{{\hat{N}}}} \Pi_{0,0}^{i},
          \label{eq:localapproh}\\
\nonumber
u_{LEG}({\tt P})=&\frac{g}{2\pi\tilde{c}} \sum_{i=1}^{{{\hat{N}}}} \left[-h_{i}^{*} \int_{\theta_{i}}^{\theta_{i+1}} \tilde{G}_{c}(\theta)d\theta + \frac{\tilde{c}}{g}\left(u_{i}^{*} \int_{\theta_{i}}^{\theta_{i+1}} \frac{\tilde{G}_{c}(\theta) \cos\theta}{\tilde{K}_{\theta}}d\theta +v_{i}^{*} \int_{\theta_{i}}^{\theta_{i+1}}\frac{\tilde{G}_{c}(\theta)
 \sin\theta}{\tilde{K}_{\theta}} d\theta\right)\right] \\
\nonumber	  &+\frac{1}{\tilde{\Lambda} }\left[\left({\tilde{g}_{12}}J_{1}-{ \tilde{g}_{11}}J_{2}\right)\left(  u_{0}^{*}  -u_{h}({\tt P})\right)+\left({ \tilde{g}_{22}}J_{1}-{ \tilde{g}_{12}}J_{2}\right)
\left( v_{0}^{*} -v_{h}({\tt P})\right)\right]\\
          &+\left[ u_{0}^{*} J_{6}-v_{0}^{*}J_{4}\right]+{\frac{\tilde{\Lambda}}{2\pi}} \sum_{i=1}^{{{\hat{N}}}} \Pi_{c,0}^{i},\label{eq:localapprou}\\
\nonumber
v_{LEG}({\tt P})=&\frac{g}{2\pi\tilde{c}} \sum_{i=1}^{{{\hat{N}}}} \left[-h_{i}^{*}\int_{\theta_{i}}^{\theta_{i+1}} \tilde{G}_{s}(\theta)d\theta + \frac{\tilde{c}}{g}\left(u_{i}^{*} \int_{\theta_{i}}^{\theta_{i+1}} \frac{\tilde{G}_{s}(\theta) \cos\theta}{\tilde{K}_{\theta}}d\theta +v_{i}^{*} \int_{\theta_{i}}^{\theta_{i+1}}\frac{\tilde{G}_{s}(\theta)
\sin\theta}{\tilde{K}_{\theta}} d\theta\right)\right] \\
\nonumber	   &+\frac{1}{\tilde{\Lambda} }\left[\left(\tilde{g}_{11}J_{1}-\tilde{g}_{12}J_{3}\right)\left(u_{0}^{*}-u_{h}({\tt P})\right)+\left(\tilde{g}_{12}J_{1}-\tilde{g}_{22}J_{3}\right)\left(v_{0}^{*}-v_{h}({\tt P})\right)\right]\\
          &+\left[ u_{0}^{*} J_{7}-v_{0}^{*}J_{5}\right]+{\frac{\tilde{\Lambda}}{2\pi}} \sum_{i=1}^{{{\hat{N}}}} \Pi_{s,0}^{i},\label{eq:localapprov}
\end{align}
where
\[
\left(h_{0}^{*},u_{0}^{*},v_{0}^{*}\right)^T=\vec{V}_{0}^{*}:=\lim_{\tau\to 0^{+}} \vec{V}_{h}({\tt P'}),
 \]
and $\vec{V}_{h}$ is the  approximate solutions in primitive variable of the { RKDLEG} methods.
 Eqs. \eqref{eq:localapproh}-\eqref{eq:localapprov} define our approximate local evolution operator, i.e.
\[
\mathcal{E}_{h,0}\vec{V}\left({\tt P_{0}}\right):=\vec{V}_{LEG}\left({\tt P}\right)=\left(h_{LEG}\left({\tt P}\right),u_{LEG}\left({\tt P}\right),v_{LEG}\left({\tt P}\right)\right)^{T},
\]
implicitly.

\begin{remark}
\label{remark:integral}
All integrals with regard to $\theta$ at the right-hand sides of  \eqref{eq:localapproh}-\eqref{eq:localapprov} can be exactly evaluated.
In fact, the integrands in those
integrals are
\begin{align*}
& \frac{\cos\theta}{\tilde{K}_{\theta}},
\  \frac{\sin\theta}{\tilde{K}_{\theta}},
\   \frac{\tilde{G}_{c}(\theta) \cos\theta}{\tilde{K}_{\theta}},
\
\frac{\tilde{G}_{c}(\theta) \sin\theta}{\tilde{K}_{\theta}},
\
  \frac{\tilde{G}_{s}(\theta) \cos\theta}{\tilde{K}_{\theta}},
\
  \frac{\tilde{G}_{s}(\theta) \sin\theta}{\tilde{K}_{\theta}}, \
 \tilde{G}_{c}(\theta),\   \tilde{G}_{s}(\theta),
\end{align*}
whose antiderivatives or primitive functions may be gotten with the aid of
the following identities
\begin{align*}
&\int^{\theta}
\frac{\sin\tilde{\theta}\cos\tilde{\theta}}{\lambda_{1}
\cos^{2}\tilde{\theta}+\lambda_{2}\sin^{2}\tilde{\theta}}d\tilde{\theta}=
\begin{cases}
-\frac{\cos^{2}\theta}{2\lambda_{1}},&\lambda_{1}=\lambda_{2},\\
\frac{\ln\left(\lambda_{1}\cos^{2}\theta+\lambda_{2}\sin^{2}\theta\right)}
{2\left(\lambda_{2}-\lambda_{1}\right)},&\lambda_{1}\neq\lambda_{2},
\end{cases}
\\
&\int^{\theta} \frac{\sin^{2}\tilde{\theta}}{\lambda_{1}\cos^{2}\tilde{\theta}
+\lambda_{2}\sin^{2}\tilde{\theta}}d\tilde{\theta}=
\begin{cases}
\frac{\theta-\sin\theta\cos\theta}{2\lambda_{1}},&\lambda_{1}=\lambda_{2},\\
\frac{\sqrt{\lambda_{1}/\lambda_{2}}}{\lambda_{1}-\lambda_{2}} \left(\arctan\left(\tan\theta\sqrt{\lambda_{1}/\lambda_{2}}\right)+\pi\lfloor\frac{\theta}{\pi}-\frac{1}{2}\rfloor\right)
-\frac{\theta}{\lambda_{1}-\lambda_{2}},&\lambda_{1}\neq\lambda_{2},
\end{cases}\\
&\int^{\theta} \frac{\cos^{2}\tilde{\theta}}{\lambda_{1}\cos^{2}\tilde{\theta}
+\lambda_{2}\sin^{2}\tilde{\theta}}d\tilde{\theta}=
\begin{cases}
\frac{\theta+\sin\theta\cos\theta}{2\lambda_{1}},&\lambda_{1}=\lambda_{2},\\
\frac{\sqrt{\lambda_{2}/\lambda_{1}}}{\lambda_{2}-\lambda_{1}} \left(\arctan\left(\tan\theta\sqrt{\lambda_{2}/\lambda_{1}}\right)+\pi\lfloor\frac{\theta}{\pi}-\frac{1}{2}\rfloor\right)
-\frac{\theta}{\lambda_{1}-\lambda_{2}},&\lambda_{1}\neq\lambda_{2},
\end{cases}\\
&\int^{\theta} \frac{\sin\tilde{\theta}}{\sqrt{\lambda_{1}\cos^{2}\tilde{\theta}
+\lambda_{2}\sin^{2}\tilde{\theta}}}d\tilde{\theta}=
\begin{cases}
-\frac{\cos\theta}{\sqrt{\lambda_{1}}},&\lambda_{1}=\lambda_{2},\\
-\frac{\ln\left(\frac{\lambda_{2}}{\sqrt{\lambda_{1}-\lambda_{2}}\left|\cos\theta\right|
+\sqrt{\left(\lambda_{1}-\lambda_{2}\right)\cos^{2}\theta+\lambda_{2}} }\right)}{\sqrt{\lambda_{1}-\lambda_{2}}}
,&\cos\theta<0, \lambda_{1}\neq \lambda_{2},\\
-\frac{\ln\left(\sqrt{\lambda_{1}-\lambda_{2}}\cos\theta+\sqrt{\left(\lambda_{1}-\lambda_{2}\right)\cos^{2}\theta
+\lambda_{2}}\right)}{\sqrt{\lambda_{1}-\lambda_{2}}},&\cos\theta\geq0, \lambda_{1}\neq \lambda_{2},
\end{cases}
\\
&\int^{\theta} \frac{\cos\tilde{\theta}}{\sqrt{\lambda_{1}\cos^{2}\tilde{\theta}
+\lambda_{2}\sin^{2}\tilde{\theta}}}d\tilde{\theta}=
\begin{cases}
\frac{\sin\theta}{\sqrt{\lambda_{1}}},&\lambda_{1}=\lambda_{2},\\
-\frac{\left(\arctan\left( \frac{ \frac{\lambda_{1}}{2\sin\theta}+\left(\lambda_{2}-\lambda_{1}\right)\sin\theta }{ \sqrt{\lambda_{1}-\lambda_{2}}\sqrt{\lambda_{1}\cos^{2}\theta+\lambda_{2}\sin^{2}\theta}}\right)
+\pi\lfloor\frac{\theta}{\pi}\rfloor-2\pi\lfloor\frac{\theta}{2\pi}\rfloor\right) }{2\sqrt{\lambda_{1}-\lambda_{2}}},&\lambda_{1}\neq \lambda_{2},
\end{cases}
\end{align*}
where $\lambda_{1}$ and $\lambda_{2} $ are the eigenvalues of  $\vec{G}^{-1}$, satisfying
\[
\vec{\tilde{G}}^{-1}=\begin{pmatrix}
\tilde{g}^{11} & \tilde{g}^{12}\\
\tilde{g}^{21}& \tilde{g}^{22}
\end{pmatrix}
=\vec{\tilde{Q}}^{T}
\begin{pmatrix}
\lambda_{1}&0\\
0&\lambda_{2}
\end{pmatrix}\vec{\tilde{Q}}, \
\vec{\tilde{Q}}=
\begin{pmatrix}
\cos\phi_{G}&-\sin\phi_{G}\\
\sin\phi_{G}&\cos\phi_{G}
\end{pmatrix},  \ \lambda_{1}\geq\lambda_{2}.
\]
\end{remark}

\subsubsection{Treatment of subregion boundaries}
\label{subsubsection:edges}
The transformations from the reference region $\tilde{\Omega}$
to six faces of the cubed sphere are different {from}  each other and
not continuous across the edges of the  cubed sphere.
It means that  the approximate local evolution operators
corresponding to different cubed sphere faces will give different states   \eqref{eq:localapproh}-\eqref{eq:localapprov}  so that
 the conservation of the numerical  flux
 cannot be ensured  on the  edges  of the cubed-sphere face.
 % Section \ref{subsubsec:nonlocal}
Thus it is necessary to propose some special  treatments
%around the edges  of the cubed-sphere face
in order to get the conservation of the numerical  flux on the  edges  of the cubed-sphere face.
%because the coordinate transformation across the edges of the cubed-sphere face is not continuous

To avoid such flaw, around the    edges  of the cubed-sphere face,
%To preserve conservation of flux which will be proofed in section \ref{subsec:conservation},
 the SWEs  in the LAT/LON coordinates %\eqref{eq:latlon}
 are  linearized
on  the edges of cubed-sphere face
 instead of linearizing the {{SWEs}}  in the reference coordinates %\eqref{eq:primitive}
 and then its approximate local evolution operator
 is derived and used to  replace that defined by  \eqref{eq:localapproh}-\eqref{eq:localapprov}.
To accomplish such task,
the SWE{{s}}  \eqref{eq:latlon} are reformulated as follows
\begin{equation}
\label{eq:latlon primitive}
\frac{\partial \vec{V}_{s} }{\partial t} + \vec{A}_{s}^{1}\left(\vec{V}_{s},\vec{\xi}\right) \frac{\partial \vec{V}_{s} }{\partial \xi} + \vec{A}_{s}^{2}\left(\vec{V}_{s},\vec {\xi}\right) \frac{\partial \vec{V}_{s} }{\partial \eta} = \vec{S}_{s}
,\end{equation}
where $\vec{V}_{s}=\left(h,u_{s},v_{s}\cos \eta \right)^{T}$, $\vec{\xi}=\left(\xi,\eta\right)$, $\vec{S}_{s}=\left(0,S_{s}^{(1)},S_{s}^{(2)}\right)^{T}$,
\begin{align*}
&\vec{A}_{s}^{1}\left(\vec{V}_{s},\vec{\xi}\right)=
\frac{1}{R\cos\eta}
 \begin{pmatrix}
u_{s} &h &0\\
gg_{s}^{11} & u_{s} & 0\\
gg_{s}^{12} & 0 & u_{s}
\end{pmatrix}, \quad
\vec{A}_{s}^{2}\left(\vec{V}_{s},\vec{\xi}\right)=
\frac{1}{R\cos\eta}
 \begin{pmatrix}
v_{s}\cos\eta &0 &h\\
gg_{s}^{12} & v_{s}\cos\eta & 0\\
gg_{s}^{22} & 0 & v_{s}\cos\eta
\end{pmatrix},
\\
&S_{s}^{(1)}=\frac{1}{R\cos\eta}\left(fv_{s}\cos\eta+u_{s}v_{s}\sin\eta-gg_{s}^{11}b_{\xi}\right),\\
&S_{s}^{(2)}=-\frac{1}{R\cos\eta}\left(fu_{s}\cos\eta+u_{s}^{2}\sin\eta+gg_{s}^{22}b_{\eta}+v_{s}^{2}\cos\eta\sin\eta\right),
\end{align*}
and
\[
\vec{G}_{s}^{-1}=
\begin{pmatrix}
g_{s}^{11} &g_{s}^{12}\\
g_{s}^{12} &g_{s}^{22}
\end{pmatrix}=
\begin{pmatrix}
1 & 0\\
0 & \cos^{2}\eta
\end{pmatrix}
.\]
Similarly,
if taking $\vec{\tilde{\xi}}=(\tilde{\xi},\tilde{\eta})$ and  $\vec{\tilde{V}}_{s}=[\tilde{h}_{s}, \tilde{u}_{s},\tilde{v}_{s}\cos\tilde{\eta}]^{T}$
and  as a reference point and state of  $\vec{V}_{s}(\xi,\eta,t)$,
then the system \eqref{eq:latlon primitive} may be linearized as follows
\begin{equation}
\label{eq:linear latlon}
\frac{\partial \vec{V}_{s} }{\partial t} + \vec{A}_{s}^{1}
(\vec{\tilde{V}}_{s},\vec{\tilde{\xi}}) \frac{\partial \vec{V}_{s} }{\partial \xi} + \vec{A}_{s}^{2}(\vec{\tilde{V}}_{s},\vec {\tilde{\xi}}) \frac{\partial \vec{V}_{s} }{\partial \eta} = \vec{S}_{s},
\end{equation}
%{eq:linear}
whose form is similar to  the previous linearized system \eqref{eq:linear} in the reference coordinates.
On the other hand, the derivation of the approximate local evolution operator of \eqref{eq:linear}
does not require the concrete form of $\vec{G}^{-1} $, $\vec{V} $ and $\vec{x}$.
 Hence
 the approximate local evolution operator of the system
\eqref{eq:linear latlon} may be
  derived in parallel by replacing $\vec{G}^{-1} $, $\vec{V} $,
 and $\vec{x} $ with
  $ \vec{G}_{s}^{-1}$,  $\vec{V}_{s} $, and  $\vec{\xi} $
  in Sections \ref{subsec:exact operator}, \ref{subsubsec:nonlocal} and \ref{subsubsec:local}, respectively.
 However, a special attention should be paid to
{calculate} the intersection points
between the bottom of the bicharacteristic cone
past the point on the edges of the cubed sphere face
and the cell edges  in the $(\xi,\eta)$ plane.
The readers are referred to Appendix \ref{sec:AppendixC}
for the detailed discussion.
%
%The cell edge $\partial C_{j+\frac12,k+\frac12}$ in the $(x,y)$ plane should be mapped
%to the LAT/LON plane and then
%the intersection points between the bottom of the bicharacteristic cone and those cell edges
%are calculated in order
% to derive the approximate local evolution operator of
% \eqref{eq:linear latlon}.
Because  the cell edges in the LAT/LON plane are not straight in general,
calculation of those {intersection} points is different from those
in the $(x,y)$ plane discussed in Appendix \ref{sec:AppendixA}.
% the equation which the point on the curve of the projection satisfied. These equations including the calculation of the intersection points between the curves and the bottom of the bicharacteristic cone see Appendix \ref{sec:AppendixC}, then we can obtain the approximate local evolution operator on LAT/LON coordinates for the point on the patch boundaries.

%It shows concretely as: $ G_{s}^{-1}$
%corresponding to $G^{-1} $, $\vec{V}_{s} $ corresponding to $\vec{V} $, $\vec{\xi} $
%corresponding to $\vec{x} $.

%For sake of convenience, the notions
%$G^{-1} $, $\vec{V} $ and $\vec{x} $ are used in the following.
\begin{remark}
Because the edges of the cubed sphere face
do not pass through the spherical pole,
  the pole singularity in the LAT/LON coordinates may be gotten around.
\end{remark}

At the end of this section, the conservation of the numerical flux
of the RKDLEG method
on the edges of cubed sphere faces.
 Let $\mathcal{L} $ and $\mathcal{L}_{s}$
 denote the edge  of the cubed sphere face
in the $(x,y)$  and  LAT/LON planes, respectively,
 and $\vec{n}$ and $\vec{n}_{s} $ be their outward unit normal vectors.
%Let $A$ denote the relation between $\left(u_{s},v_{s}\right)$ and %$\left(u,v\right)$   %\eqref{veclocity}
% for one patch,
% the metric tensor is $G=A^{T}A $,

\begin{theorem}
\label{conservation}
The numerical flux
 of the RKDLEG method in the LAT/LON plane
\[
\int_{\mathcal{L}}
\begin{pmatrix}
\vec{I} & 0\\
0 & \vec{A}
\end{pmatrix}
 \vec{F}\left(\mathcal{E}_{h,0}\vec{V}_{h}\left(t\right)\right)\cdot\vec{n}dl,
\]
does not depend on the transformation from
$(x,y)$ to   $(\xi,\eta)$,
but relies on the height $h$ and velocity
$\left(u_{s},v_{s}\right)$ in the LAT/LON plane.
\end{theorem}
\begin{proof}
\label{proof:conservation}
Due to the transformation between the reference coordinates $(x,y)$
and LAT/LON coordinates $(\xi,\eta)$, one has
 $\vec{A}^{T}\cdot \vec{n}_{s}=\vec{n} $.
 Using  \eqref{veclocity}, \eqref{metric}, and
%\[
%G=A^{T}A,\quad A\begin{pmatrix}
%u\\
%v
%   \end{pmatrix}=
%\begin{pmatrix}
%u_{s}\\
%v_{s}
%\end{pmatrix},
%\]
\[
\vec{F}(\vec{U}) =
\begin{pmatrix}
\Lambda hu & \Lambda hv \\
 \Lambda\left(hu^{2}+\frac{1}{2}gg^{11}h^{2}\right) & \Lambda\left(huv+\frac{1}{2}gg^{12}h^{2}\right)\\
\Lambda \left(huv+\frac{1}{2}gg^{12}h^{2}\right) & \Lambda \left(hu^{2}+\frac{1}{2}gg^{22}h^{2}\right)
\end{pmatrix},
\]
gives
\begin{align*}
& \int_{\mathcal{L}}
\begin{pmatrix}
\vec{I} & 0\\
0 & \vec{A}
\end{pmatrix}
 \vec{F}\left(\vec{U}\right)\cdot\vec{n}dl
= \int_{\mathcal{L}}
\begin{pmatrix}
\vec{I} & 0\\
0 & \vec{A}
\end{pmatrix}
 \vec{F}\left(\vec{U}\right)\vec{A}^{T}\vec{n}_{s}dl\\
&=\int_{\mathcal{L}}
\begin{pmatrix}
\vec{I} & 0\\
0 & \vec{A}
\end{pmatrix}
\begin{pmatrix}
\Lambda hu & \Lambda hv \\
 \Lambda\left(hu^{2}+\frac{1}{2}gg^{11}h^{2}\right) & \Lambda\left(huv+\frac{1}{2}gg^{12}h^{2}\right)\\
\Lambda \left(huv+\frac{1}{2}gg^{12}h^{2}\right) & \Lambda \left(hv^{2}+\frac{1}{2}gg^{22}h^{2}\right)
\end{pmatrix}
 \vec{A}^{T}\vec{n}_{s}dl\\
&=\int_{\mathcal{L}}
\begin{pmatrix}
\Lambda hu_{s} & \Lambda hv_{s} \\
 \Lambda\left((hu_{s}^{2}+\frac{1}{2}gh^{2}\right) & \Lambda\left(hu_{s}v_{s}+\frac{1}{2}gh^{2}\right)\\
\Lambda \left((hu_{s}v_{s}+\frac{1}{2}gh^{2}\right) & \Lambda \left(hv_{s}^{2}+\frac{1}{2}gh^{2}\right)
\end{pmatrix}
\vec{n}_{s}dl\\
&=\int_{\mathcal{L}_{s}}
\begin{pmatrix}
hu_{s} &  hv_{s} \\
\left(hu_{s}^{2}+\frac{1}{2}gh^{2}\right) & \left(hu_{s}v_{s}+\frac{1}{2}gh^{2}\right)\\
 \left(hu_{s}v_{s}+\frac{1}{2}gh^{2}\right) & \left(hv_{s}^{2}+\frac{1}{2}gh^{2}\right)
\end{pmatrix}
\vec{n}_{s}dl_{s},
\end{align*}
in which the states %in the last boundary integral
are evolved through the linearized SWE{{s}} in the LAT/LON coordinates.
The proof is completed.
\qed\end{proof}

%Due {\color{red}{to}} the treatment of patch boundaries in Section \ref{subsubsection:edges} ({\color{red}{using}} approximate local evolution operator for the linearized system on LAT/LON coordinates for the points on the patch boundaries), we obtain that the numerical flux which gives in the patches located by both sides of the interface are consistent, for the cell interface which {\color{red}{locates}} on the patch boundaries.

%%%%%%%%%%%%%%%%%%%%%%%%%%%%%%%%%%%%%%%%%%%%%%%%%%%%%%%%%%%%%%
\section{Numerical experiments}
\label{sec:numerical-results}
This section will apply  the proposed  % high-order accurate Runge-Kutta discontinuous local evolution Galerkin
RKDLEG methods to  several benchmark problems \cite{Williamson:1992} for the SWE{{s}} on the sphere to demonstrate the accuracy and performance of the present methods.
In our computations,  the CFL number $C_{cfl}$ is taken as
0.25, 0.15, and 0.1 for the ${ \mathbb{P}^{1}}$-, ${ \mathbb{P}^{2}}$-, and ${\mathbb{P}^{3}}$-based RKDLEG methods, respectively,
and the region $\{ (x,y)|~ x,y\in {\tilde{\Omega}}\}$ in
Fig. \ref{fig:sphere} {(b)}
is divided into $N\times N$ uniform cells, that is, the sphere surface is  partitioned into $6(N\times N)$ cells.

\begin{figure}[htbp]
  \centering
  \includegraphics[width=10cm,height=5cm]{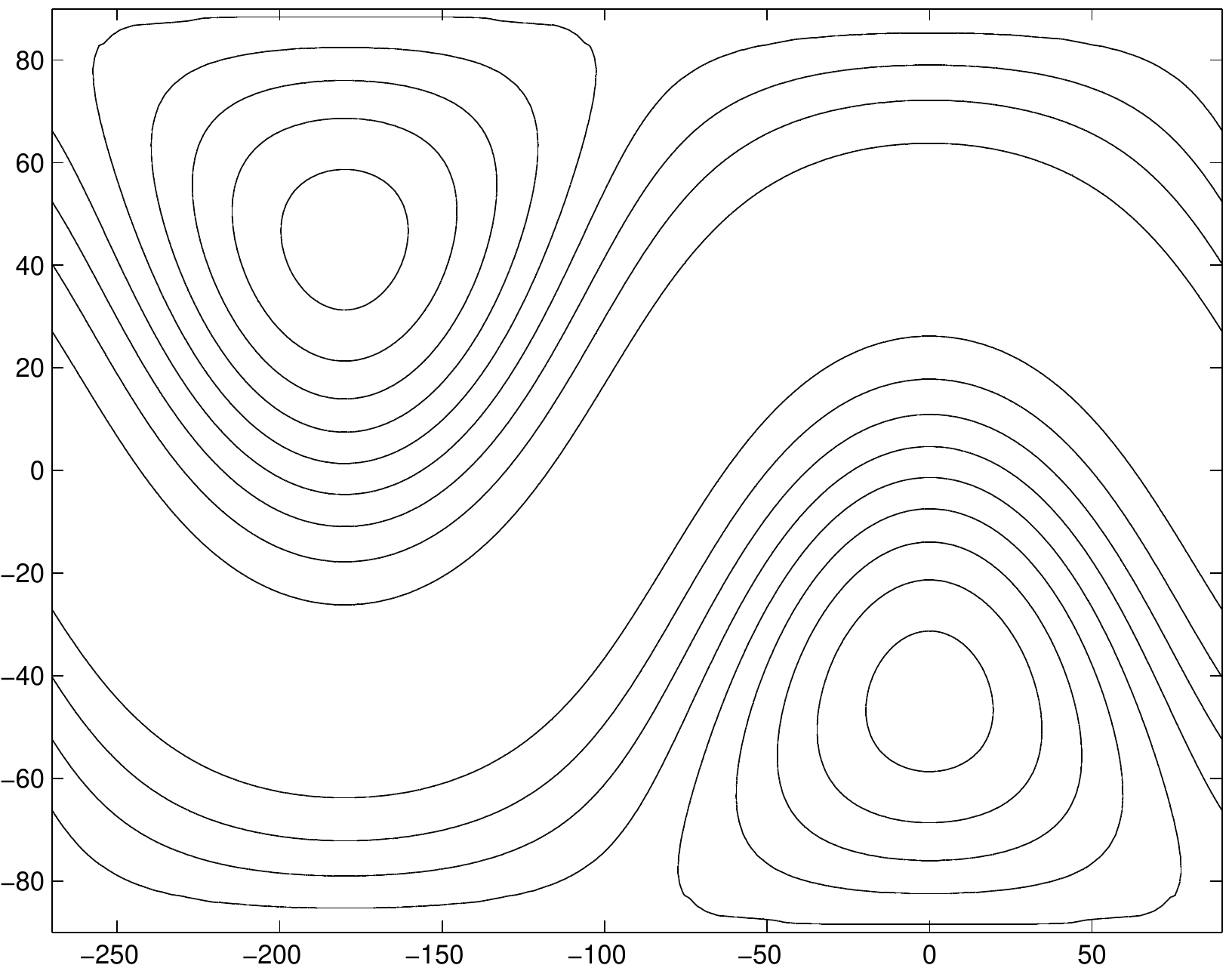}
  \caption{\small Example \ref{steady}: The height  $h(\xi,\eta,t)$ at $t=3$ days obtained by using
  	the ${ \mathbb{P}^{3}}$-based RKDLEG method with $N=64$. Contour lines are equally spaced {{from}} $1150$ m to $2950$ m
  	with a stepsize of $200$ m. }
  \label{fig:steady}
\end{figure}

\begin{example}[Steady state zonal geostrophic flow] \label{steady}\rm
This example is   Williamson's test case 2 \cite{Williamson:1992}, in which the initial height and divergence-free velocity vector in the LAT/LON coordinates $(\xi,\eta)$
are given by
\begin{align}\label{EQ-Example4.1} \begin{aligned}
h(\xi,\eta,0)=&h_{0}-g^{-1}\left(R\Omega u_{0}+\frac{u_{0}^{2}}{2}\right)\left(-\cos\xi\cos\eta\sin\alpha+\sin\eta\cos\alpha\right)^2,
\\
%the initial divergence-free velocity field in spherical component form is taken as
u_{s}(\xi,\eta,0)=& u_{0}\left(\cos\eta\cos\alpha+\cos\xi\sin\eta\sin\alpha\right),\
v_{s}(\xi,\eta,0)=-u_{0}\sin\xi\sin\alpha,
\end{aligned}\end{align}
where $h_{0}=2.94\times 10^{4}  {\mbox{ m}}$, $u_{0}=\frac{2\pi R}{12}\mbox{ day}^{-1}$, and $\alpha$ denotes the angle between the rotational  and polar axises of the sphere (or {the} Earth) and may be chosen as $\alpha=0$ or $ \frac{\pi}{4} $ or $\frac{\pi}{2}$.
The Coriolis force is calculated as
\[
f=2\Omega\left(-\cos\xi\cos\eta\sin\alpha+\sin\eta\cos\alpha\right).
\]
The exact solution to this problem describes a steady state flow, where
the physical variables $h$, $u_s$, and $v_s$ at any time are the same as the initial.
Fig.~\ref{fig:steady}  shows the height $h(\xi,\eta,t)$  at $t=3$ days obtained by the ${ \mathbb{P}^{3}}$-based RKDLEG method with $N=64$.
Tables \ref{tab:steady1}-\ref{tab:steady3} list the relative errors in the height $h$
at $t=3$ days  and corresponding convergence rates of   the ${ \mathbb{P}^{1}}$-, ${ \mathbb{P}^{2}}$-, and ${\mathbb{P}^{3}}$-based RKDLEG methods, where
the $l_1$-, $l_2$-, and $l_\infty$-error{{s}} are respectively measured by \cite{Williamson:1992}
\[
\frac{\int_{S}\left|  h_{h}-h\right|ds}{\int_{S}\left|h\right|ds},\quad \frac{\left[\int_{S}\left(h_{h}-h\right)^{2}ds\right]^{\frac{1}{2}}}
{\left(\int_{S}h^{2}ds\right)^{\frac{1}{2}}},\quad
\frac{\max  \{ \left|h_{h}-h\right| \}   }{\max \{\left|h\right| \}  }.
\]
Here $S$ is the whole sphere surface,
$h_{h}$, and $h$ denote the numerical and exact heights, respectively, and
those  integrations are calculated by using the Gauss-Lobatto quadrature rule.
Those data show that the ${\mathbb{P}^{K}}$-based RKDLEG method is of {$(K+1)$th} order of convergence, $k=1,2,3$.
Fig. \ref{fig:steadyerror3} displays the  time evolutions of the log of relative errors to base 10 in $h,u,v$  obtained by the ${ \mathbb{P}^{3}}$-based RKDLEG method with $N=64$.

%We examined the convergence rate on three gradually refined grids with resolutions of
%$N=16,32$ and $64$.
%According to the observation of $\ell_{1}$, $\ell_{2}$ and $\ell_{\infty}$ relative errors,  the second-order,  third-order and fourth-order convergence rate are obtained for ${\color{red} \mathbb{P}^{1}}, {\color{red} \mathbb{P}^{2}}$ and ${\color{red} \mathbb{P}^{3}}$-based RKDLEG methods.

 Fig.~\ref{fig:steadycon3} plots the  relative conservation errors
 of the ${ \mathbb{P}^{3}}$-based RKDLEG method with $N=64$ in
 the total mass, energy, and potential enstrophy \cite{Ullrich:2010},
  defined by
\[
\mathcal{M}(t)=\int_{S}hds, \ \mathcal{E}(t)=\int_{S}\left(\frac{1}{2}h\left(u_{s}^{2}+v_{s}^{2}\right)+\frac{1}{2}g\left(\left(h+b\right)^{2}-b^{2}\right)\right)ds
,  \
\mathcal{P}(t)=\int_{S}\frac{\left(\varsigma+f\right)^{2}}{2h}ds,
\]
where $\varsigma=\frac{1}{\Lambda}\left(\frac{\partial \hat{v}}{\partial x}-\frac{\partial \hat{u}}{\partial y}\right)$ denotes the relative vorticity, and $(\hat{u},\hat{v})$ are given in Eq. \eqref{veclocity}.
% the relative error of total mass is defined by $ (\mathcal{M}(t)-\mathcal{M}(0))/\mathcal{M}(0)$.
The results show that  the error of total mass $ (\mathcal{M}(t)-\mathcal{M}(0))/\mathcal{M}(0)$ is very close to the machine (or round-off) precision,
 the  error of  total energy is very small and oscillatory decreasing,
 while the error of potential enstrophy is also small but monotonically increasing.
%$N=64$
%

\end{example}

\begin{table}[htbp]\small
  \centering
    \caption{\small Example \ref{steady}: The relative errors in the height $h$ at $t=3$ days and
    convergence rates by ${\mathbb{P}^{1}}$-based RKDLEG method.
    }
\begin{tabular}{|c||c|c||c|c||c|c|}
  \hline
{$N$}
% &\multicolumn{2}{c|}{$l_{1}$}&\multicolumn{2}{c|}{$l_{2}$}&\multicolumn{2}{c|}{$l_{\infty}$}\\
% \cline{2-7}
 &$l_{1}$-error &order & $l_{2}$-error &order & $l_{\infty}$-error &order \\
 \hline
16&1.23e-03& --   & 1.56e-03 &   -- & 1.07e-02 &--\\
32&2.67e-04&2.2019& 3.54e-04 &2.1392& 3.05e-03 &1.8046\\
64&6.09e-05&2.1326& 8.42e-05 &2.0707& 1.07e-03  &1.8931\\
\hline
\end{tabular}\label{tab:steady1}

  \centering
    \caption{\small Same as Table \ref{tab:steady1} except for ${ \mathbb{P}^{2}}$-based RKDLEG method.
  }
\begin{tabular}{|c||c|c||c|c||c|c|}
  \hline
%\multirow{3}{2pt}{$N$}
% &\multicolumn{2}{c|}{$l_{1}$}&\multicolumn{2}{c|}{$l_{2}$}&\multicolumn{2}{c|}{$l_{\infty}$}\\
% \cline{2-7}
% &error &order & error &order & error &order \\
 {$N$}
 % &\multicolumn{2}{c|}{$l_{1}$}&\multicolumn{2}{c|}{$l_{2}$}&\multicolumn{2}{c|}{$l_{\infty}$}\\
 % \cline{2-7}
 &$l_{1}$-error &order & $l_{2}$-error &order & $l_{\infty}$-error &order \\
 \hline
16&2.83e-05& --   & 4.28e-05 &   -- & 4.11e-04 &--\\
32&3.47e-06&3.0260& 5.32e-06 &3.0059& 5.64e-05 &2.8648\\
64&4.31e-07&3.0070& 6.65e-07 &3.0018& 8.24e-06 &2.7759\\
\hline
\end{tabular}\label{tab:steady2}

  \centering
    \caption{\small Same as Table \ref{tab:steady1} except for ${ \mathbb{P}^{3}}$-based RKDLEG method.
  }
\begin{tabular}{|c||c|c||c|c||c|c|}
  \hline
%\multirow{3}{2pt}{$N$}
 %&\multicolumn{2}{c|}{$l_{1}$}&\multicolumn{2}{c|}{$l_{2}$}&\multicolumn{2}{c|}{$l_{\infty}$} %\\
 %\cline{2-7}
% &error &order & error &order & error &order \\
 {$N$}
 % &\multicolumn{2}{c|}{$l_{1}$}&\multicolumn{2}{c|}{$l_{2}$}&\multicolumn{2}{c|}{$l_{\infty}$}\\
 % \cline{2-7}
 &$l_{1}$-error &order & $l_{2}$-error &order & $l_{\infty}$-error &order \\
 \hline
16&9.82e-07& --   & 1.61e-06 &   -- & 4.68e-05 &--\\
32&6.04e-08&4.0222& 9.99e-08 &4.0060& 3.83e-06 &3.6088\\
64&3.82e-09&3.9843& 6.55e-09 &3.9318& 7.07e-07 &2.4384\\
\hline
\end{tabular}\label{tab:steady3}
\end{table}

\begin{figure}[htbp]
	\centering
	\begin{minipage}{5cm}
		\includegraphics[width=1\textwidth]{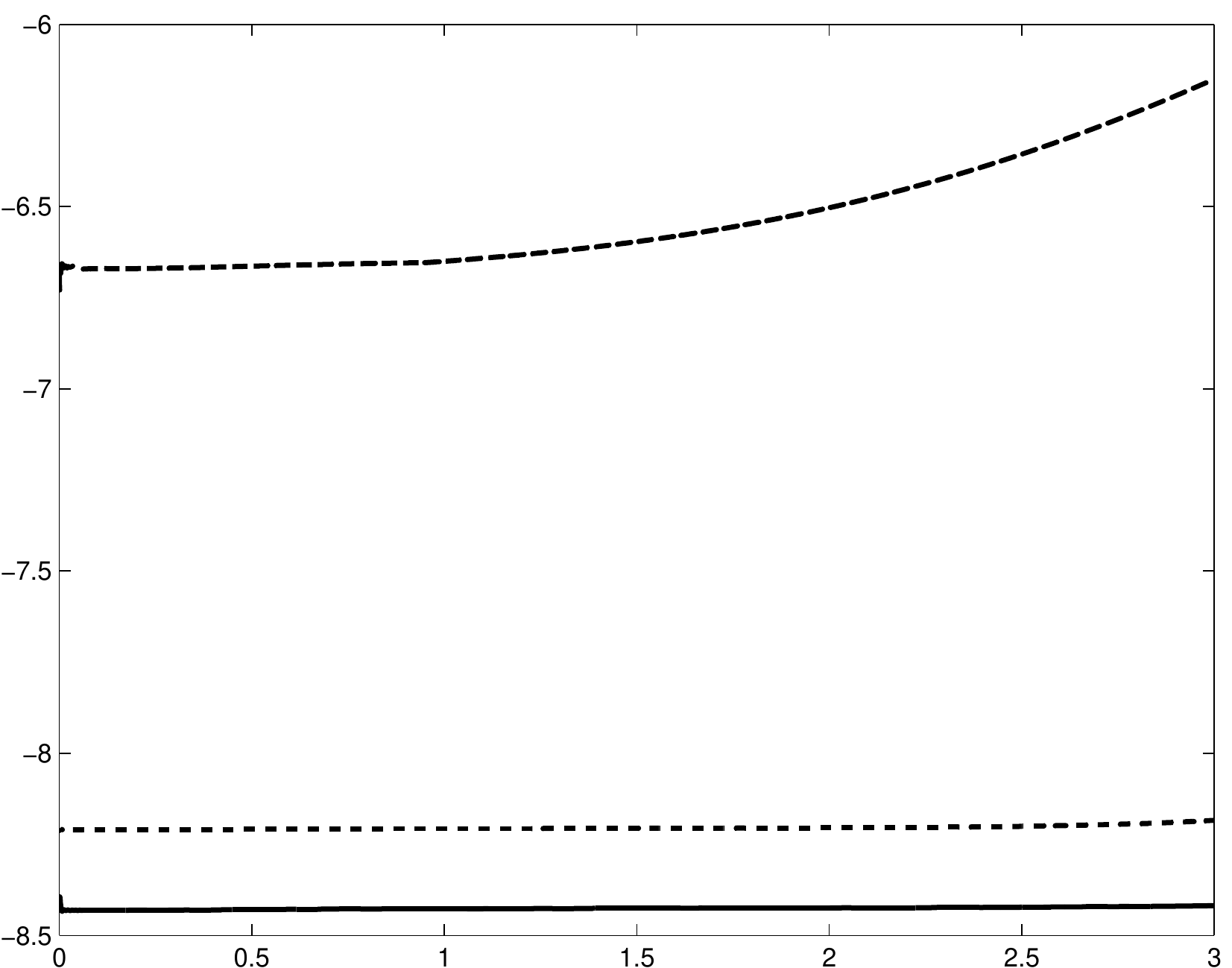}
	\end{minipage}
	\begin{minipage}{5cm}
		\includegraphics[width=1\textwidth]{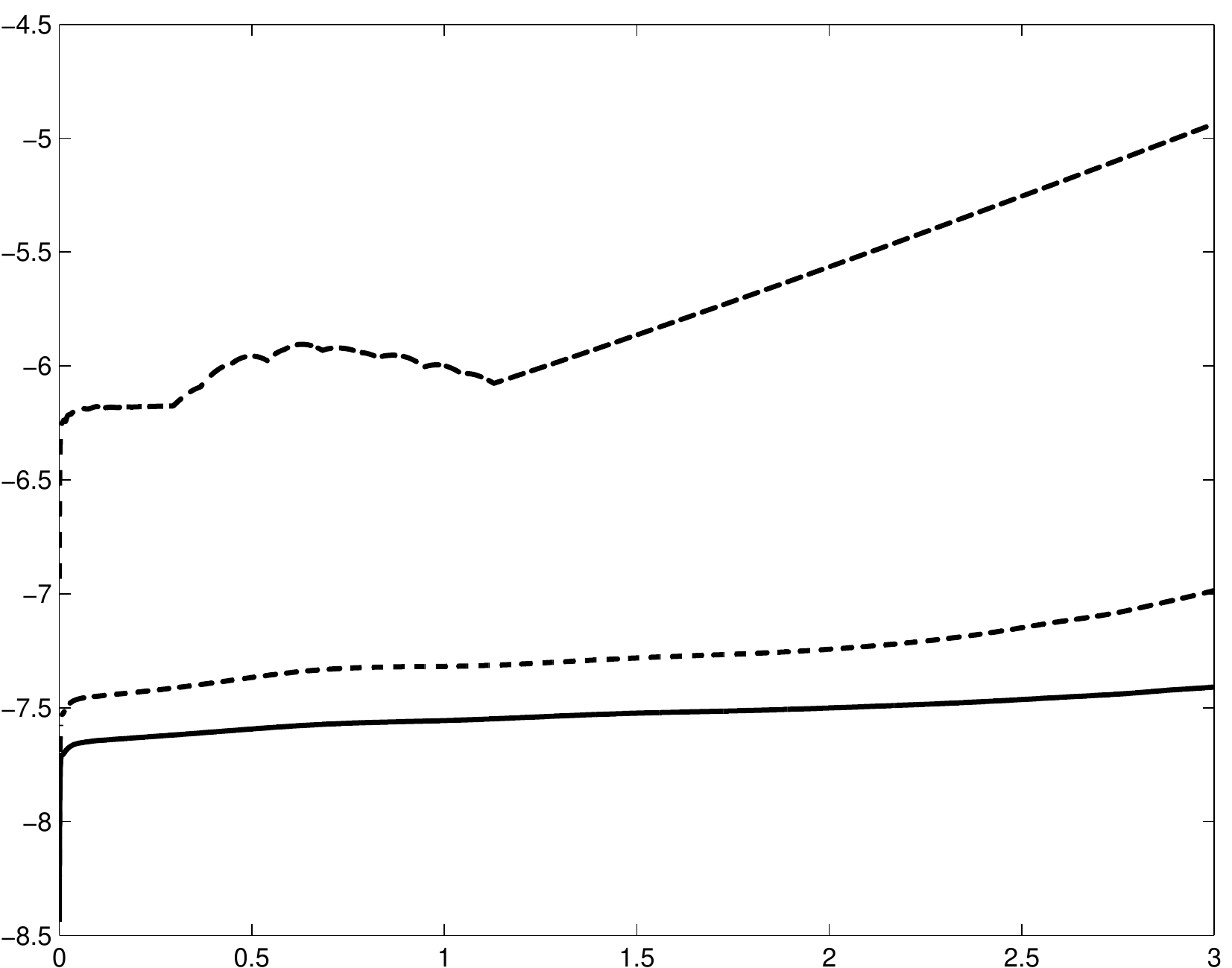}
	\end{minipage}
	\begin{minipage}{5cm}
		\includegraphics[width=1\textwidth]{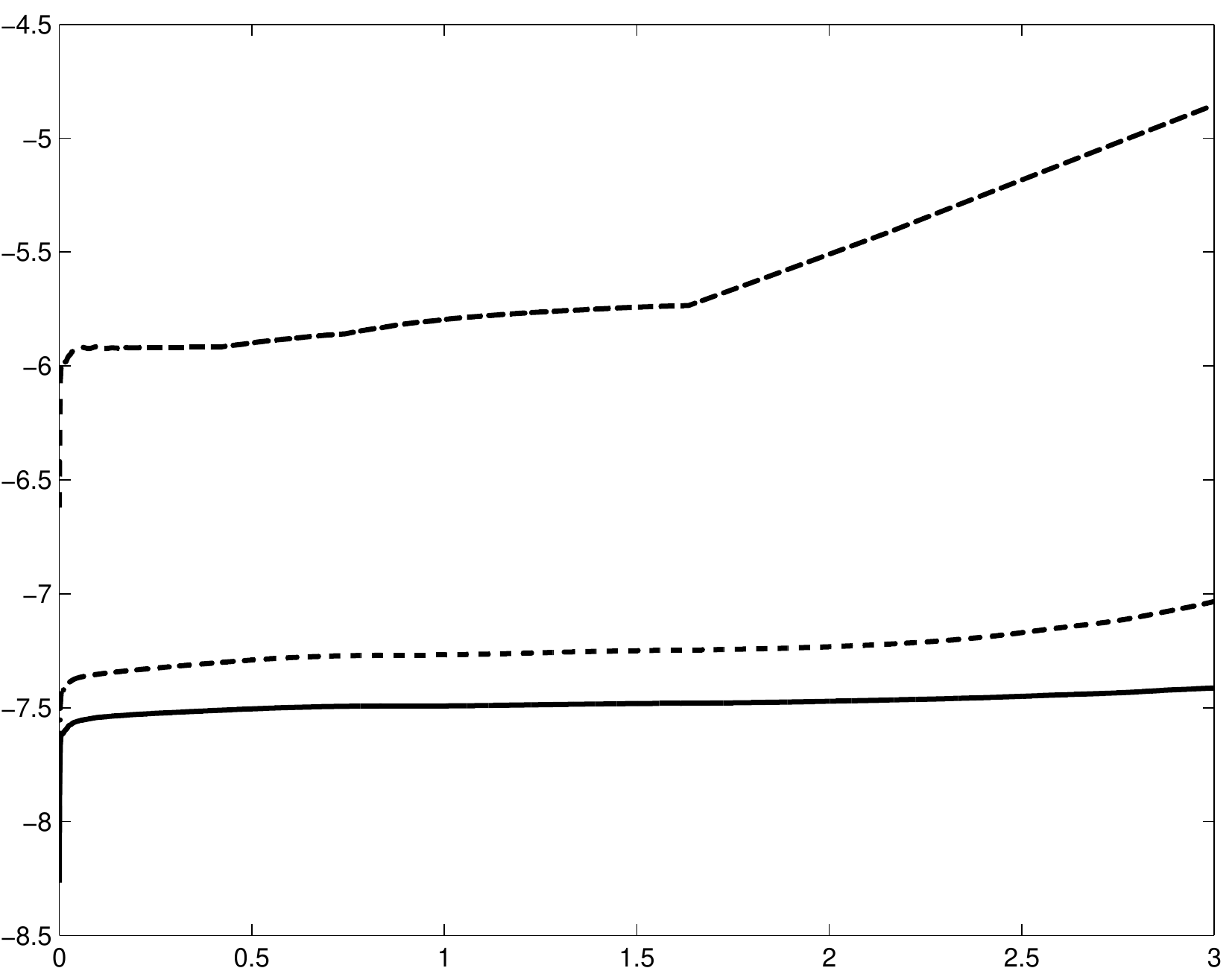}
	\end{minipage}
	\caption{\small Example \ref{steady}. The time evolution of the log of  relative errors  to base 10 in $h$, $u$, and $v$ (from left to right) obtained by using the ${\mathbb{P}^{3}}$-based RKDLEG method with $N=64$, where the solid, dashed, and dotted lines denote  the $l_{1}$-, $l_{2}$-, and $l_{\infty}$-errors, respectively.}
	\label{fig:steadyerror3}
\end{figure}

\begin{figure}[htbp]
  \centering
    \begin{minipage}{5cm}
  \includegraphics[width=1\textwidth]{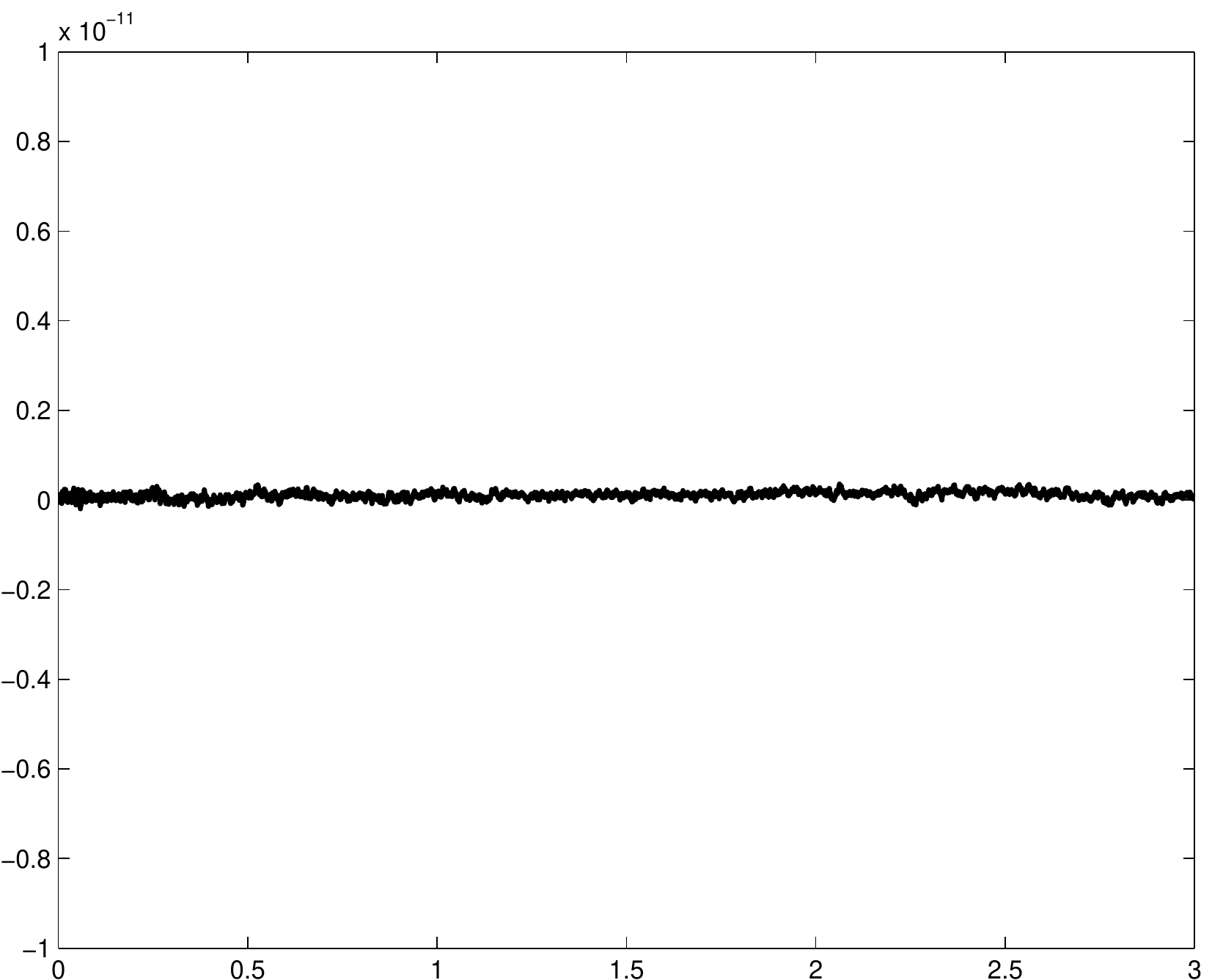}
  \end{minipage}
  \begin{minipage}{5cm}
  \includegraphics[width=1\textwidth]{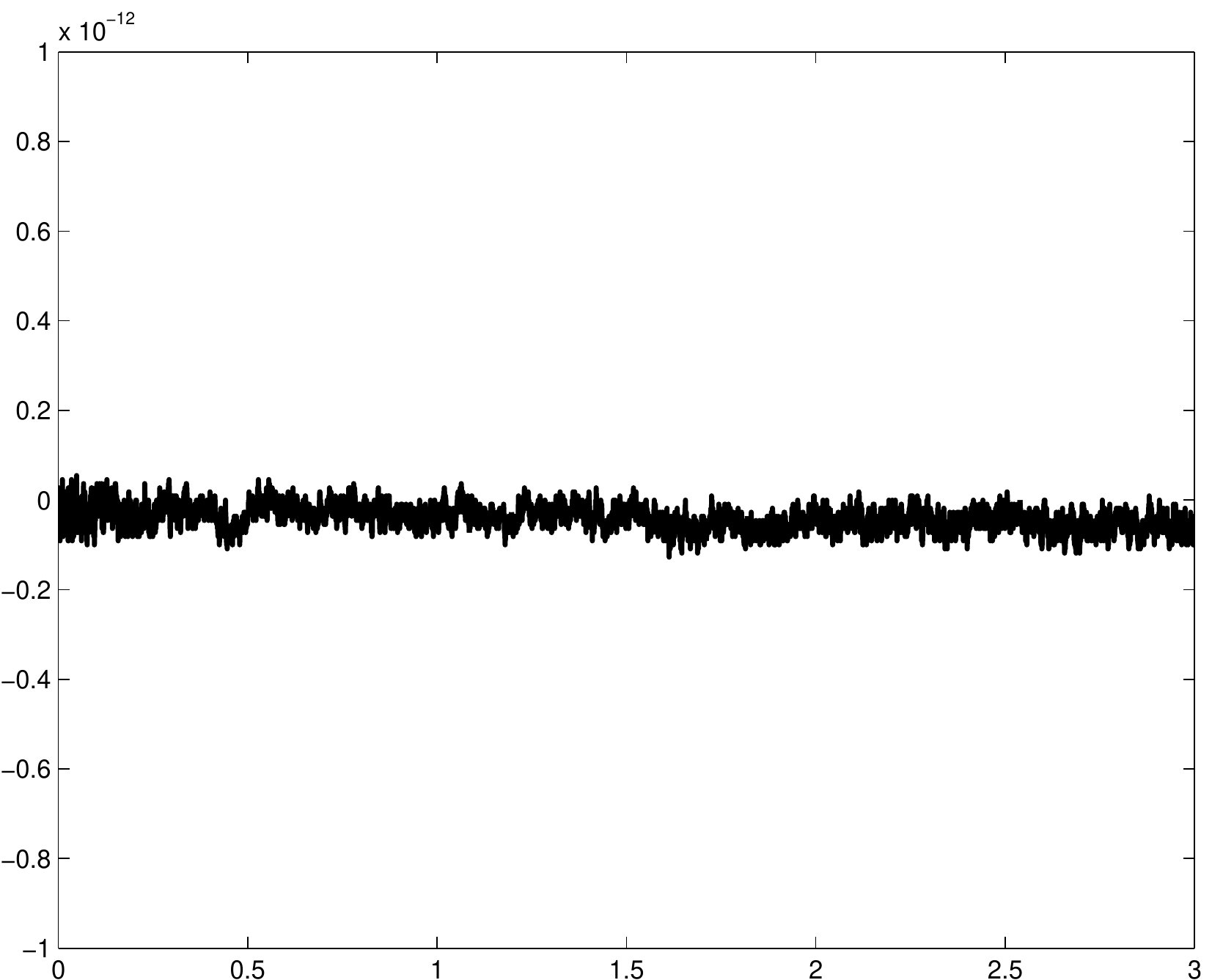}
  \end{minipage}
  \begin{minipage}{5cm}
  \includegraphics[width=1\textwidth]{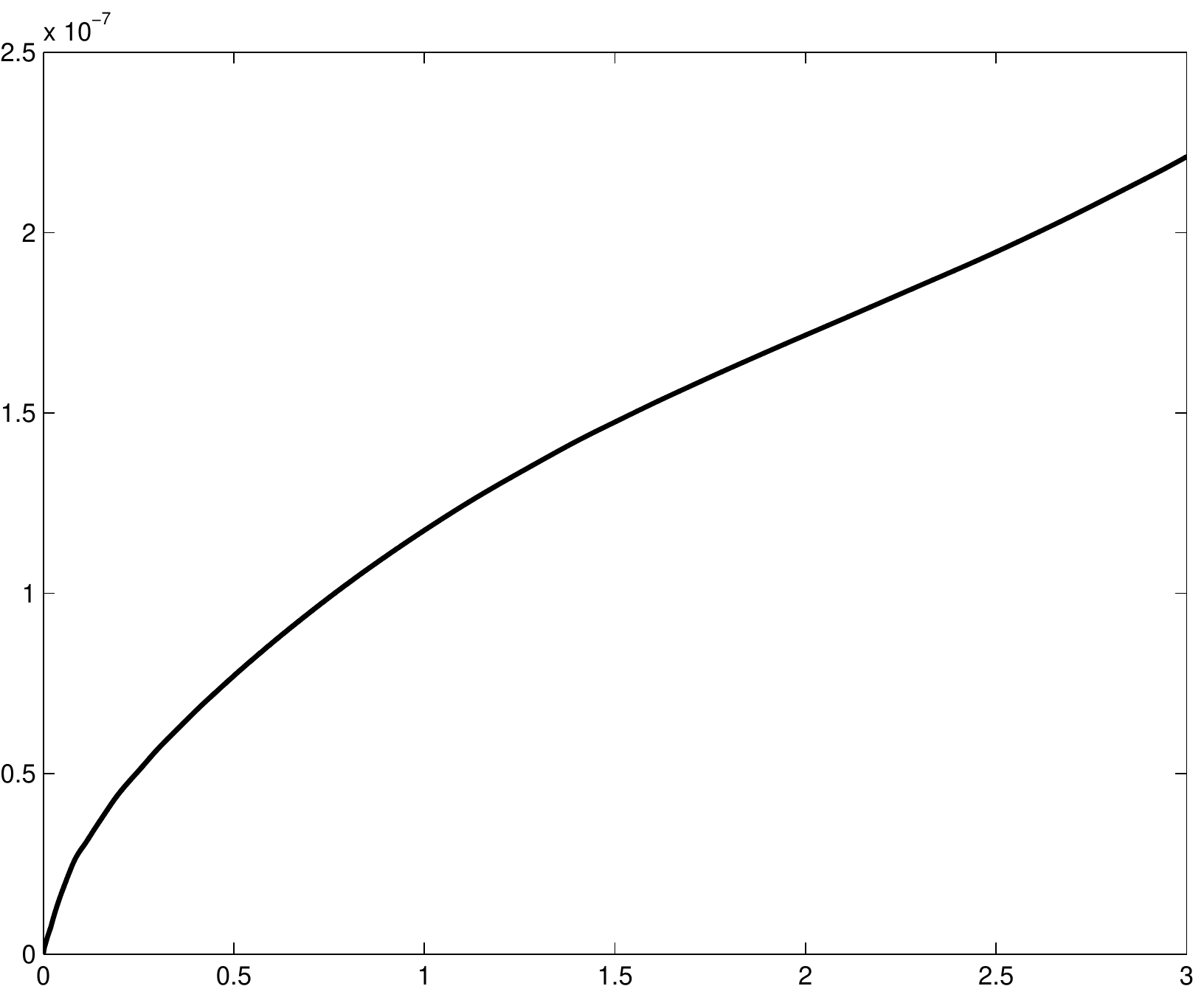}
  \end{minipage}
  \caption{\small Example \ref{steady}: The time evolution of the relative conservation errors of total mass, energy, and potential enstrophy (from left to right) obtained by using the ${\mathbb{P}^{3}}$-based RKDLEG method with $N=64$.}
  \label{fig:steadycon3}
\end{figure}

% % % % % % % % % % % % % % % % % % % % % % % % % % % % % % % % % % % % % % % % % % % % %
\begin{example}[Time dependent zonal flow]\label{time}\rm
	This example is about  a time dependent zonal flow and challenging to evaluate the  numerical methods,
	see \cite{Chen:2014,Pudykiewicz: 2011}.
The analytical solutions to this problem   \cite{Lauter:2005} can be given by
%	The height field depends on time
%$t$ as
\begin{align*}
h\left(\xi,\eta,t\right)=&-\frac{1}{2g}\left[u_{0}\sin\alpha\cos\theta
\left(-\cos\xi\cos\left(\Omega t\right)+\sin\xi\sin\left(\Omega t\right)+\cos\alpha\sin\eta\right)\right.\\
&+\left.R\Omega\sin\eta\right]^{2}+\frac{1}{2g}\left(R\Omega\sin\eta\right)^{2}+g^{-1}k_{1}-b(\eta),\\
u_{s}\left(\xi,\eta,t\right)=&u_{0}\left[ \sin\alpha\sin\eta\left(\cos\xi\cos\left(\Omega t\right)
-\sin\xi\sin\left(\Omega t\right)\right)+\cos\alpha\cos\eta\right],\\
v_{s}\left(\xi,\eta,t\right)=&-u_{0}\left[\sin\alpha\left(\sin\xi\cos\left(\Omega t\right)
+\cos\xi\sin\left(\Omega t\right)\right)\right],
\end{align*}
where $b(\eta)=\frac{1}{2g}\left(R\Omega\sin\eta\right)^{2}+g^{-1}k_{2}$ is the {{height}} of bottom mountain,
$u_{0}=\frac{2\pi R}{12}\mbox{ day}^{-1}$, $k_{1}=133681$ m, $k_{2}=0$ m, and $\alpha$ is taken as $\frac{\pi}{4}$.
Fig. \ref{fig:time}  gives the contour plot of  height $h$ at $t=5$ days  obtained by using the ${ \mathbb{P}^{2}}$-based RKDLEG method with $N=64$.

Fig.~\ref{fig:errorcompare} gives a comparison of the ${\mathbb{P}^{1}}$-based  RKDLEG method
to the second-order accurate finite volume {{LEG}} (abbr. FVLEG) method.
Those error plots show that
 the errors of the ${ \mathbb{P}^{1}}$-based RKDLEG method grows more slowly and much smaller %over a long time period
 than the FVLEG method.  The relative errors in $h,u,v$  of the ${ \mathbb{P}^{2}}$-based RKDLEG method  shown in Fig. \ref{fig:timeerror2}
 is similar to  the ${ \mathbb{P}^{1}}$-based RKDLEG method.
% display respectively the  time evolutions of the log of relative errors to base 10
Fig. \ref{fig:timecon2}
plots the relative conservation errors in the total mass,
energy, and potential enstrophy  obtained by the ${ \mathbb{P}^{2}}$-based RKDLEG method with $N=64$.
%Fig. \ref{fig:timecon2} shows the time evolution of the relative conservation errors in the total mass,
%energy, and potential enstrophy (from left to right) by using the ${ \mathbb{P}^{2}}$-based RKDLEG method with $N=64$
The results show  the total mass is numerically conservative, the total energy is  decreasing with an approximate
slope of $-7.2\times 10^{-11}$,
and the error in the total potential enstrophy is about $O(10^{-6})$.
%  about the convergence rates, relative errors, and relative conservation errors.
%Since we can see relative errors and the convergence rates in the figures, we don't show the tables of them again.
%Fig.~\ref{fig:time} gives

\end{example}
\begin{figure}[htbp]
	\centering
	\includegraphics[width=10cm,height=5cm]{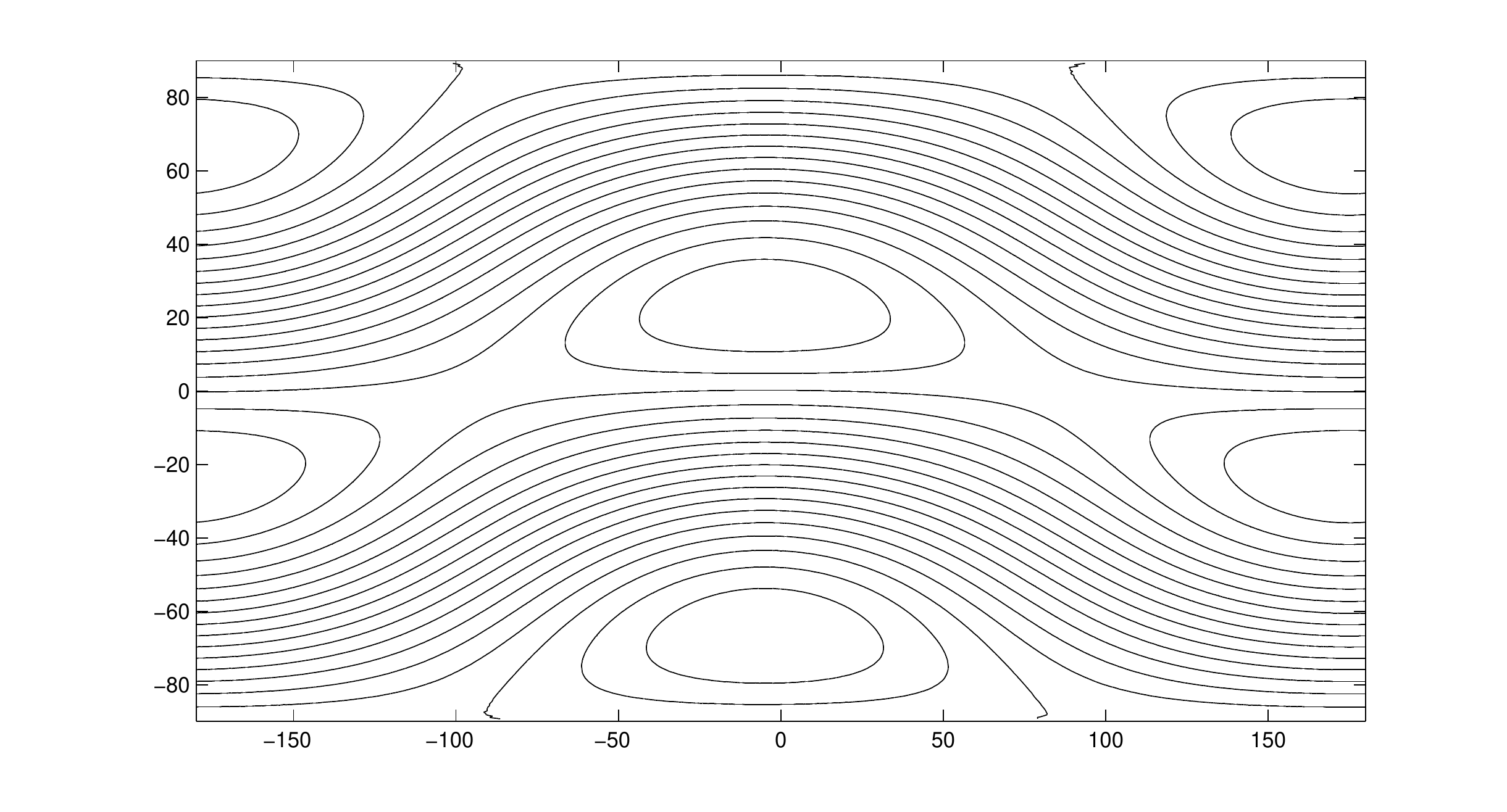}
	\caption{\small Example \ref{time}: The height  $h(\xi,\eta,t)$ at $t=5$ days obtained by using the ${ \mathbb{P}^{2}}$-based RKDLEG method with $N=64$. Contour lines are equally spaced {{from}} $12000$ m to $13800$ m
		with a stepsize of $100$ m.}
	\label{fig:time}
\end{figure}

\begin{figure}[htbp]
	\centering
	\begin{minipage}{5cm}
		\includegraphics[width=1\textwidth]{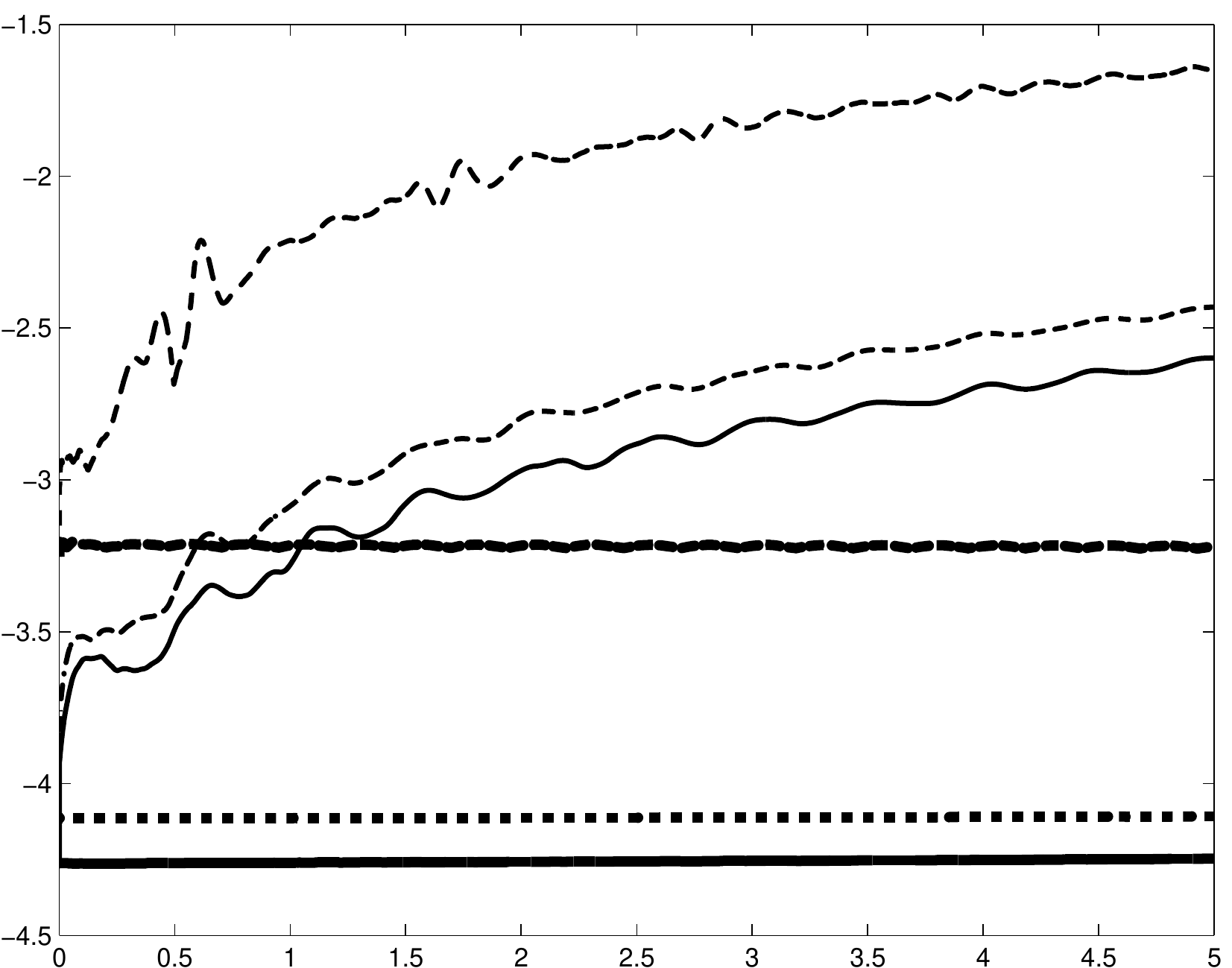}
	\end{minipage}
	\begin{minipage}{5cm}
		\includegraphics[width=1\textwidth]{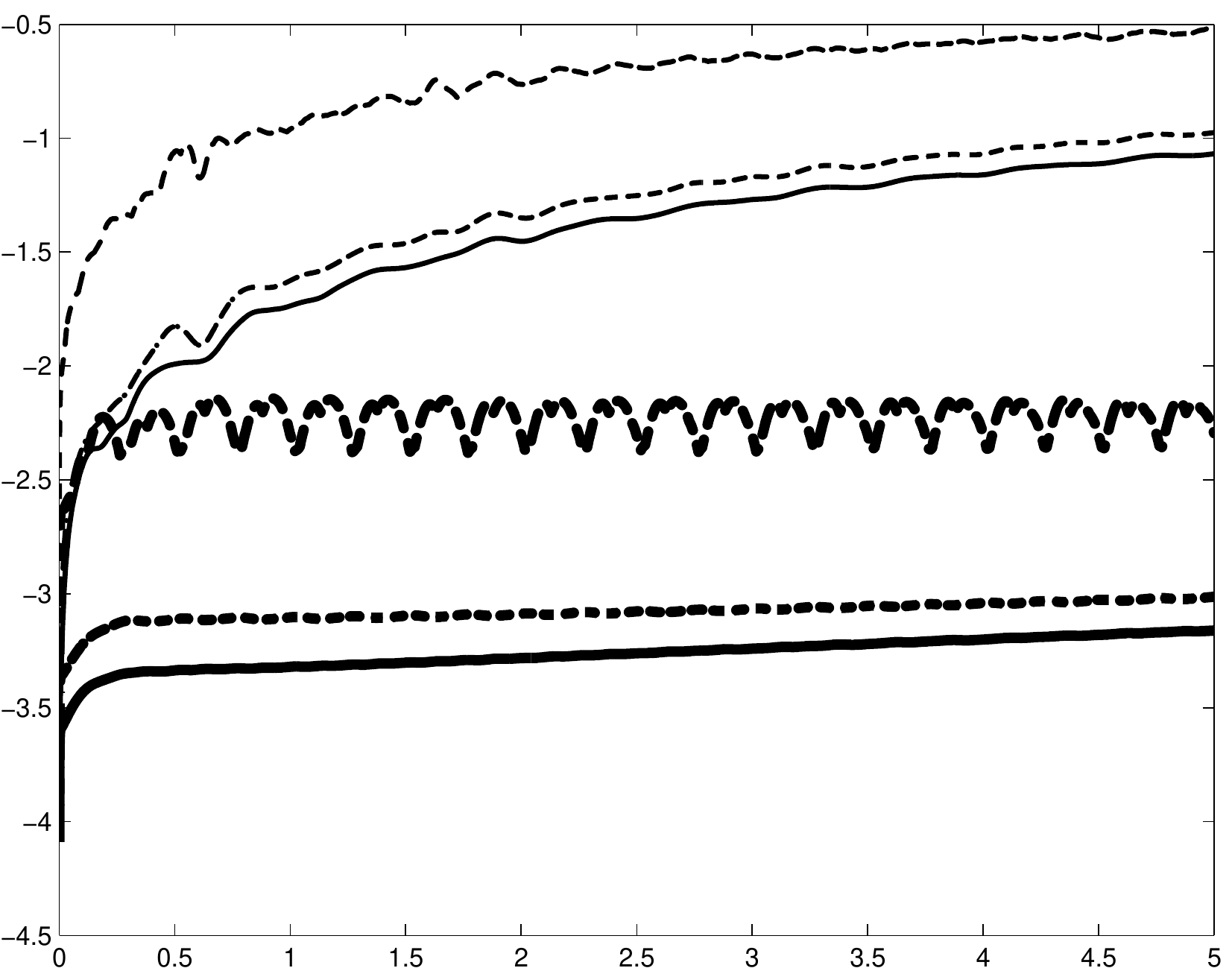}
	\end{minipage}
	\begin{minipage}{5cm}
		\includegraphics[width=1\textwidth]{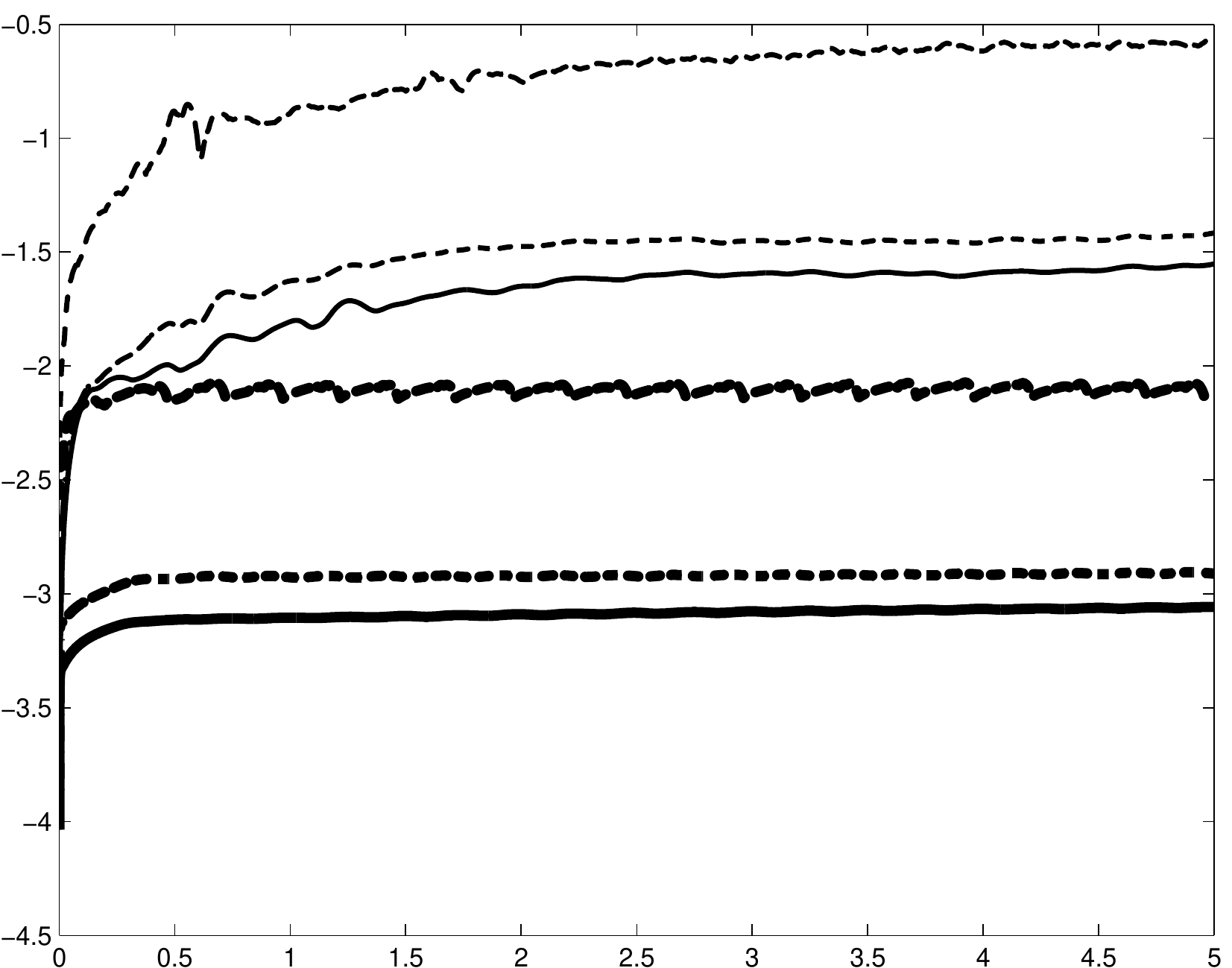}
	\end{minipage}
	\caption{\small Example \ref{time}: The time evolution of the log of relative errors to base 10 in $h$, $u$, and $v$ (from left to right) obtained by using the ${ \mathbb{P}^{1}}$-based RKDLEG method (wide lines) and FVLEG method (thin lines) with $N=16,32,64$ (resp. dashed, dotted, and solid lines), respectively.}
	\label{fig:errorcompare}
\end{figure}

\begin{figure}[htbp]
	\centering
	\centering
	\begin{minipage}{5cm}
		\includegraphics[width=1\textwidth]{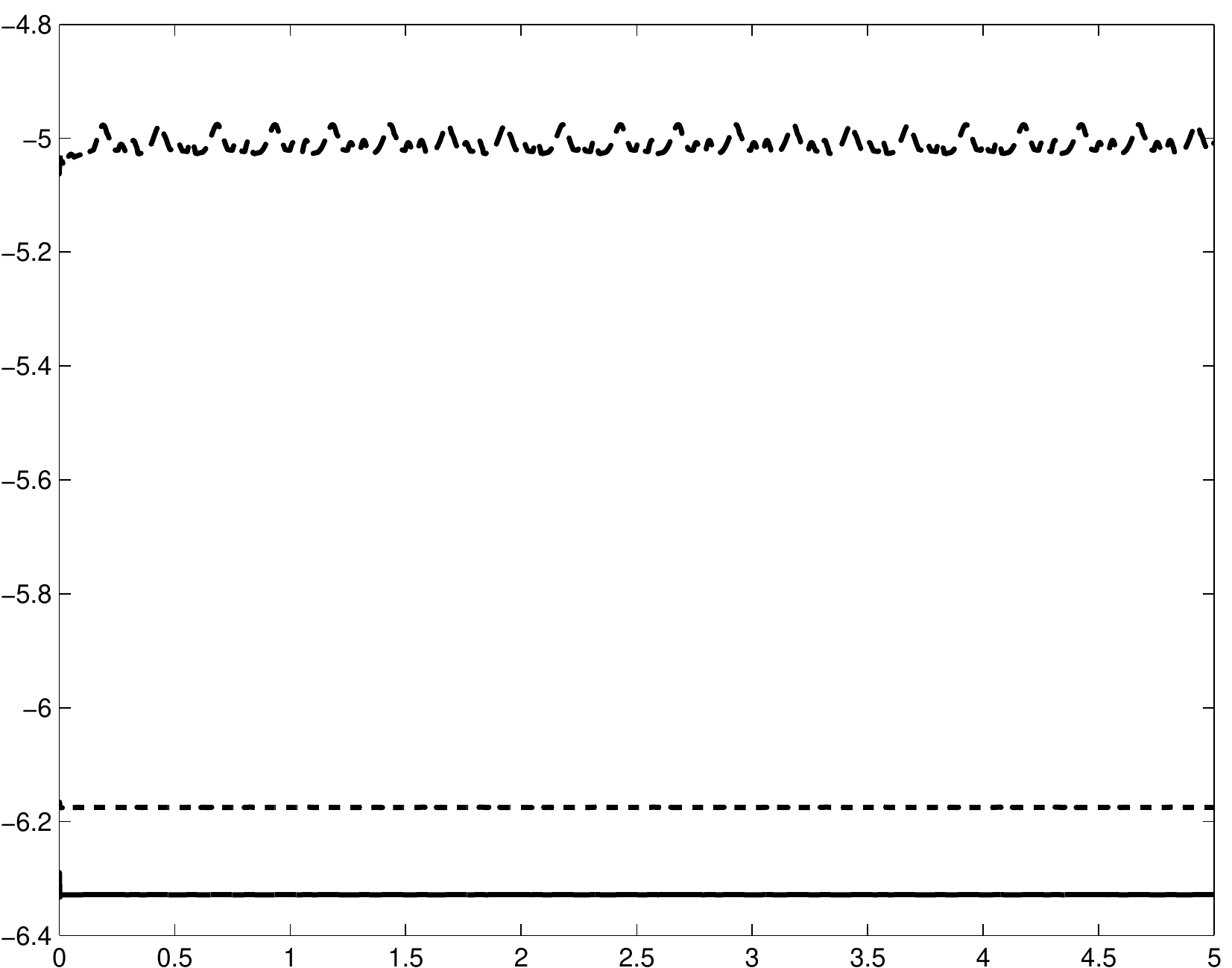}
	\end{minipage}
	\begin{minipage}{5cm}
		\includegraphics[width=1\textwidth]{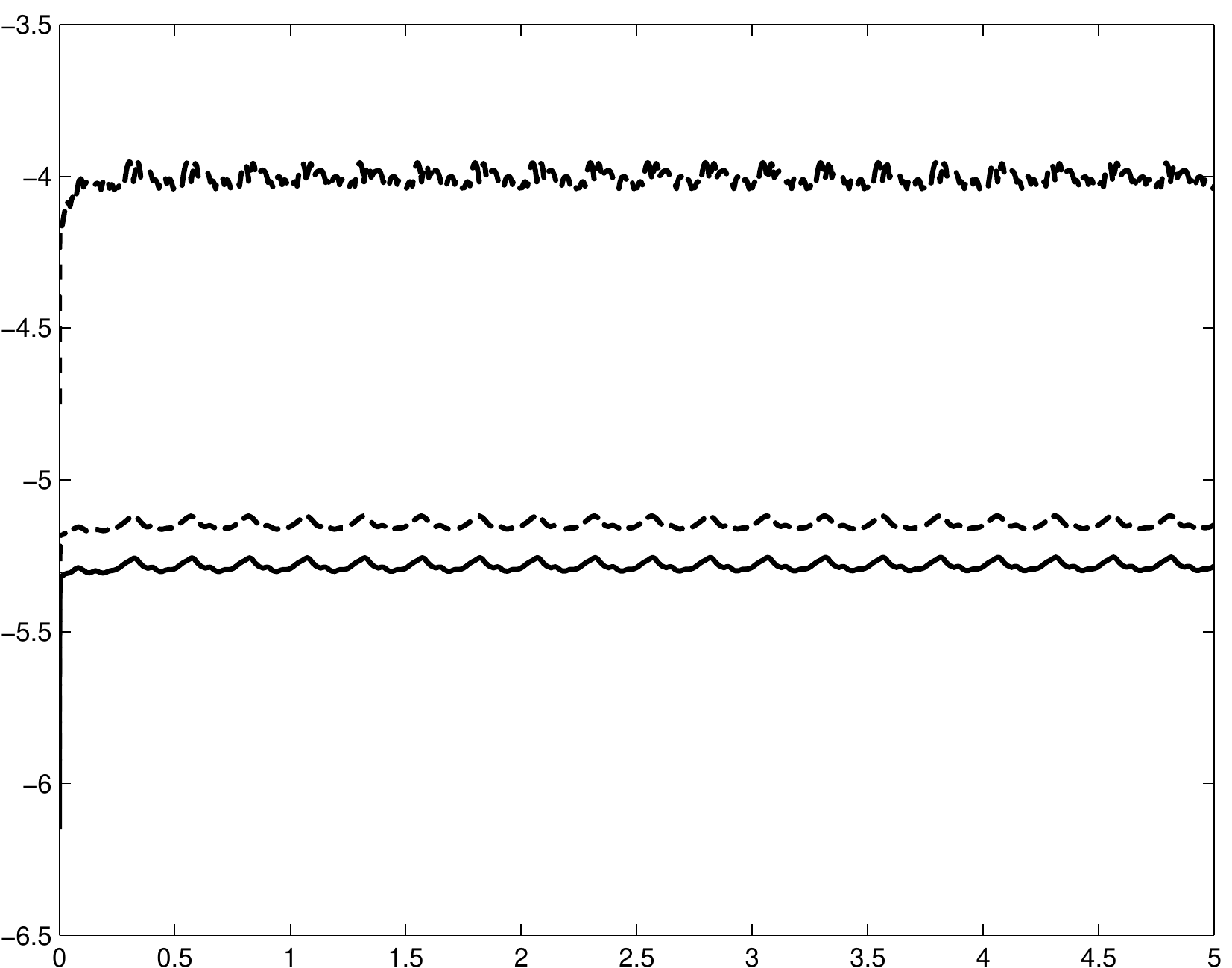}
	\end{minipage}
	\begin{minipage}{5cm}
		\includegraphics[width=1\textwidth]{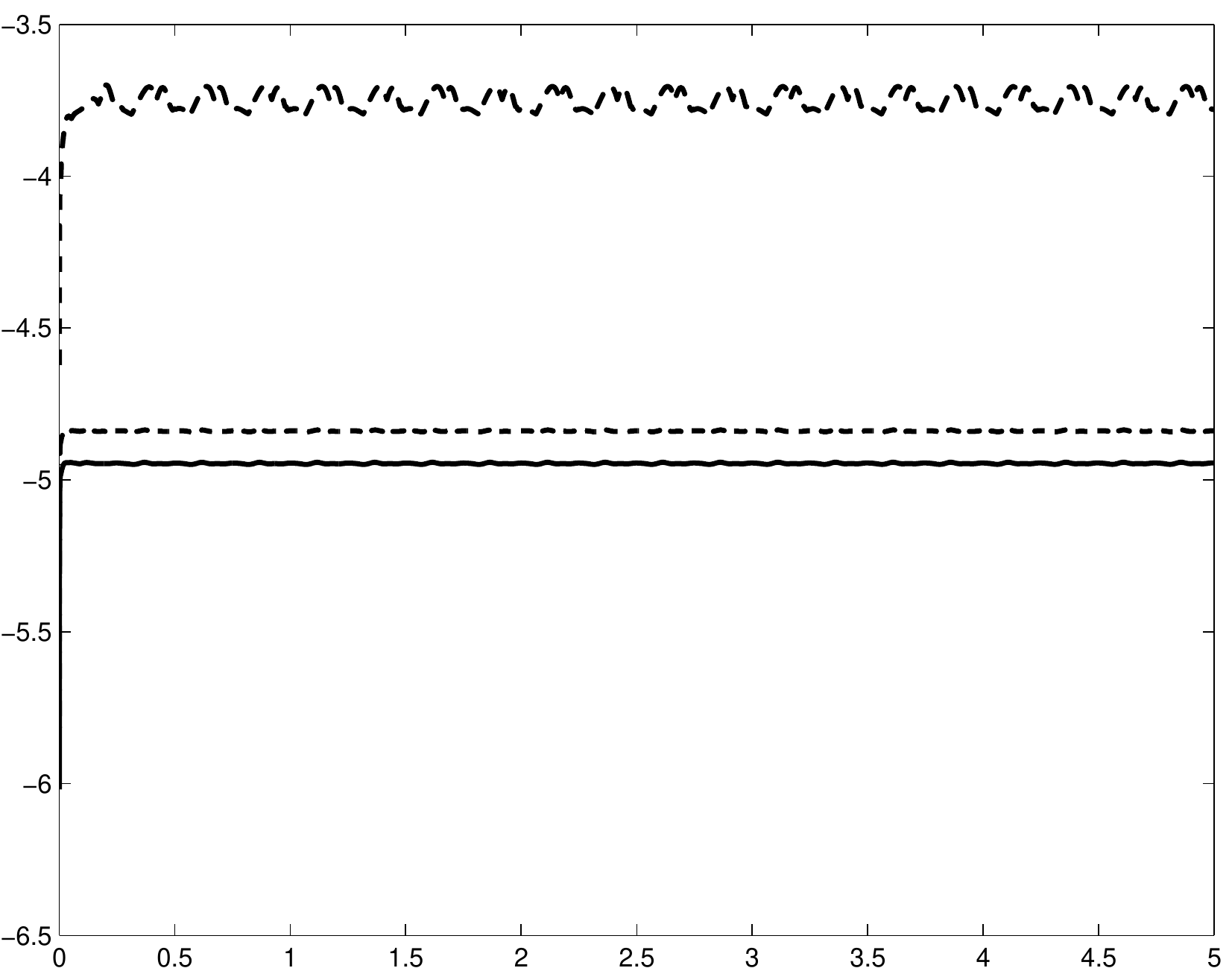}
	\end{minipage}
	\caption{\small Example \ref{time}: {Same as Fig. \ref{fig:steadyerror3} except for example \ref{time} by using $\mathbb{P}^{2}$-based RKDLEG method.}}%The time evolution of the log of relative errors to base 10
%		in $h$, $u$, and $v$ (from left to right) obtained by using the ${\color{red} \mathbb{P}^{2}}$-based RKDLEG method with $N=64$, where the solid, dashed, and dotted lines denote  the $l_{1}$-, $l_{2}$-, and $l_{\infty}$-errors, respectively.}
	\label{fig:timeerror2}
\end{figure}

\begin{figure}[htbp]
  \centering
    \begin{minipage}{5cm}
  \includegraphics[width=1\textwidth]{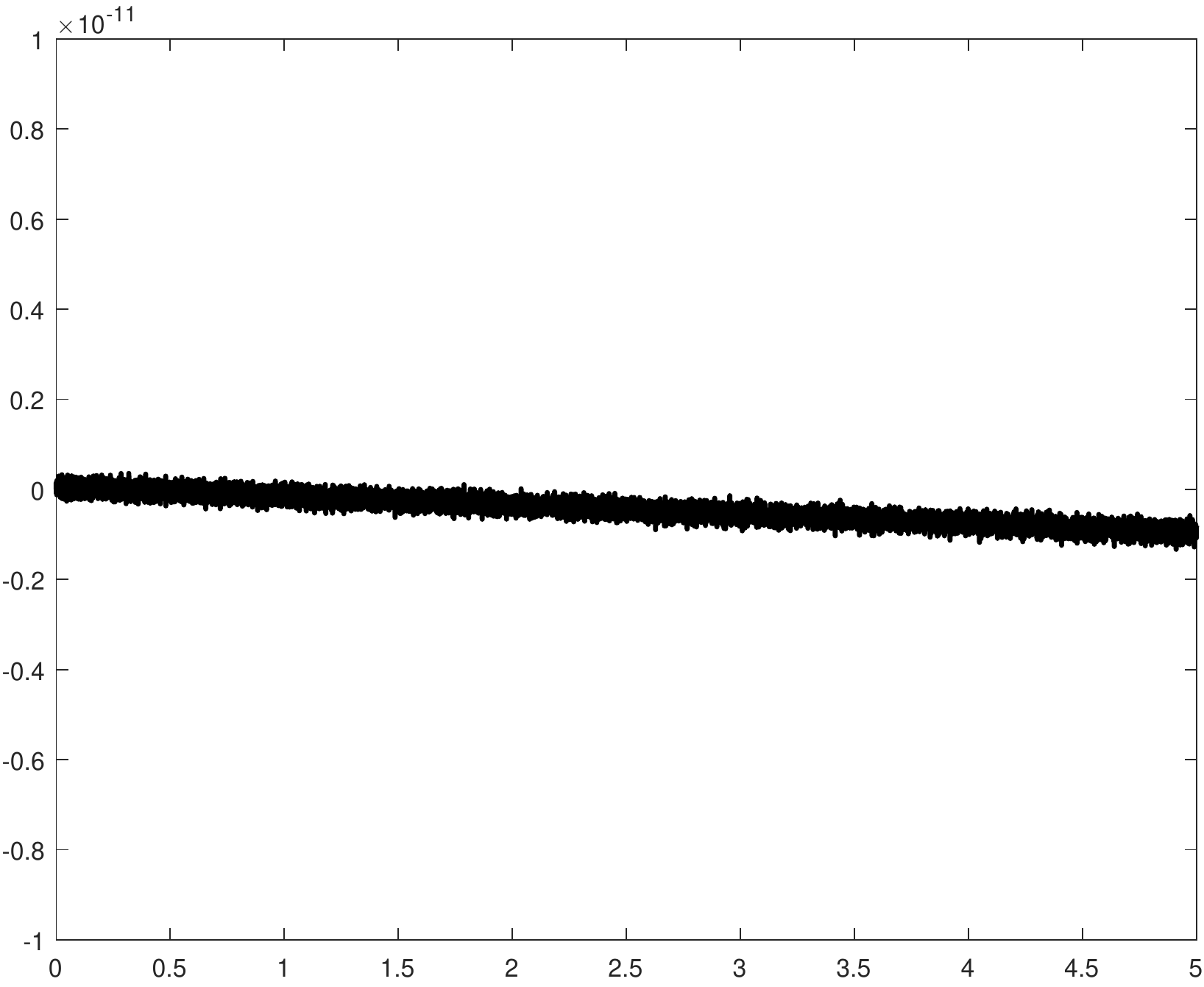}
  \end{minipage}
  \begin{minipage}{5cm}
  \includegraphics[width=1\textwidth]{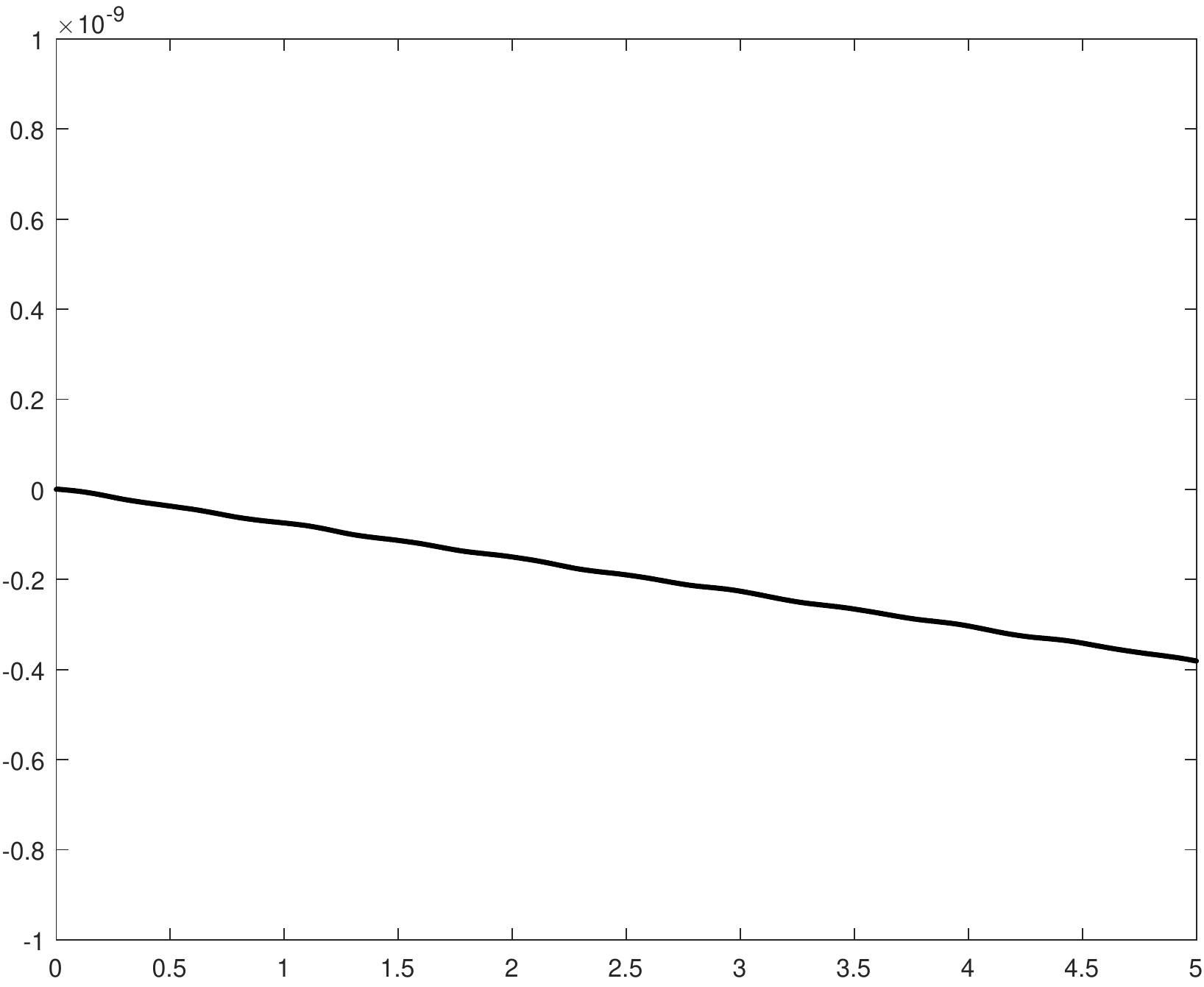}
  \end{minipage}
  \begin{minipage}{5cm}
  \includegraphics[width=1\textwidth]{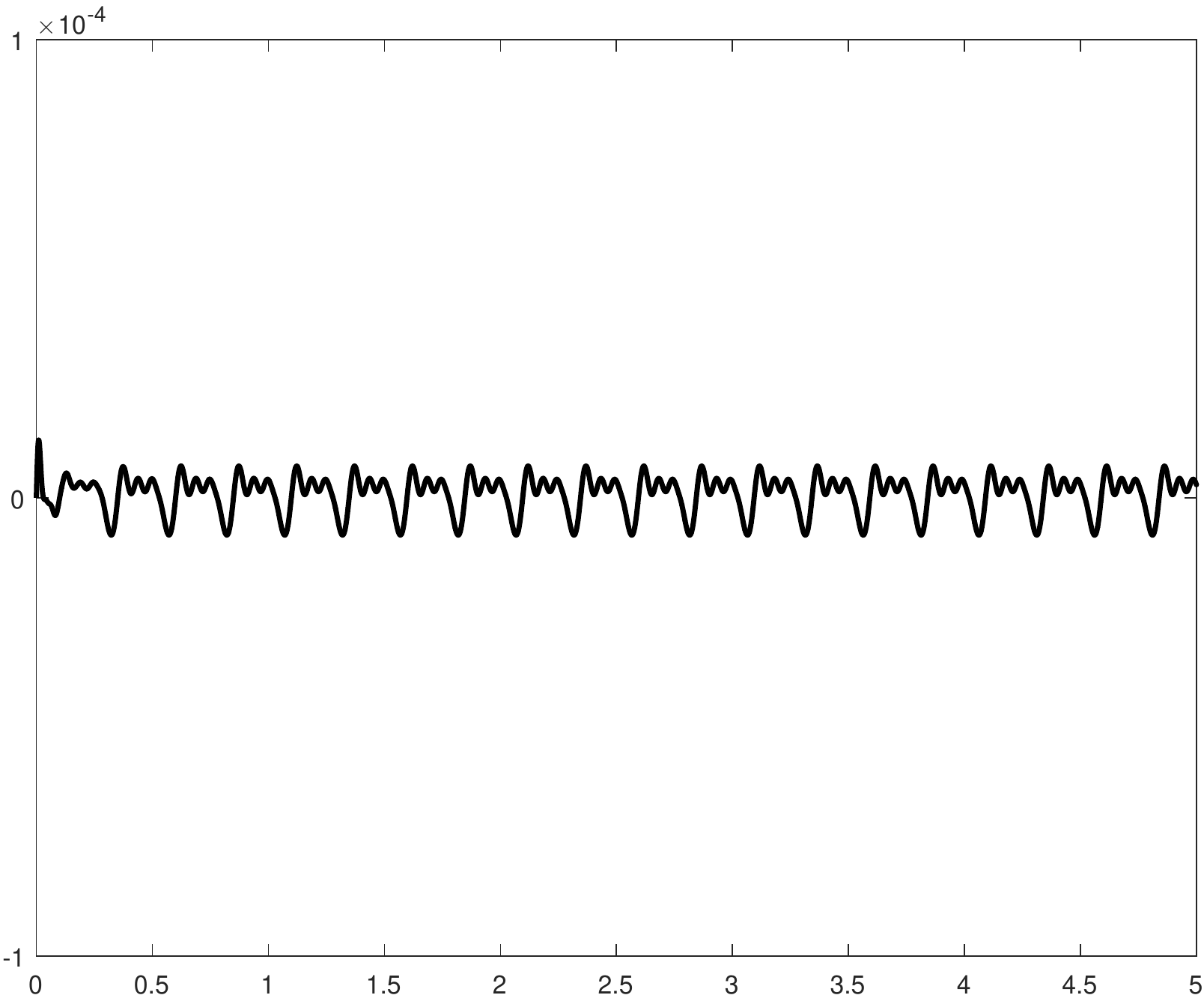}
  \end{minipage}
  \caption{\small Example \ref{time}: {Same as Fig. \ref{fig:steadycon3} except for example \ref{time} by using $\mathbb{P}^{2}$-based RKDLEG method.} }
  \label{fig:timecon2}
  \end{figure}

% % % % % % % % % % % % % % % % % % % % % % % % % % % % % % % % % % % % % % % % % % % %
\begin{figure}[htbp]
   \centering
  \includegraphics[width=11cm,height=6cm]{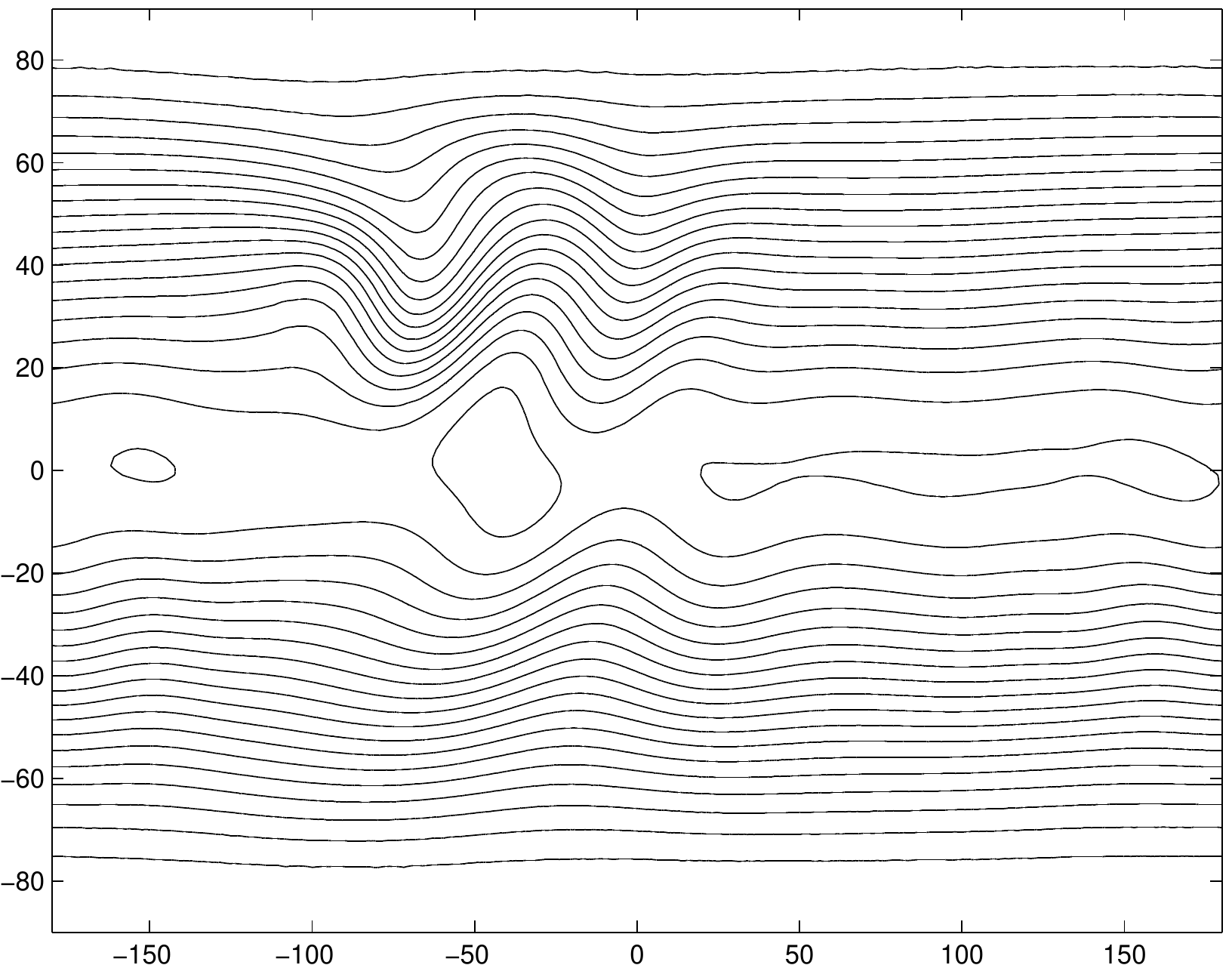}

  \centering
  \includegraphics[width=11cm,height=6cm]{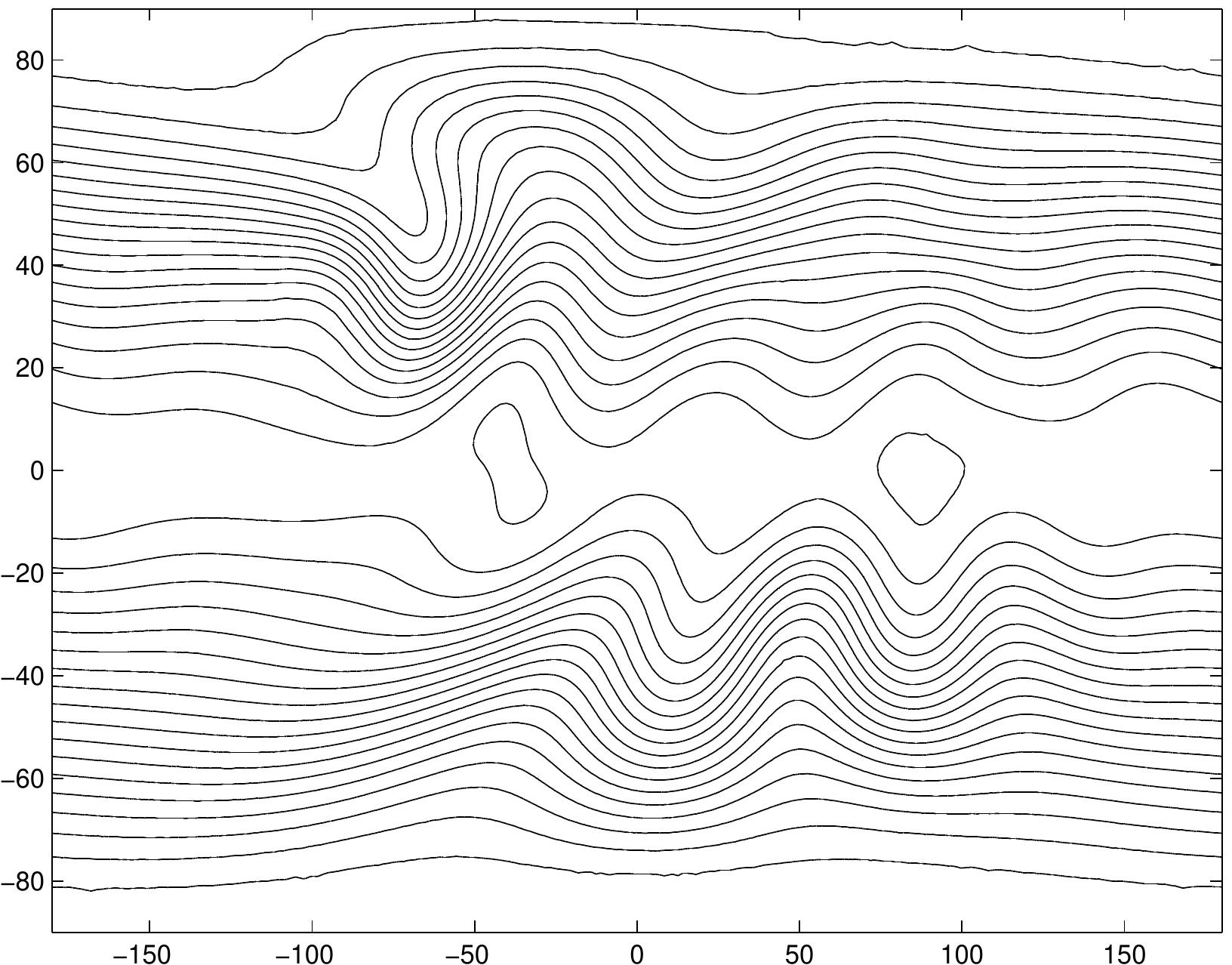}

    \centering
  \includegraphics[width=11cm,height=6cm]{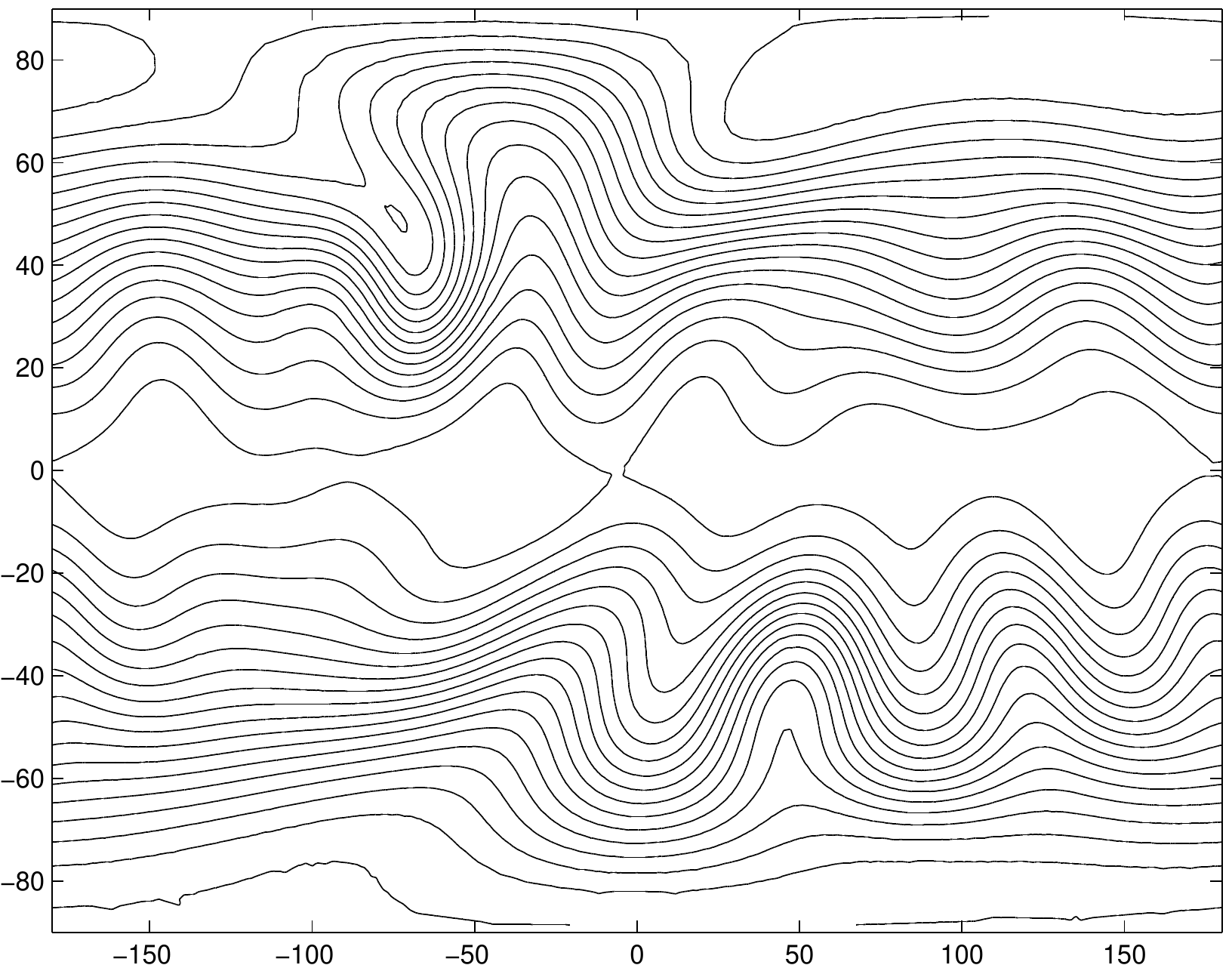}

  \caption{\small Example \ref{zonal}: The   heights $h(\xi,\eta,t)$ at $t = 5,10$ and $15$ days (from top to bottom)
  obtained by using	the ${ \mathbb{P}^{3}}$-based RKDLEG method with $N = 32$. Contour lines are equally spaced from $5050$ m to $5950$ m with a stepsize of $50$ m.}\label{fig:zonal}
\end{figure}

% % % % % % % % % % % % % % % % % % % % % % % % % % % % % % % % % % % % % % % % % % % % % %
\begin{example}[Zonal flow over a bottom mountain]\label{zonal}\rm
	% in order	to further evaluate  the numerical methods in solving the problem with a bottom mountain
The third example considers Williamson's test case 5 \cite{Williamson:1992}, in which the initial height $h$
and velocity vector $(u_s, v_s)$ are given in
Eq. \eqref{EQ-Example4.1} with
 $h_{0}=5960$ m, $\alpha=0$, and $u_{0}=20\mbox{ s}^{-1}$.
The  bottom mountain is
centered at $\left(\xi_{c},\eta_{c}\right)=\left(-\frac{\pi}{2},\frac{\pi}{6}\right)$, and its height is given by
\[
b=b_{0}\left(1-\frac{r}{r_{0}}\right),
\]
where $ r=\min\left\{r_{0},\sqrt{\left(\xi-\xi_{c}\right)^{2}+\left(\eta-\eta_{c}\right)^2}\right\}$, $b_{0}=2000$ m, and $r_{0}=\frac{\pi}{9}$.
Fig.~\ref{fig:zonal} shows the contour plots of height $h$ at $t=5,10$ and $15$ days obtained by using the ${ \mathbb{P}^{3}}$-based  RKDLEG method.  Due to the bottom  mountain, the flow pattern  observed here is unsteady and fully different  from  that in Fig. \ref{fig:steady}.
Comparing those  with the reference solution in \cite{Jakob: 1995},  the present RKDLEG  method may calculate the height  accurately.
Corresponding relative {{conservation}} errors in the total mass, total energy and potential enstrophy {{are}} given in Fig.~\ref{fig:zonalcon}. We see that the total mass is  conservative and the errors in
the total energy and potential enstrophy are about $O(10^{-8})$ and $O(10^{-4})$, respectively.
%{\color{red}{and the same phenomenon can be observed as in Fig. \ref{fig:steadycon3}.}}
\end{example}

\begin{figure}[htbp]
  \centering
  \begin{minipage}{5cm}
  \includegraphics[width=1\textwidth]{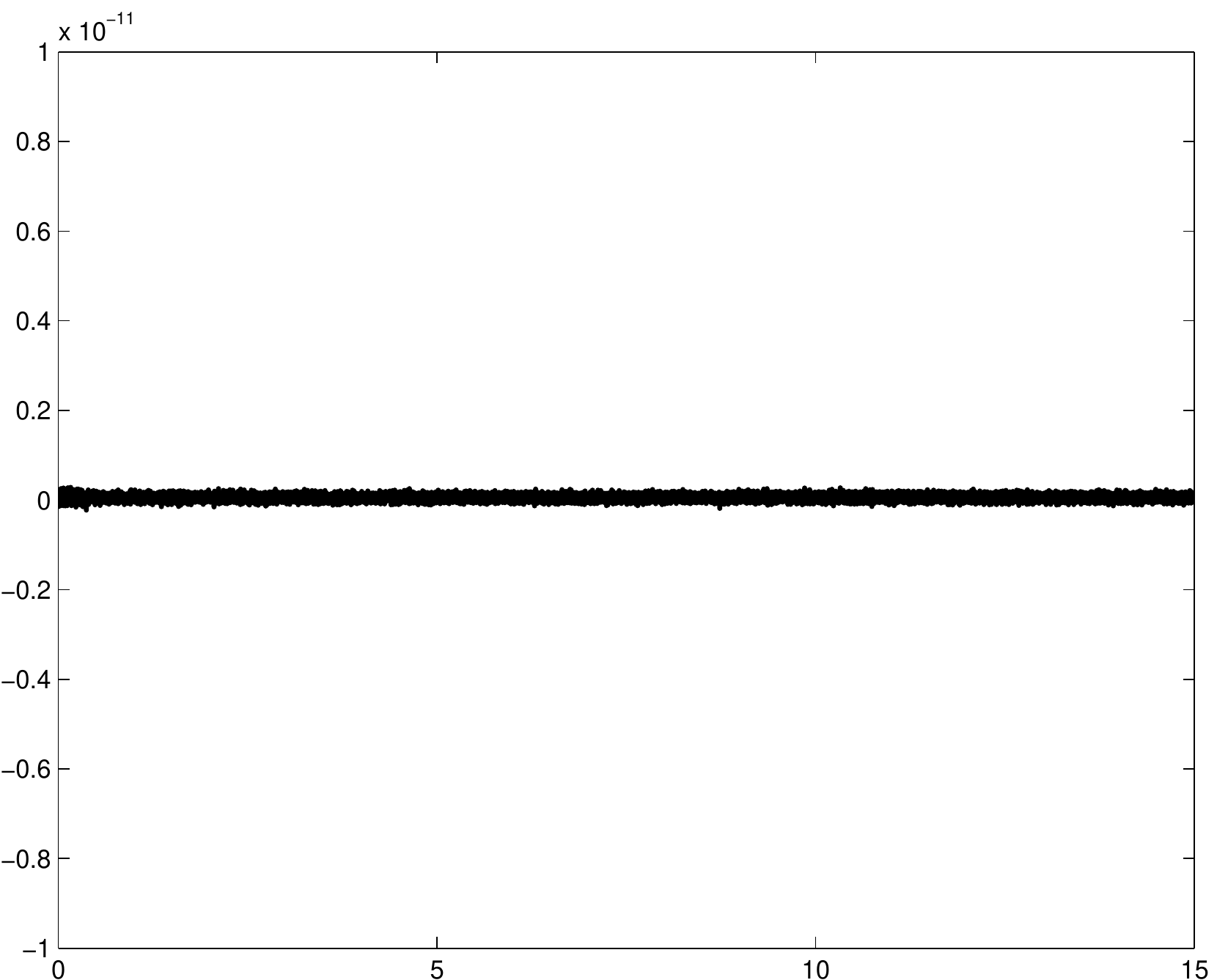}
  \end{minipage}
  \begin{minipage}{5cm}
  \includegraphics[width=1\textwidth]{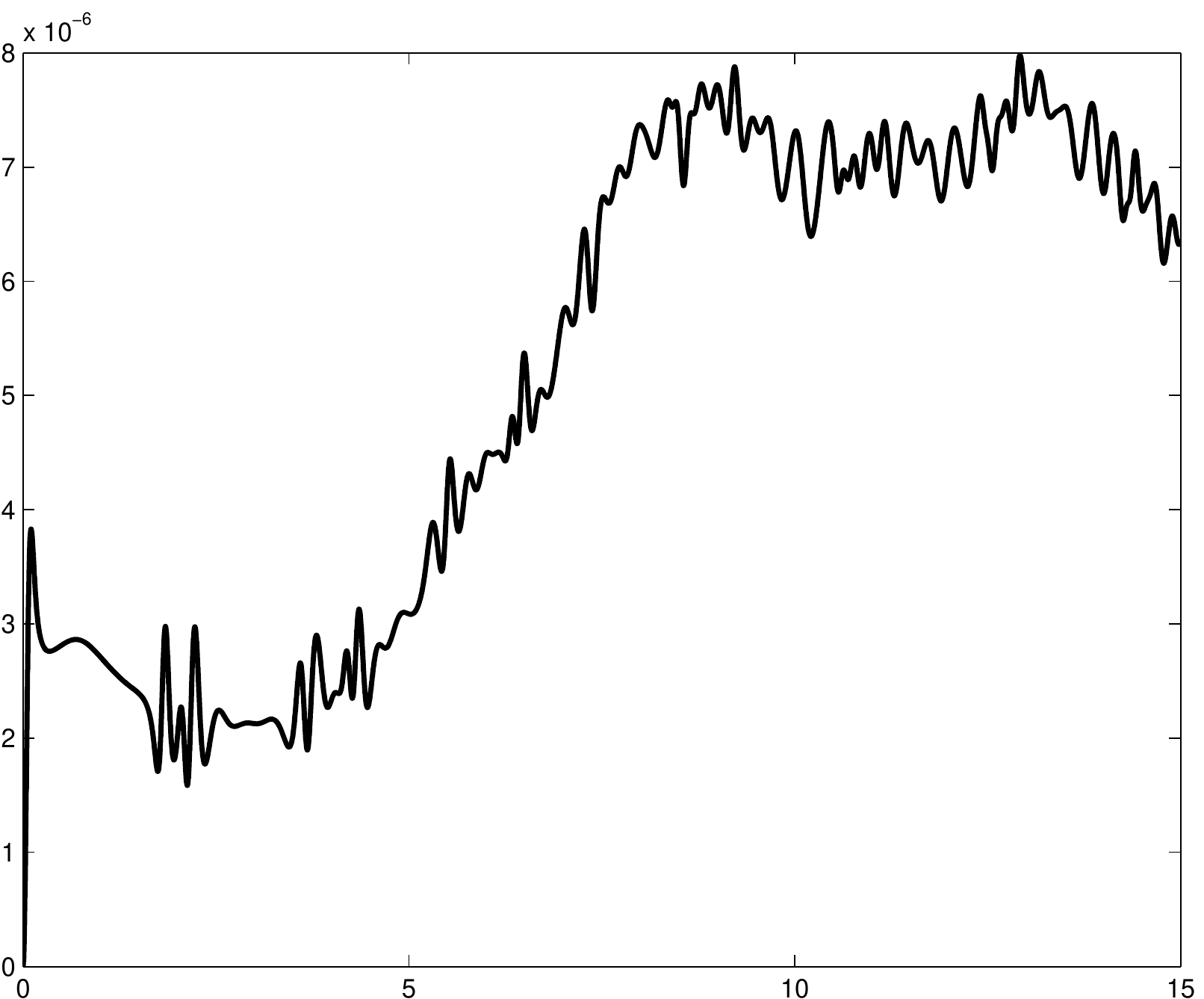}
  \end{minipage}
  \begin{minipage}{5cm}
  \includegraphics[width=1\textwidth]{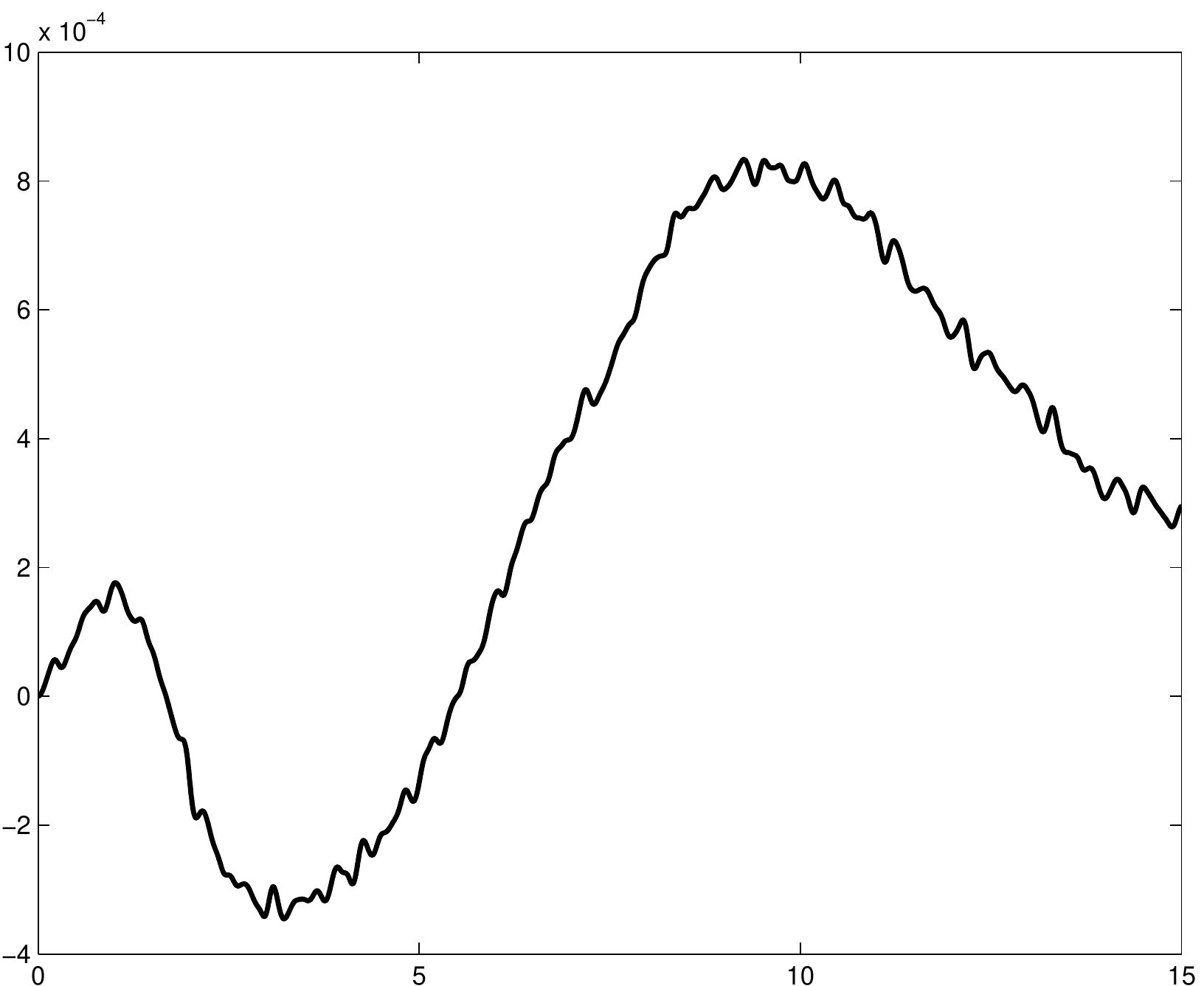}
  \end{minipage}
  \caption{\small Example \ref{zonal}: {Same as Fig. \ref{fig:steadycon3} except for example \ref{zonal} with $N=32$.}}
  \label{fig:zonalcon}
\end{figure}

\begin{figure}[htbp]
   \centering
   \begin{minipage}{7cm}
  \includegraphics[width=7cm,height=7cm]{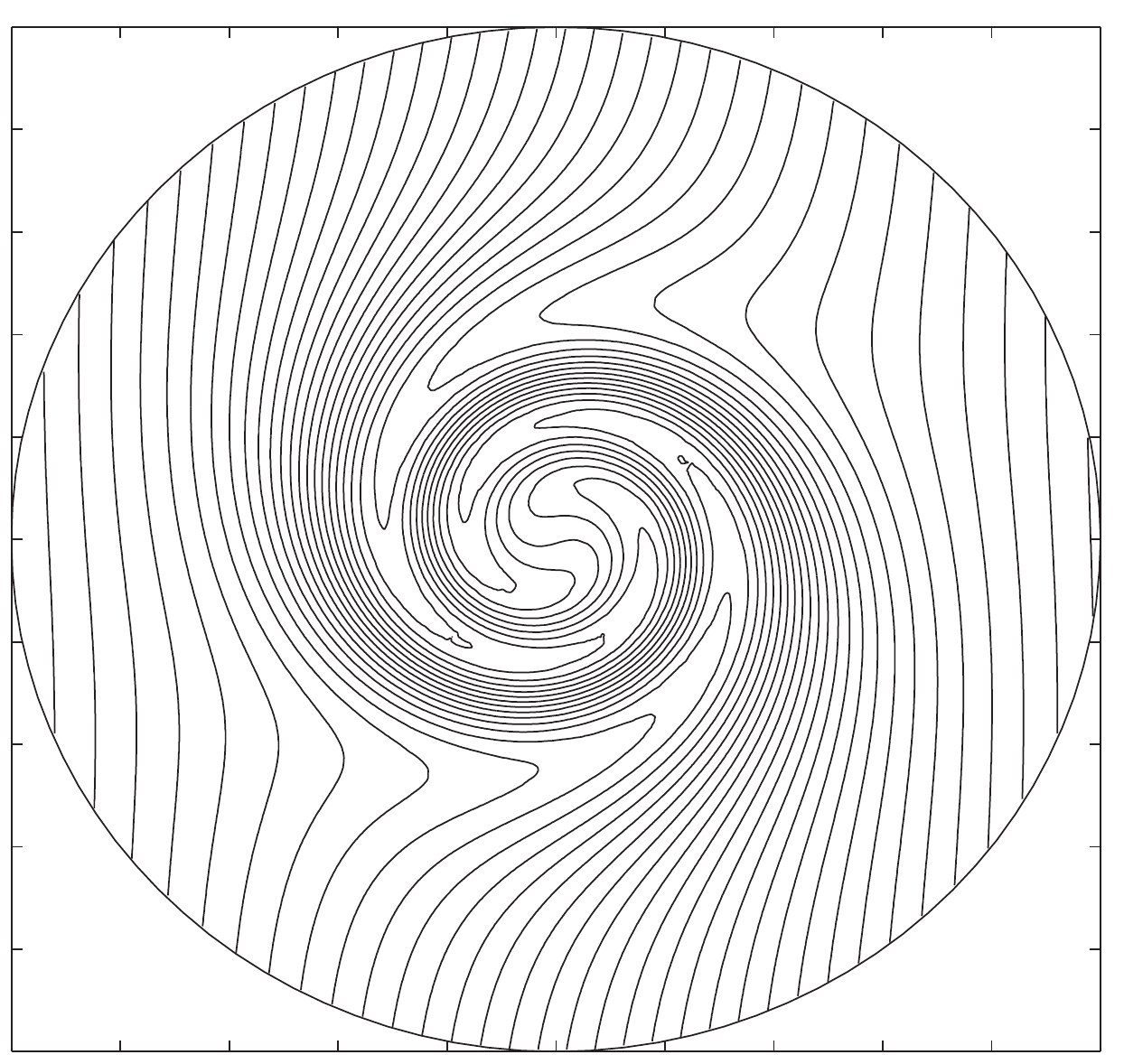}
  \end{minipage}
   \begin{minipage}{7cm}
  \includegraphics[width=7cm,height=7cm]{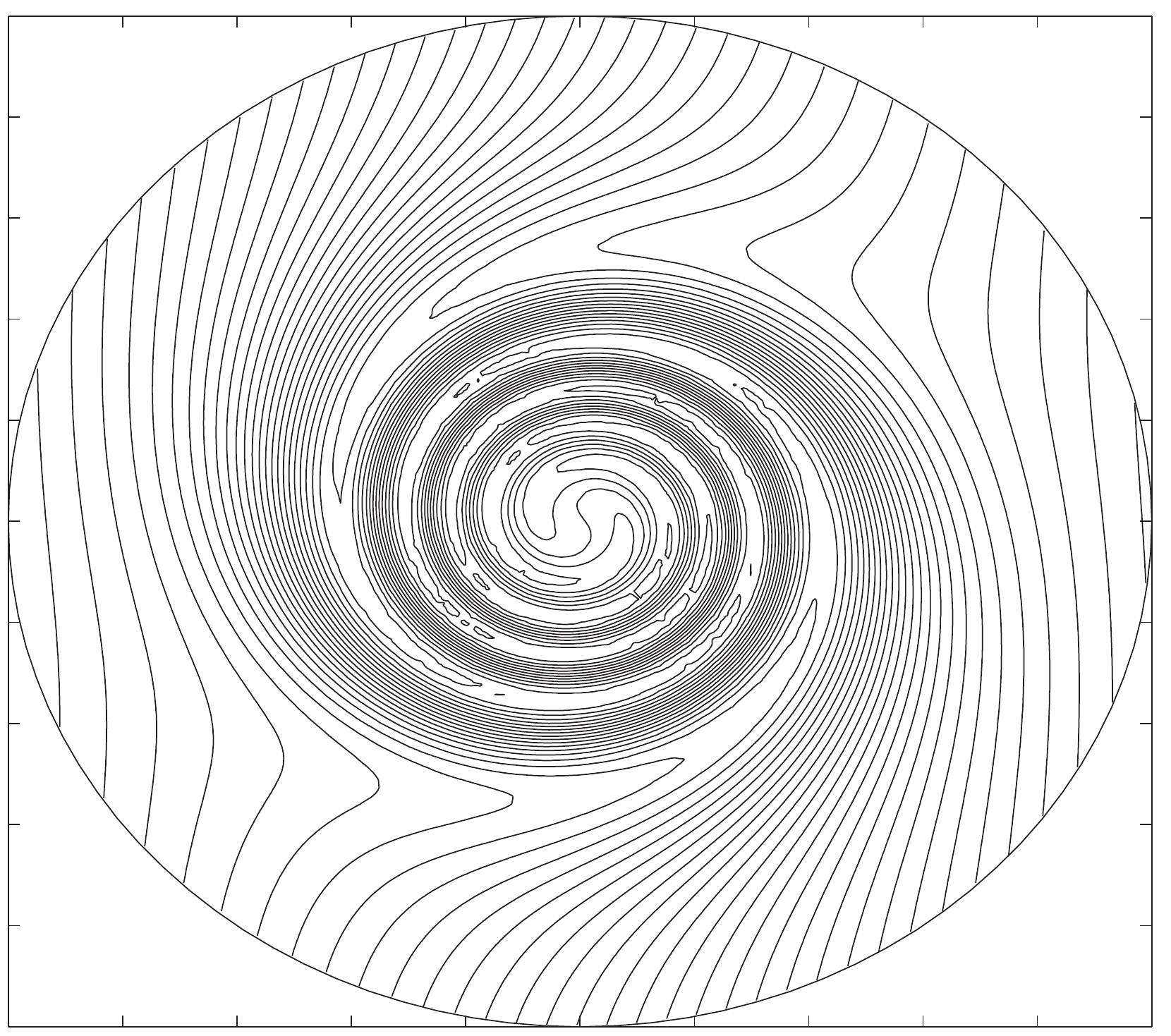}
  \end{minipage}
   \caption{\small Example \ref{defor}: The exact heights $h$  at $t = 6, 12$ days (from left to right), which are viewed from the North {{pole}} of the Earth. Contour lines are equally spaced  {{from}} 300 m to 960 m with a {stepsize}  of 20 m.}
    \label{fig:def0}
\end{figure}

\begin{figure}
\centering
   \begin{minipage}{7cm}
  \includegraphics[width=7cm,height=7cm]{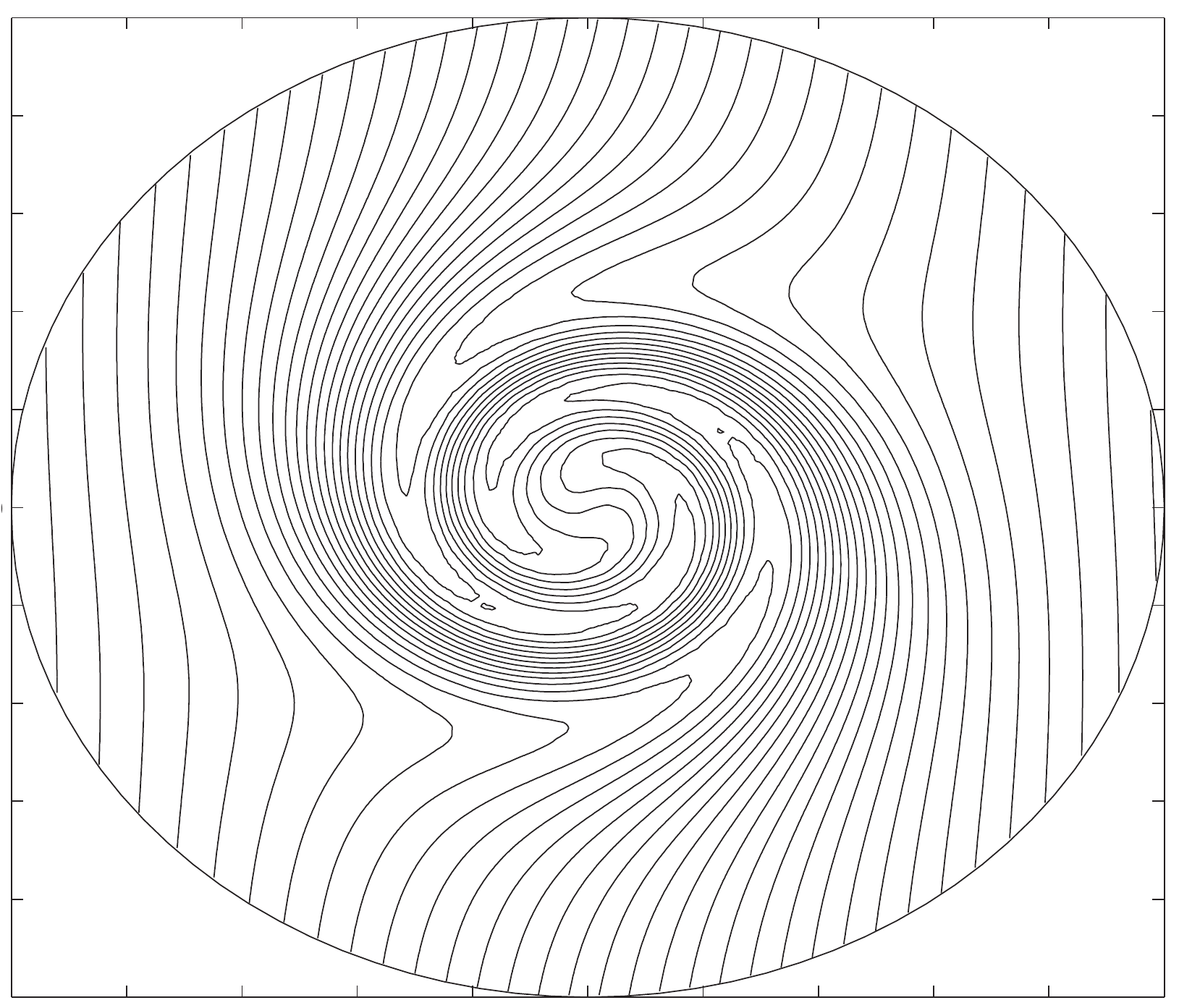}
  \end{minipage}
   \begin{minipage}{7cm}
  \includegraphics[width=7cm,height=7cm]{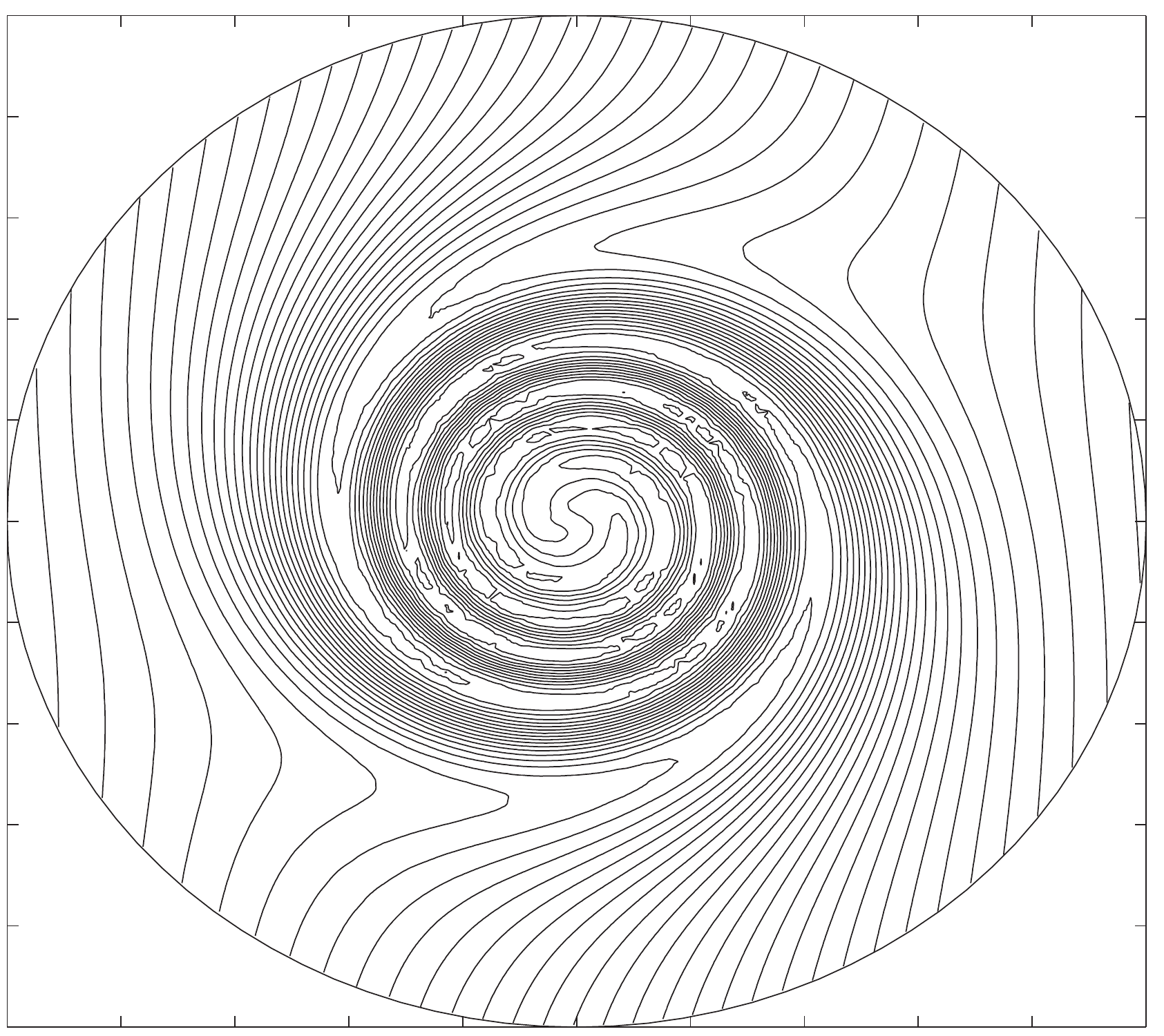}
  \end{minipage}
  \caption{\small Same as Fig.~\ref{fig:def0} except for the numerical heights $h$ obtained by using ${ \mathbb{P}^{3}}$-based RKDLEG
method with $N = 32$.}
  \label{fig:def1}
\end{figure}

% % % % % % % % % % % % % % % % % % % % % % % % % % % % % % % % % % % % % % % % % % % % % % % % % %
\begin{example}[Deformational flow]\label{defor}\rm
This example is an extension of  the pure advection  flow  but with deformation introduced in \cite{Nair: 2002}
to the SWEs \eqref{eq:latlon primitive} by adding
two ``source'' terms
\[
%\begin{cases}
%\frac{\partial u_{s}}{\partial t}+\frac{1}{R\cos\eta}\frac{\partial u_{s}}{\partial \xi}u_{s}+\frac{\partial u_{s}}{\partial \eta}v_{s}-\left(f+\frac{u_{s}}{R}\tan\eta\right)v_{s}+\frac{g}{R\cos\eta}\frac{\partial h}{\partial \xi}&=\frac{g}{R\cos\eta}\frac{\partial h^{(e)}}{\partial \xi},\\
%\frac{\partial v_{s}}{\partial t}+\frac{1}{R\cos\eta}\frac{\partial v_{s}}{\partial \xi}u_{s}+\frac{\partial v_{s}}{\partial \eta}v_{s}+\left(f+\frac{u_{s}}{R}\tan\eta\right)u_{s}+\frac{g}{R}\frac{\partial h}{\partial \eta}&=\frac{g}{R}\frac{\partial h^{(e)}}{\partial \eta}+\left(f+\frac{u_{s}^{(e)}}{R}\tan\eta\right)u_{s}^{(e)},
%\end{cases}
-\frac{g}{\cos\eta}     \frac{\partial     \big(\tanh(\frac{\rho}{\gamma}\sin(\xi-\omega t))\big)                 }{\partial \xi},
\ -{g} \frac{\partial         \big(\tanh(\frac{\rho}{\gamma}\sin(\xi-\omega t))\big)           }{\partial \eta}
+\left(f+          \omega\sin\eta \right) R\omega\cos\eta.
\]
to the right-hand side of two momentum equations respectively.
%should be added to the right hand side of the momentum equations in Eq. \eqref{eq:latlon primitive}, respectively.
% and  computed in \cite{Nair:2005-1,Pudykiewicz: 2011,Chen:2008,Chen:2014}.
The initial height $h$ and  velocity  vector $(u_s, v_s)$ are specified by
\[
h(\xi,\eta,0)=R-R\tanh\left(\frac{\rho}{\gamma}\sin\xi\right),\quad
u_{s}(\xi,\eta,0)=R\omega\cos\eta,\quad
v_{s}(\xi,\eta,0)=0,
\]
where $\rho=\rho_{0}\cos\eta$, $\rho_{0}=3$, $\gamma=5$,  and the angular velocity
\[
\omega=
\begin{cases}
\frac{3\sqrt{3}}{2}u_{0}\frac{\tanh\rho}{\rho\cosh^{2}\rho},&   \rho\neq0,\\
0,&    \rho=0,
\end{cases}
\]
with $u_{0}=\frac{\pi}{6} \mbox{day}^{-1}$.
The exact height field is taken as
\[
h(\xi,\eta,t)=R-R\tanh(\frac{\rho}{\gamma}\sin(\xi-\omega t)),
\]
 and shown in Fig.~\ref{fig:def0} for  $t=6$ and $12$ days.

Fig.~\ref{fig:def1} gives the heights $h$ at $t=6$ and $12$ days  obtained by using the ${ \mathbb{P}^{3}}$-based RKDLEG method with $N=32$. It is obvious that they are in accordance with those in Fig.~\ref{fig:def0}.
Fig.~\ref{fig:defcompare} gives  the numerical results at $t=6$ days obtained by using the ${ \mathbb{P}^{1}}$-based RKDLEG method and  ${\mathbb{P}^{1}}$-based RKDG method with {Godunov's} flux.
Comparing them to the exact
solution in the left plot of Fig.~\ref{fig:def0},
we see that the ${ \mathbb{P}^{1}}$-based RKDG method with {Godunov's} flux
gives an inaccurate solution in this case.

%It is obvious that the numerical result of RKDLEG method is closer to the exact solution.
\end{example}

\begin{figure}[htbp]
   \centering
   \begin{minipage}{7cm}
  \includegraphics[width=7cm,height=7cm]{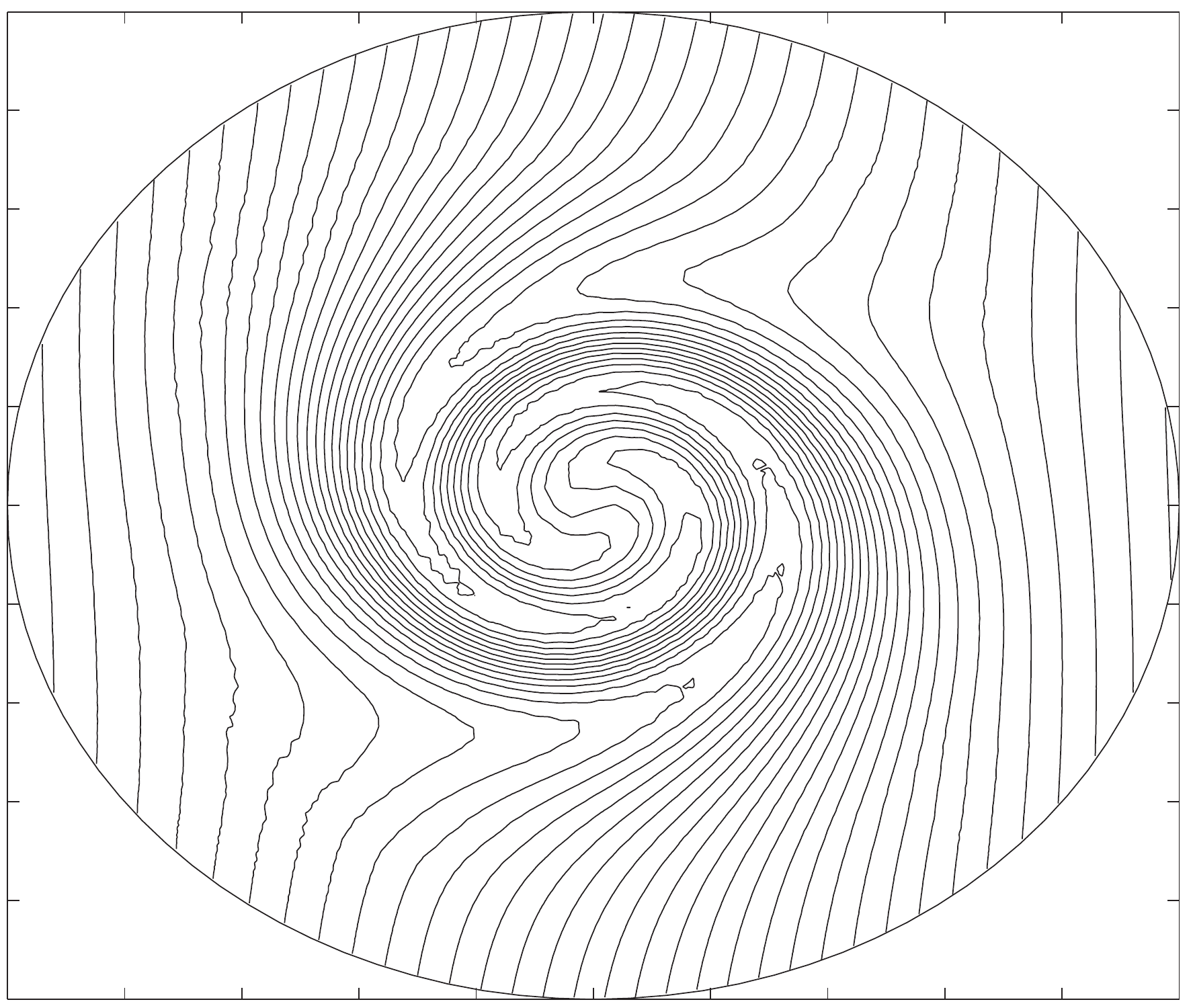}
  \end{minipage}
   \begin{minipage}{7cm}
  \includegraphics[width=7cm,height=7cm]{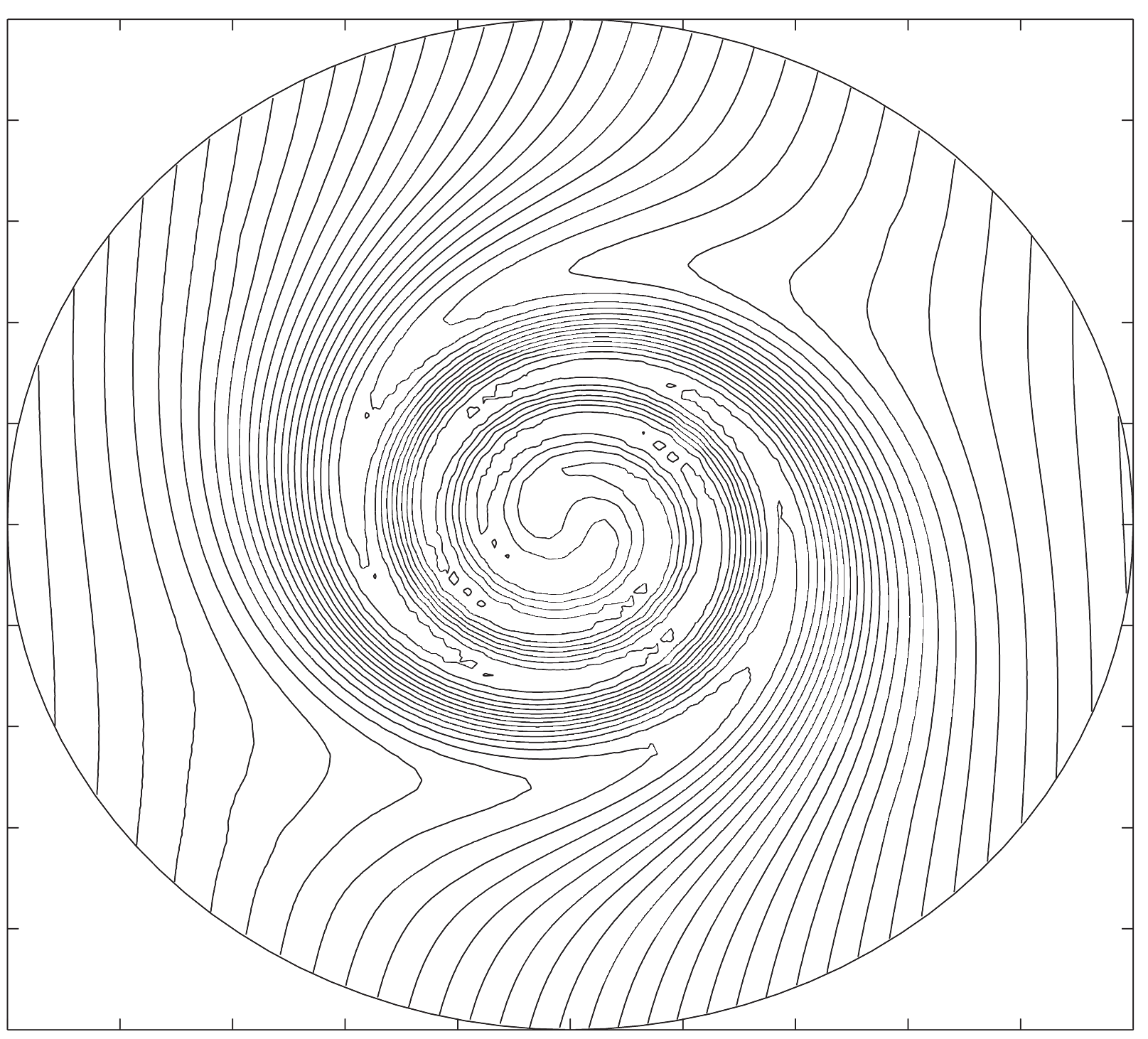}
  \end{minipage}
  \caption{\small Same as the left plot in Fig.~\ref{fig:def1} except for the ${\mathbb{P}^{1}}$-based RKDLEG method (left) and RKDG method with {Godunov's} flux (right).}
  \label{fig:defcompare}
\end{figure}

% % % % % % % % % % % % % % % % % % % % % % % % % % % % % % % % % % % % % % % % % % % %
\begin{example}[Rossby-Haurwitz wave]\label{rossby}\rm
Rossby-Haurwitz waves are steadily propagating solutions of the fully nonlinear non-divergent barotropic vorticity equation on a sphere
and have  been used to test shallow water numerical models,
see the 6th case of the standard shallow-water test {provided} by Williamson et al. \cite{Williamson:1992}.
Rossby-Haurwitz waves with zonal wave-numbers less than or equal to 5 are commonly believed to be stable, otherwise unstable.

The initial height and divergence-free velocity vector are
specified as follows
\begin{align*}
h(\xi,\eta,0)=&h_{0}+g^{-1}R^{2}\left(A\left(\eta\right)+B\left(\eta\right)\cos\left(r\xi\right)+C\left(\eta\right)\cos\left(2r\xi\right)\right),\\
u_{s}(\xi,\eta,0)=&RK\cos\eta+RK\cos^{r-1}\eta\left(r\sin^{2}\eta-\cos^2\eta\right)\cos\left(r\xi\right),\\
v_{s}(\xi,\eta,0)=&-RKr\cos^{r-1}\eta\sin\eta\sin\left(r\xi\right),
\end{align*}
where
\begin{align*}
A\left(\eta\right)= &\frac{K}{2}\left(2\Omega+K\right)\cos^{2}\eta+\frac{1}{4}K^{2}\cos^{2r}\eta\left[\left(r+1\right)\cos^{2}\eta+\left(2r^{2}-r-2\right)-\frac{2r^2}{\cos^{2}\eta}\right],\\
B\left(\eta\right)= &\frac{2\left(\Omega+K\right)K}{\left(r+1\right)\left(r+2\right)}\cos^{r}\eta\left[\left(r^2+2r+2\right)-\left(r+1\right)^{2}\cos^{2}\eta\right],\\
C\left(\eta\right)= &\frac{1}{4}K^{2}\cos^{2r}\eta\left[\left(r+1\right)\cos^{2}\eta-\left(r+2\right)\right],
\end{align*}
with $ K=7.848\times10^{-6}$ s$^{-1}$, the zonal wave number $r=4$,
and $h_{0}=8000$ m.
% $A,B,C$ are functions of the latitude $\eta$.
%Some comments of properties and numerical solutions from some global models of Rossby-Haurwitz wave are described in \cite{Thuburn:2000}.

%In particular, the Rossby-Haurwitz wave remains stable during the entire time of simulation propagating eastwards and the initial structure of the wave number 4 is well maintained with only minimal vacillations in shape, so it provides a good test for long time simulation \cite{Pudykiewicz: 2011}.

%The divergence-free initial velocity field is given by stream function which is in spherical component
%\[
%\psi(\xi,\eta)=-R^{2}K\sin\eta+R^{2}K\cos^{r}\eta\sin\eta\cos\left(r\xi\right).
%\]

Fig.~\ref{fig:rossby2} displays the heights at $t = 7$ and $14$ days obtained by using the ${  \mathbb{P}^{3}}$-based RKDLEG method with $N=48$.
Those results agree well with the widely accepted reference solutions \cite{Jakob: 1995}, the wave propagates steadily eastward, and superposed on this steady propagation are small vacillations in the wave structure.
%without change of shape with the angular velocity %$\omega$
%$\big(r(3+r)K-2\Omega\big)/{(1+r)(2+r)}$.
Fig. \ref{fig:steadycon3} gives the time evolution of the  relative {{conservation}} errors of total mass, total energy and potential enstrophy.
It is obvious that the total mass is conservative and the error in the total
energy is about $O(10^{-7})$, but the error in the total
potential enstrophy is slightly big and its order of magnitude
is ${-2}$.
%Orders of magnitude are written in powers of 10. For example, the order of magnitude of 1500 is 3, since 1500 may be written as 1.5¡Á103.

Figs.~\ref{fig:rossby0} and \ref{fig:rossby1} also give the solutions at $t=7$ and $14$ days
obtained respectively by using the ${ \mathbb{P}^{1}}$-based RKDLEG method and ${ \mathbb{P}^{1}}$-based RKDG method with {Godunov's} flux with $N=48$.
Comparing them, it is not difficult to see that
 the RKDLEG method may get the solutions more similar to the reference solutions \cite{Jakob: 1995} than  the RKDG method with {Godunov's} flux.

\end{example}

\begin{figure}[htbp]
   \centering
   \begin{minipage}{7.5cm}
  \includegraphics[width=7cm,height=4cm]{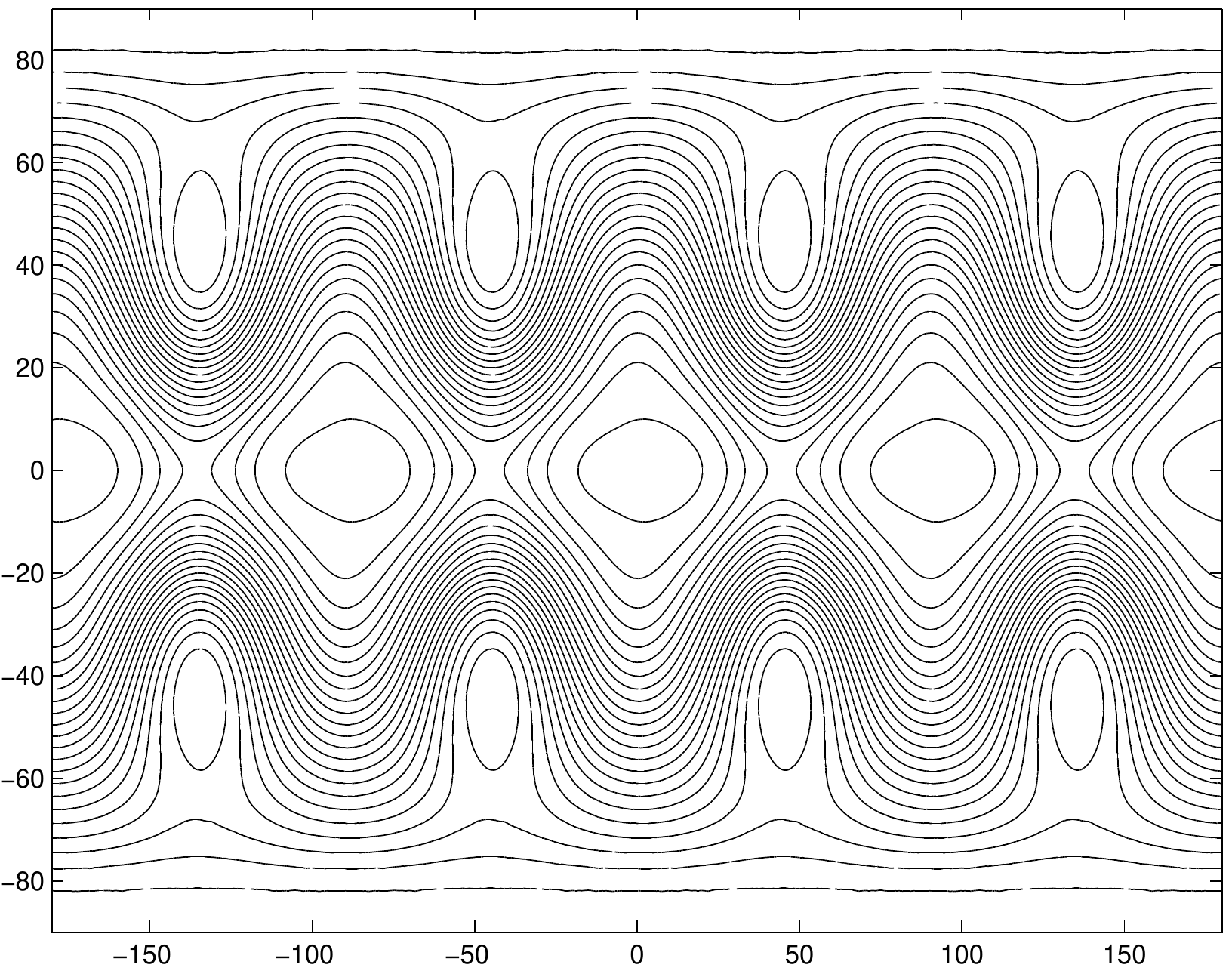}
  \end{minipage}
   \begin{minipage}{7.5cm}
  \includegraphics[width=7cm,height=4cm]{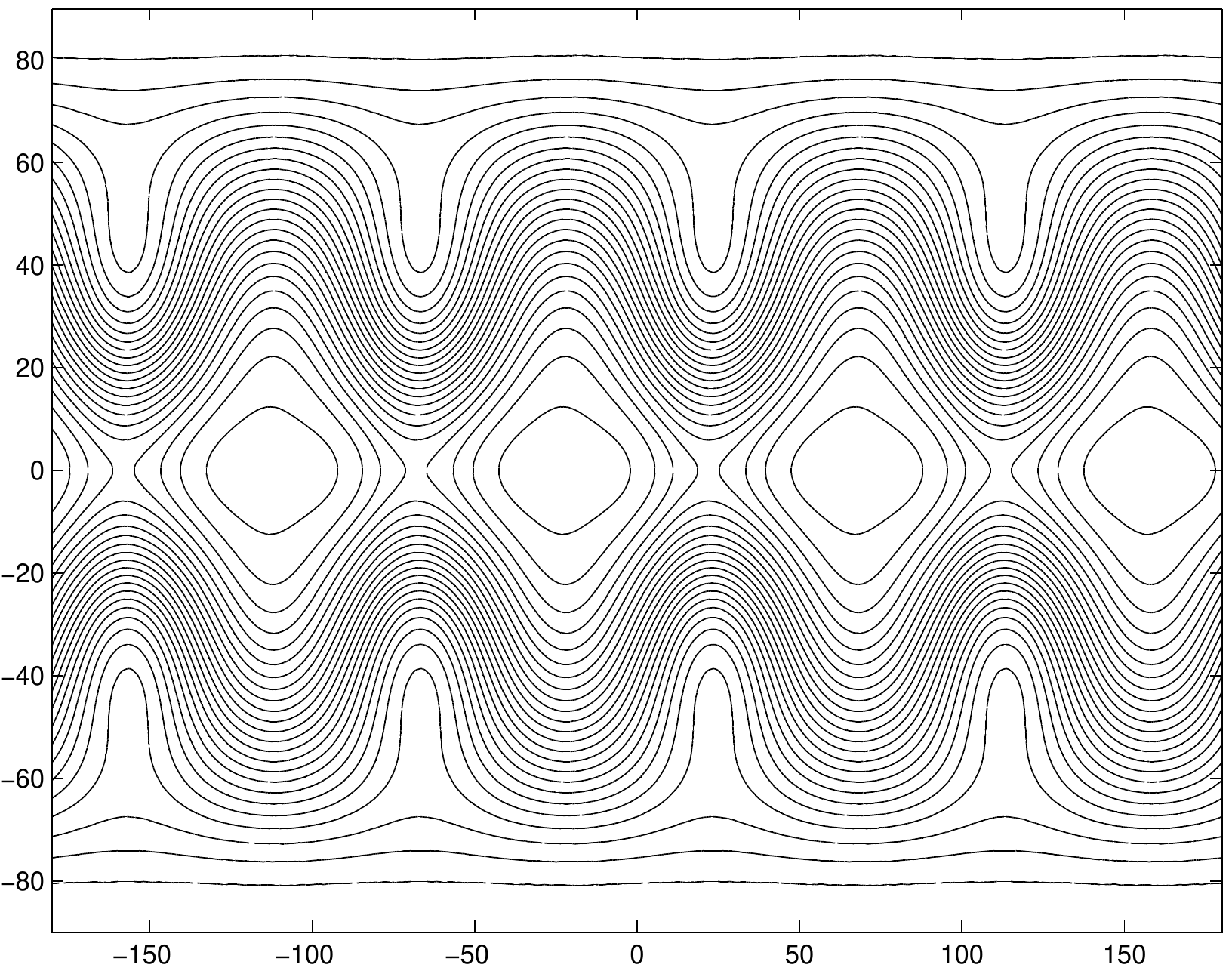}
\end{minipage}
  \caption{\small
  Example \ref{rossby}: The heights at $t = 7$ (left) and 14 days (right) obtained by using ${ \mathbb{P}^{3}}$-based RKDLEG method with $N = 48$. Contour lines are equally spaced {from} $8100$ m to $11000$ m with a stepsize of $100$ m.
  }\label{fig:rossby2}
\end{figure}

\begin{figure}[htbp]
  \centering
  \begin{minipage}{5cm}
  \includegraphics[width=1\textwidth]{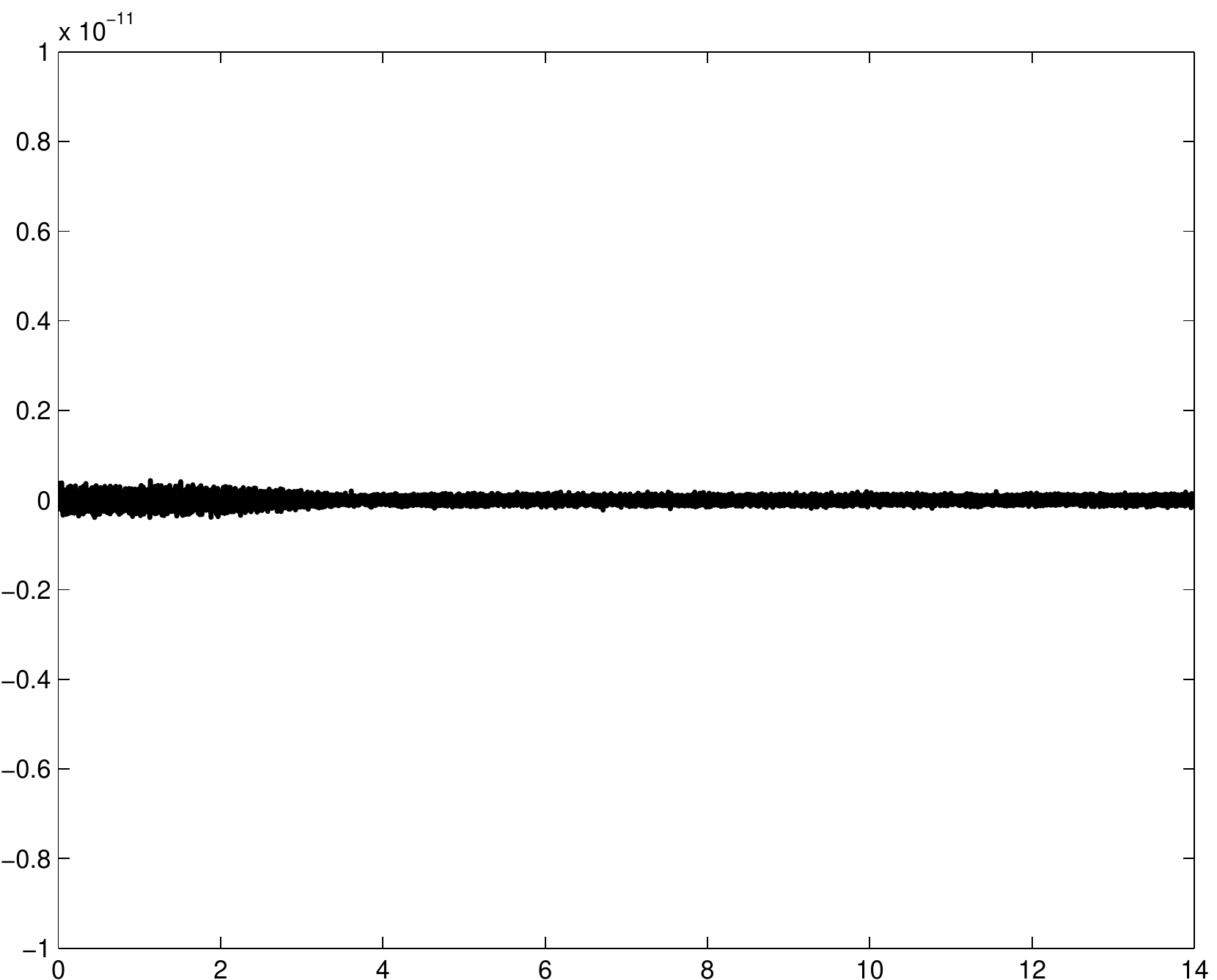}
  \end{minipage}
  \begin{minipage}{5cm}
  \includegraphics[width=1\textwidth]{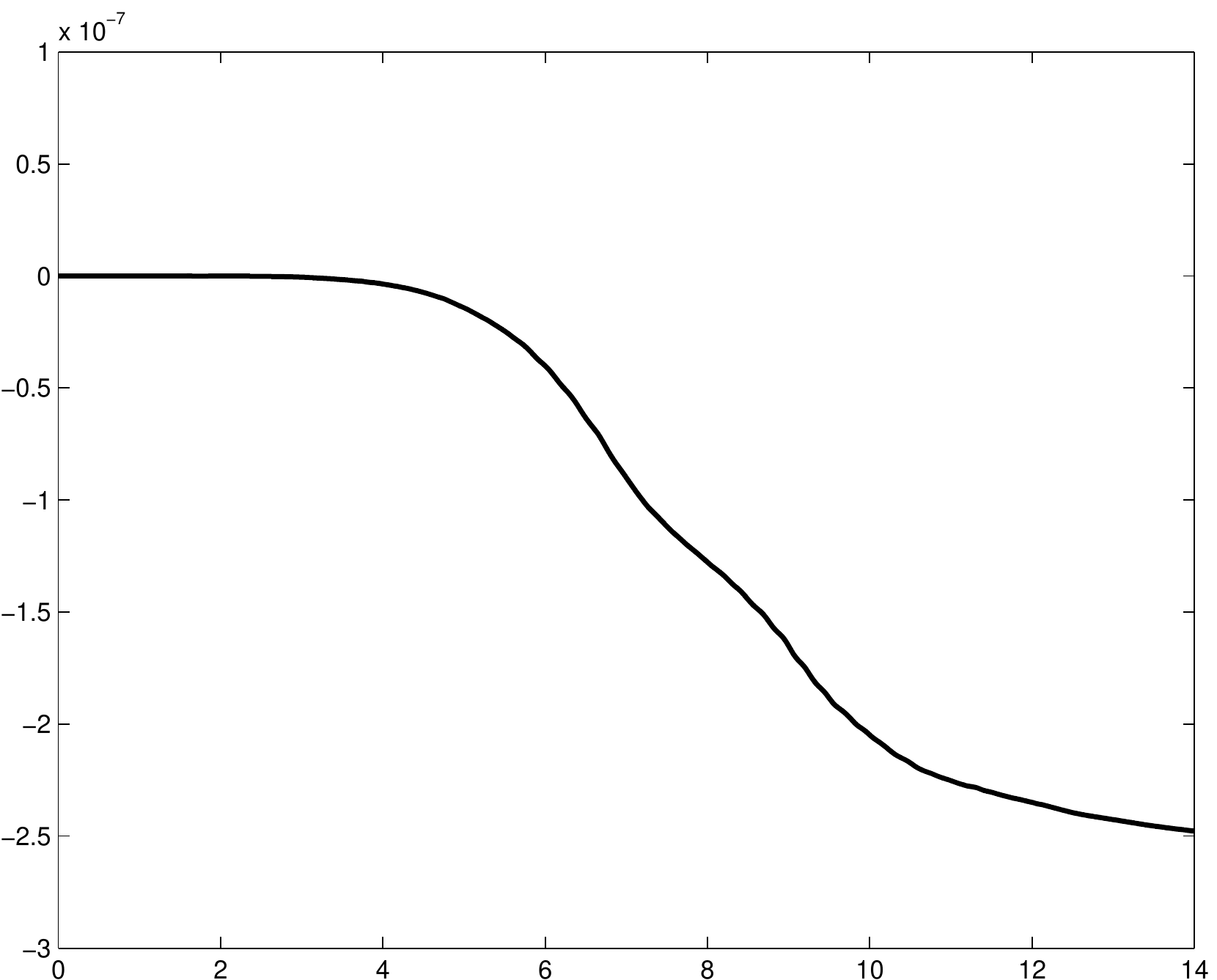}
  \end{minipage}
  \begin{minipage}{5cm}
  \includegraphics[width=1\textwidth]{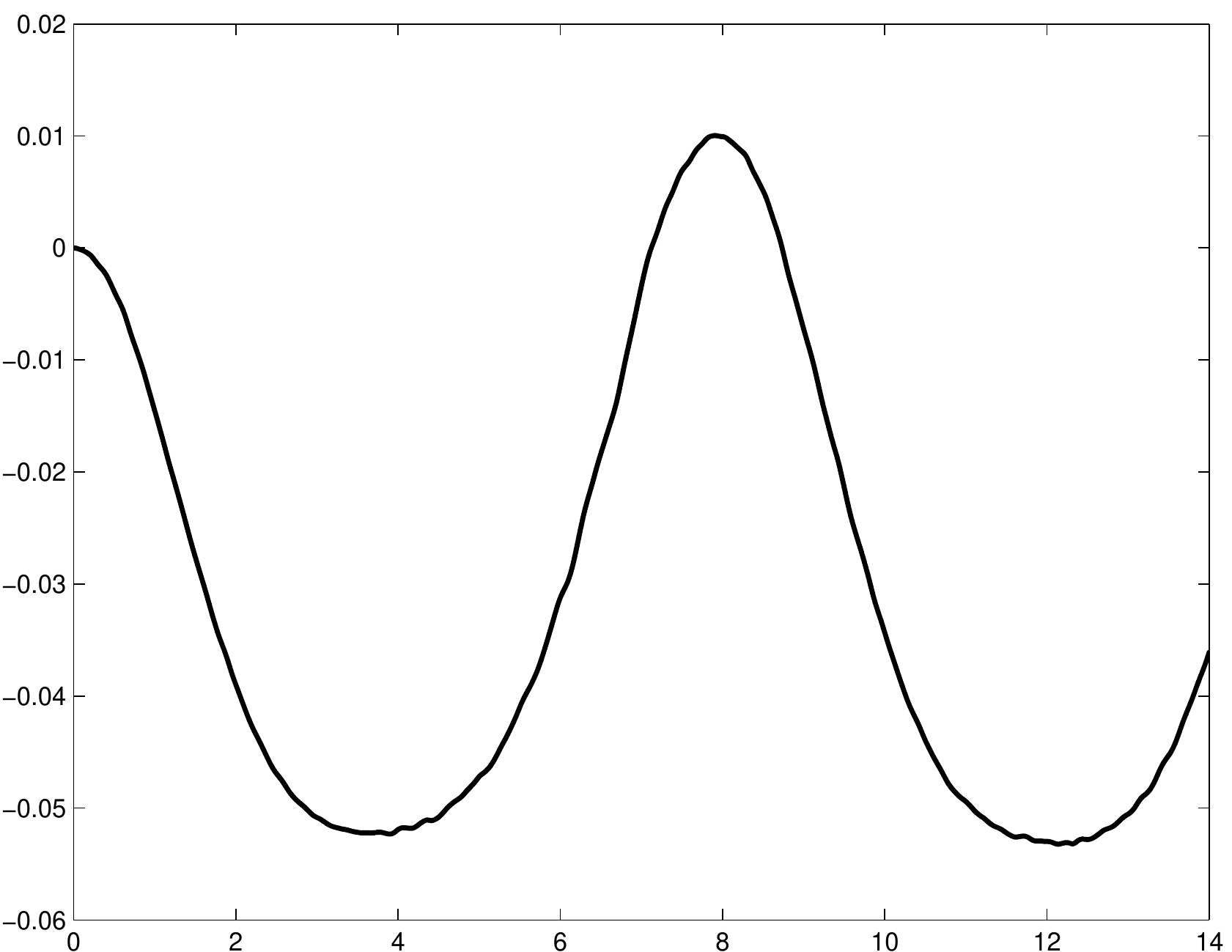}
  \end{minipage}
  \caption{\small Example \ref{rossby}.
  {Same as Fig. \ref{fig:steadycon3} except for example \ref{rossby} with $N=48$.}}
  \label{fig:rossbycon}
\end{figure}

\begin{figure}[htbp]
   \centering
   \begin{minipage}{7.5cm}
  \includegraphics[width=7cm,height=4cm]{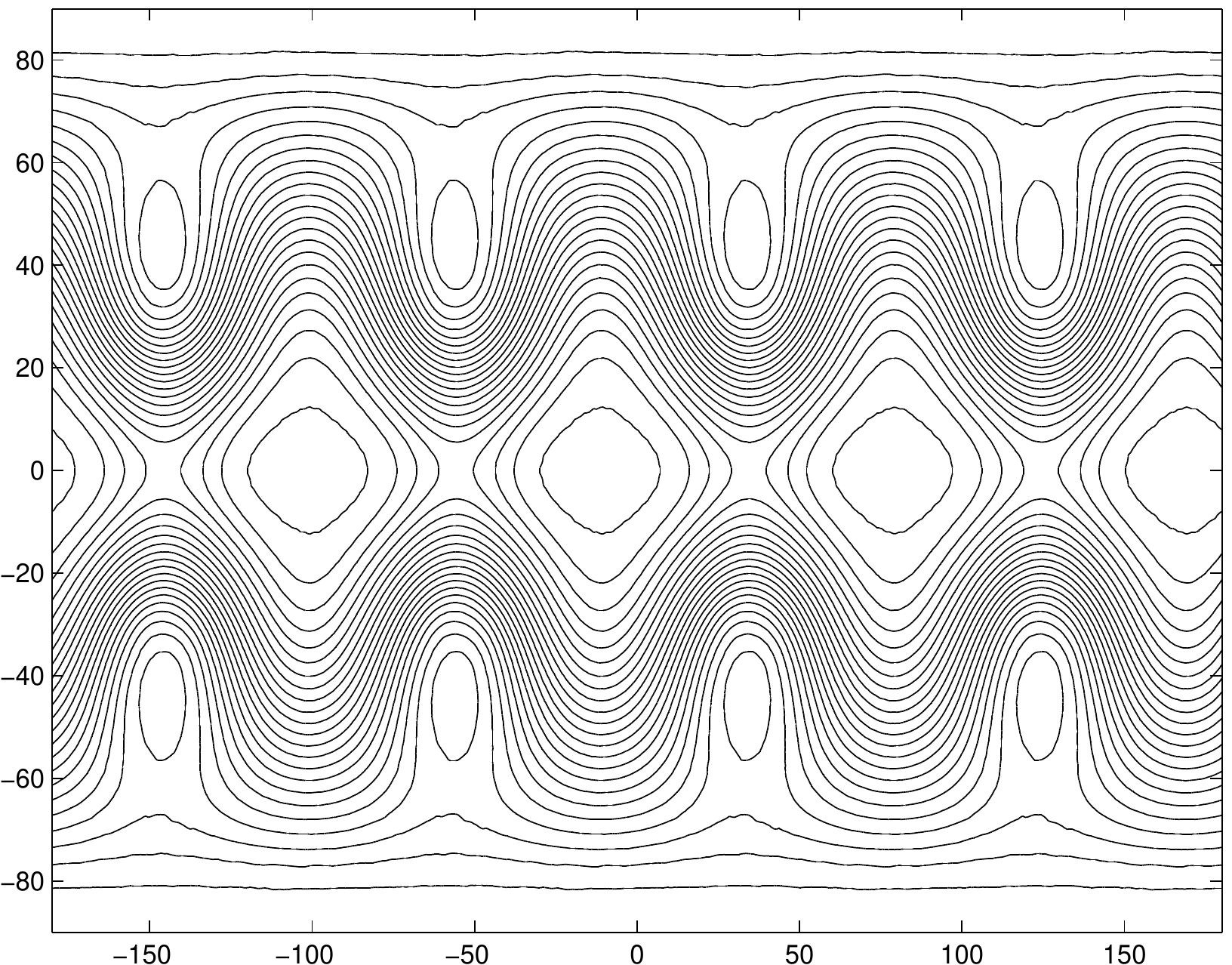}
  \end{minipage}
   \begin{minipage}{7.5cm}
  \includegraphics[width=7cm,height=4cm]{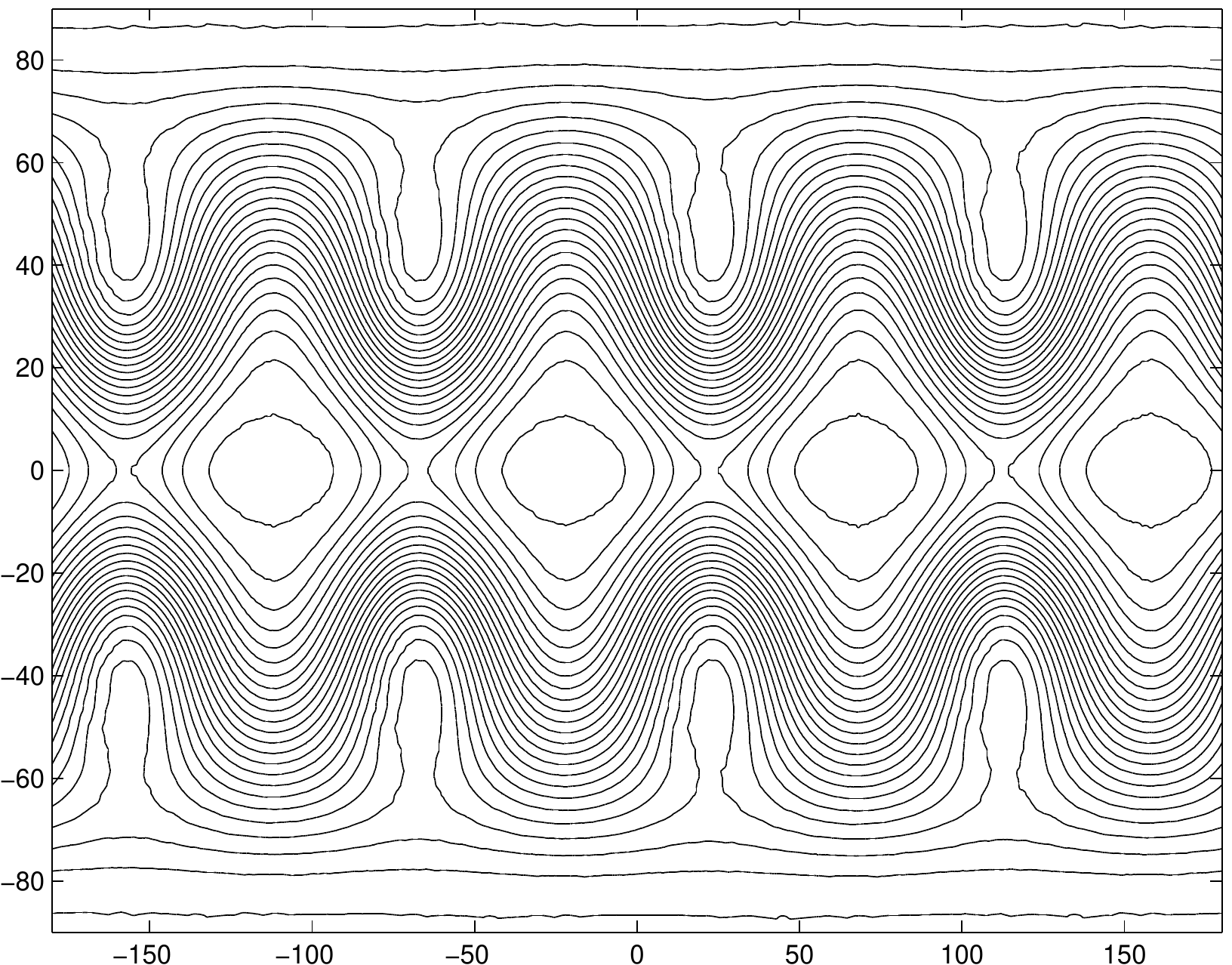}
  \end{minipage}
    \caption{\small
    Same as Fig.~\ref{fig:rossby2} except for the ${\mathbb{P}^{1}}$-based RKDLEG method.}
   % Example \ref{rossby}: The heights at $t = 7$ (left) and 14 days (right) obtained by using ${ \mathbb{P}^{1}}$-based RKDLEG method with $N = 48$. Contour lines are equally spaced {from} $8100$ m to $11000$ m with a stepsize of $100$ m.}
\label{fig:rossby0}
\end{figure}

\begin{figure}
   \centering
   \begin{minipage}{7.5cm}
  \includegraphics[width=7cm,height=4cm]{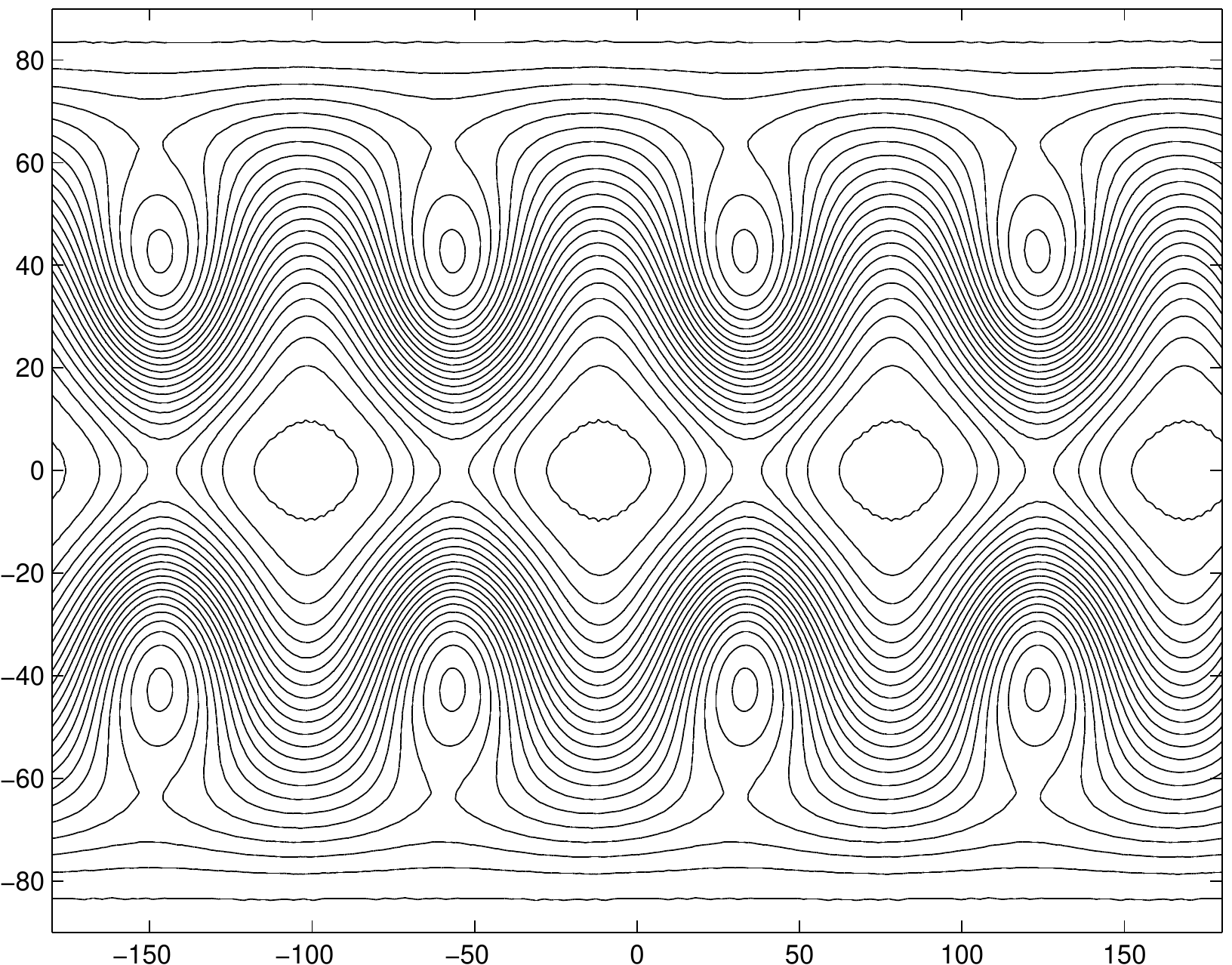}
  \end{minipage}
   \begin{minipage}{7.5cm}
  \includegraphics[width=7cm,height=4cm]{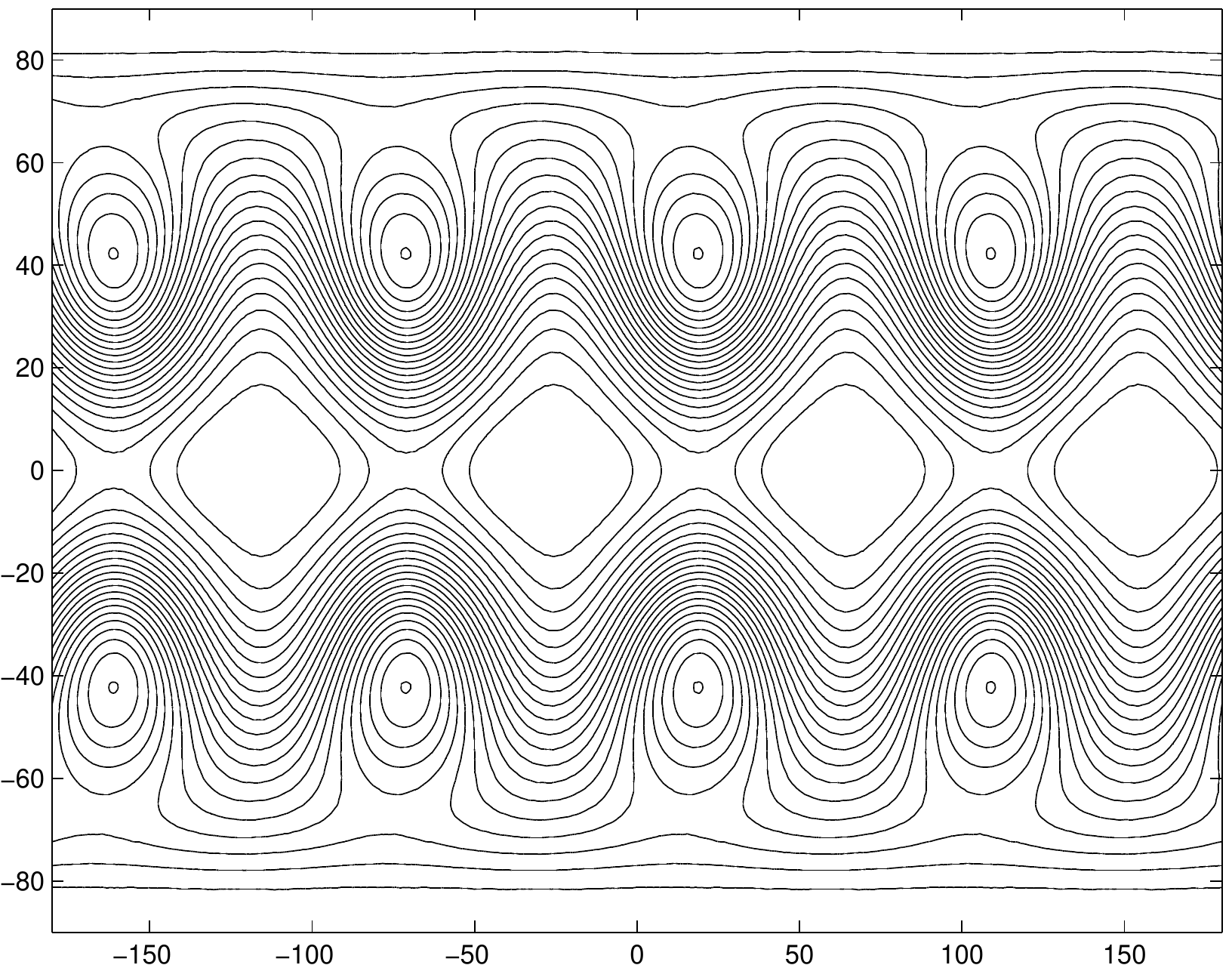}
  \end{minipage}
  \caption{\small Same as Fig.~\ref{fig:rossby2} except for the ${\mathbb{P}^{1}}$-based RKDG method with {Godunov's} flux.}
  \label{fig:rossby1}
\end{figure}

% % % % % % % % % % % % % % % % % % % % % % % % % % % % %
\begin{example}[Cross-polar flow]\label{polar}\rm
It is first proposed in \cite{McDonald:1989}.
Initially, there are a low and high patterns which are symmetrically located at the left and  right hand sides of the pole, respectively, when it is viewed from above. The low or high pattern rotates in clockwise direction around the pole \cite{Nair:2005-2}.

The initial height and velocity vector are taken as
\begin{align*}
h(\xi,\eta,0)=&h_{0}-2g^{-1}R\omega u_{0}\sin^{3}\eta\cos\eta\sin\xi,\\
u_{s}(\xi,\eta,0)=&-u_{0}\sin\xi\sin\eta\left(4\cos^{2}\eta-1\right),\\
v_{s}(\xi,\eta,0)=&u_{0}\sin^{2}\eta\cos\xi,
\end{align*}
where $h_{0}=5.768\times10^{4}g^{-1}\mbox{ m}$
and $u_{0}=20 \mbox{ m/s}$.
It means that the initial cross-polar flow is of strength $u_0$,
and both wind components were zero at the equator.

Fig.~\ref{fig:polarhuv} shows the solutions $(h,u_s,v_s)$ at $t=10$ days obtained  by using the ${ \mathbb{P}^{3}}$-based RKDLEG method  with $N=54$, where
the contour lines of $h$  are equally spaced from $5350$ m to $6330$ m with an interval of $80$ m, while
the contour lines of $u_{s}$ and $v_{s}$ are taken from $-21\mbox{ m/s}$ to $21\mbox{ m/s}$  and from  $-13\mbox{ m/s}$ to $13\mbox{ m/s}$ with a stepsize of $2\mbox{ m/s}$, respectively.
Our results are comparable to those given in \cite{Giraldo: 2002,Nair:2005-2}.
The conservation of total mass,   energy and potential enstrophy
may be demonstrated via
the relative {{conservation}} error plots given in Fig.~\ref{fig:polarcon}.
The errors in the total energy and potential enstrophy
are about $O(10^{-13})$ and $O(10^{-3})$ respectively.
\end{example}

\begin{figure}[htbp]
   \centering
   \begin{minipage}{5cm}
  \includegraphics[width=5cm,height=5cm]{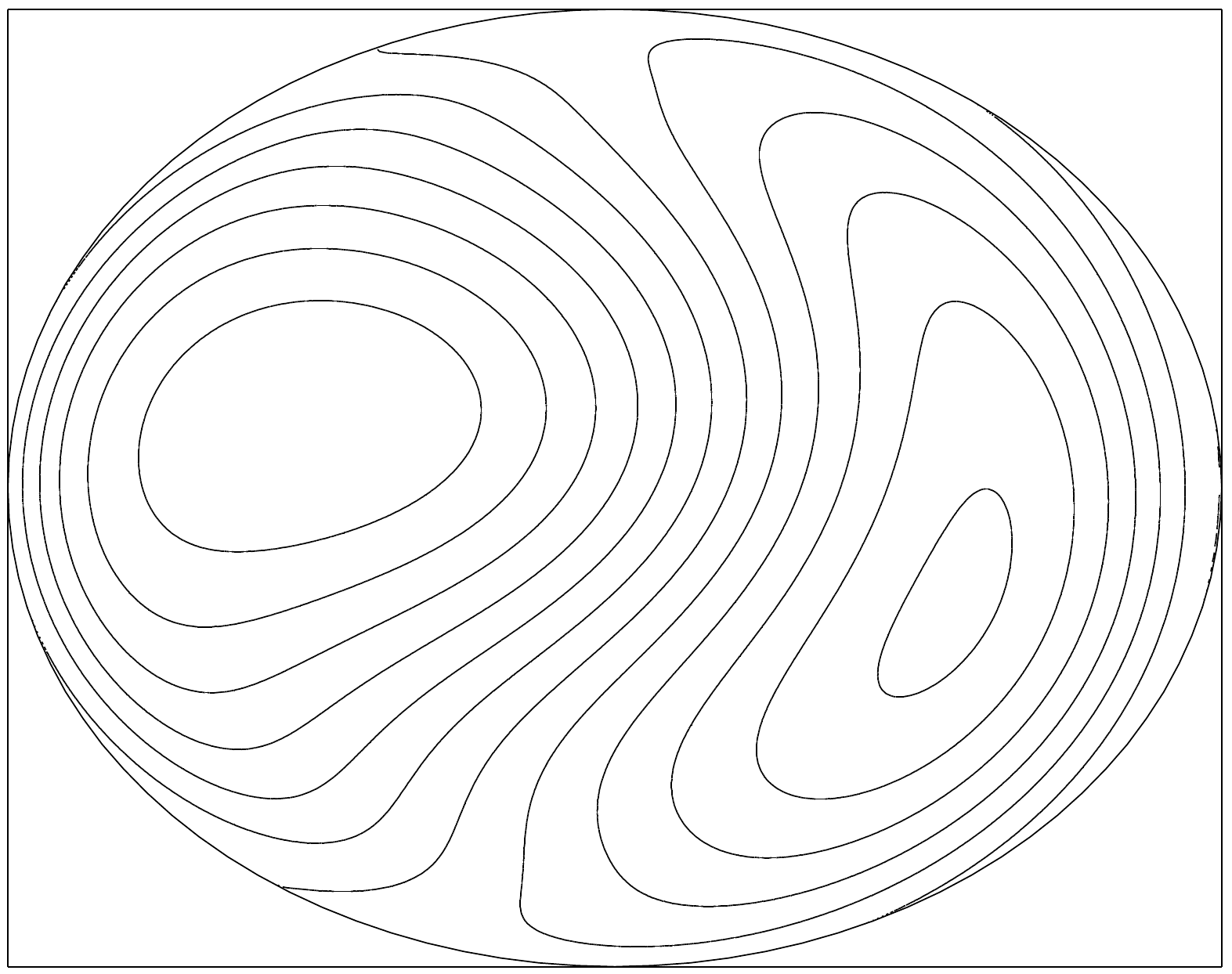}
  \end{minipage}
   \begin{minipage}{5cm}
  \includegraphics[width=5cm,height=5cm]{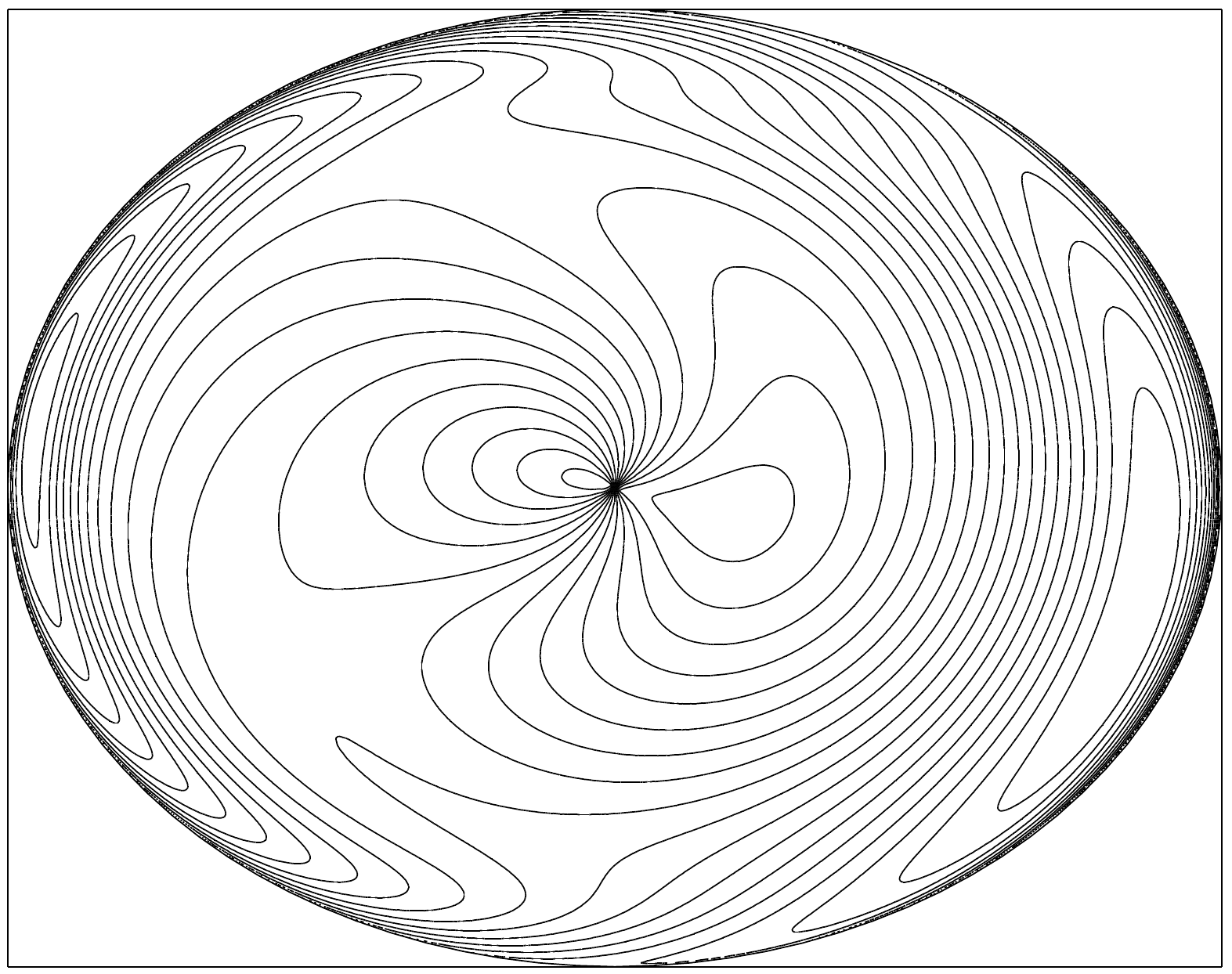}
  \end{minipage}
  \begin{minipage}{5cm}
  \includegraphics[width=5cm,height=5cm]{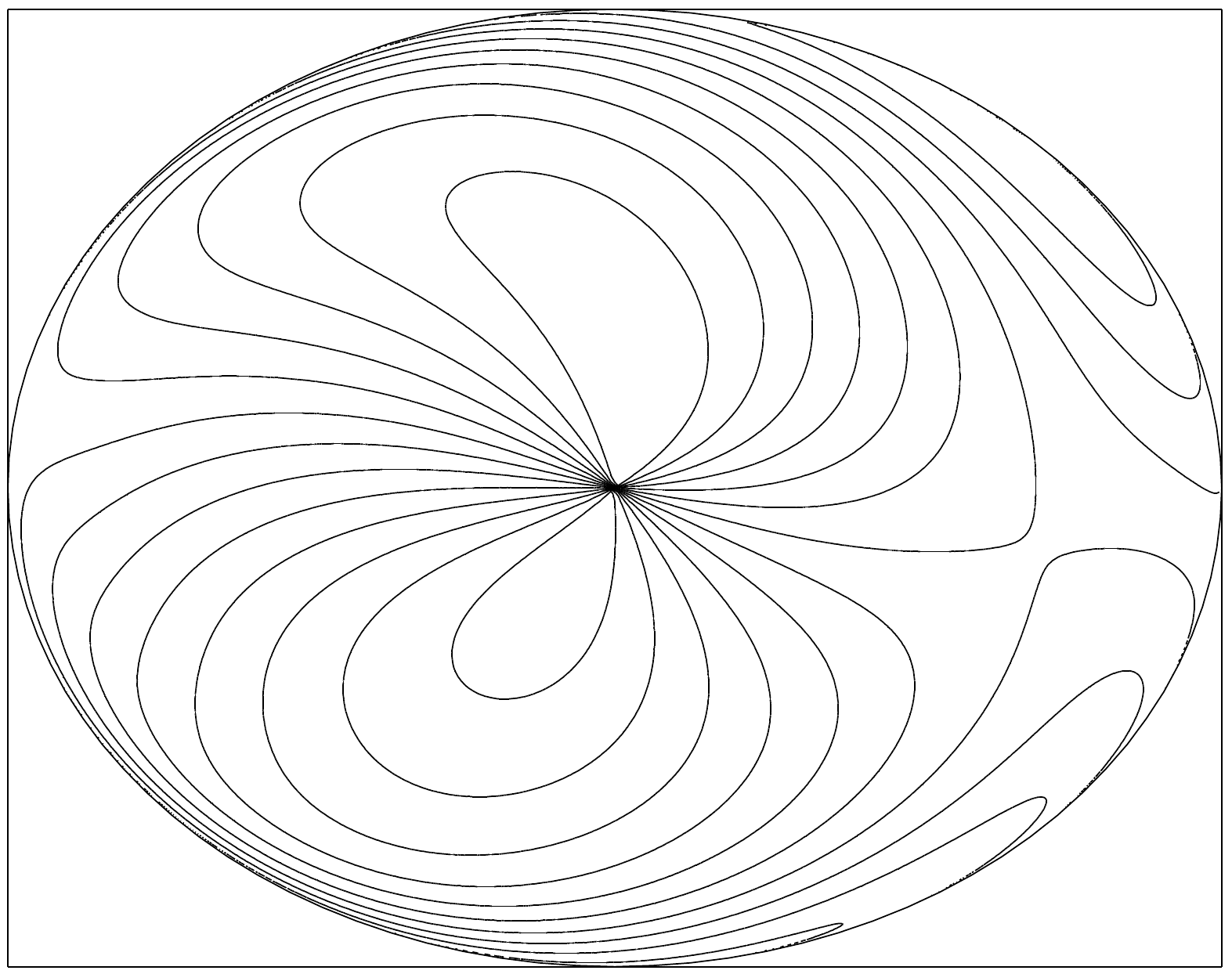}
  \end{minipage}
  \caption{\small Example \ref{polar}: The solutions $h$, $u_{s}$, and $v_{s}$ (from left to right) at $t=10$ days obtained by using ${\mathbb{P}^{3}}$-based RKDLEG method  with $N=54$, which are viewed from the North {{pole}}.}
  \label{fig:polarhuv}
\end{figure}

\begin{figure}[htbp]
  \centering
  \begin{minipage}{5cm}
  \includegraphics[width=1\textwidth]{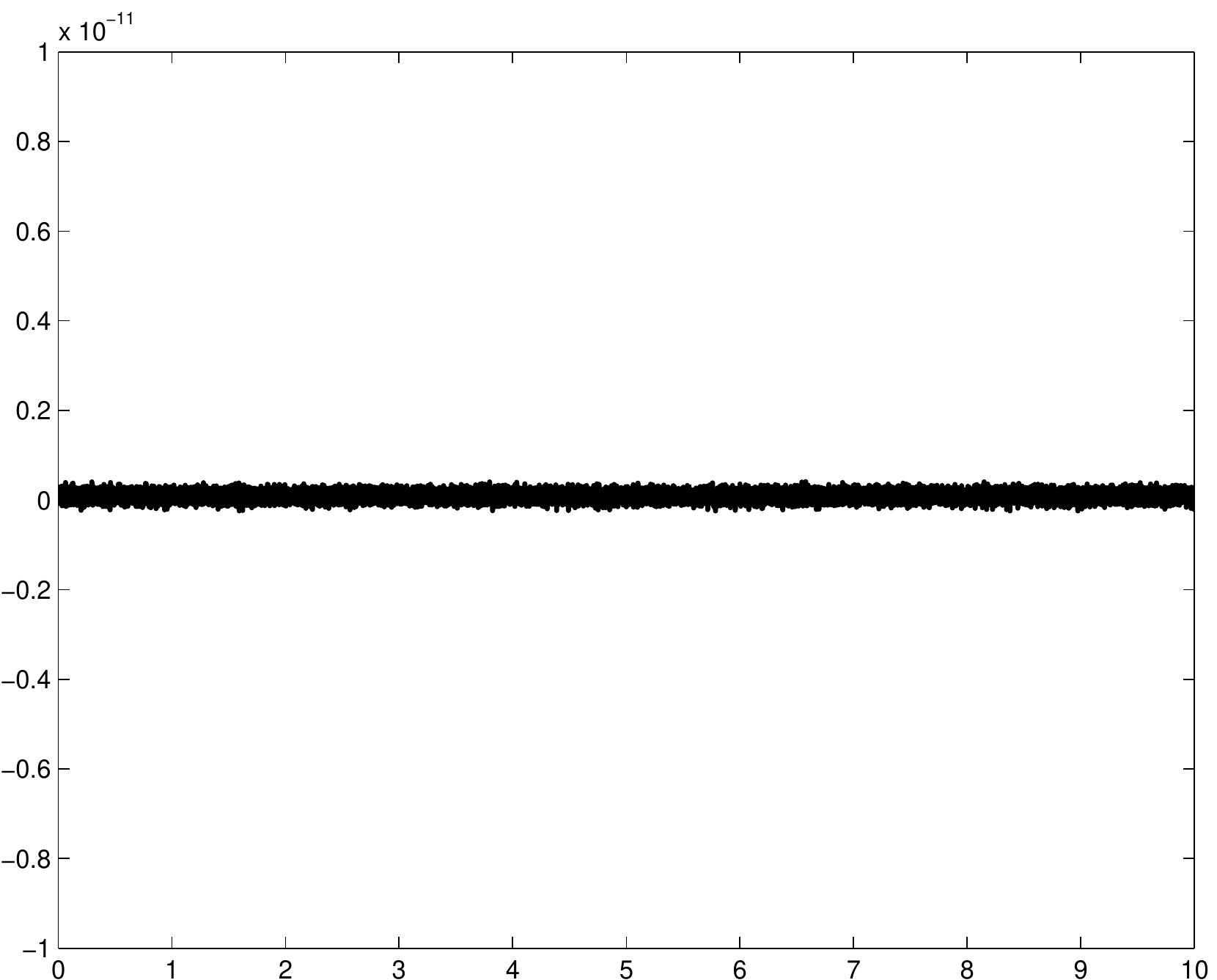}
  \end{minipage}
  \begin{minipage}{5cm}
  \includegraphics[width=1\textwidth]{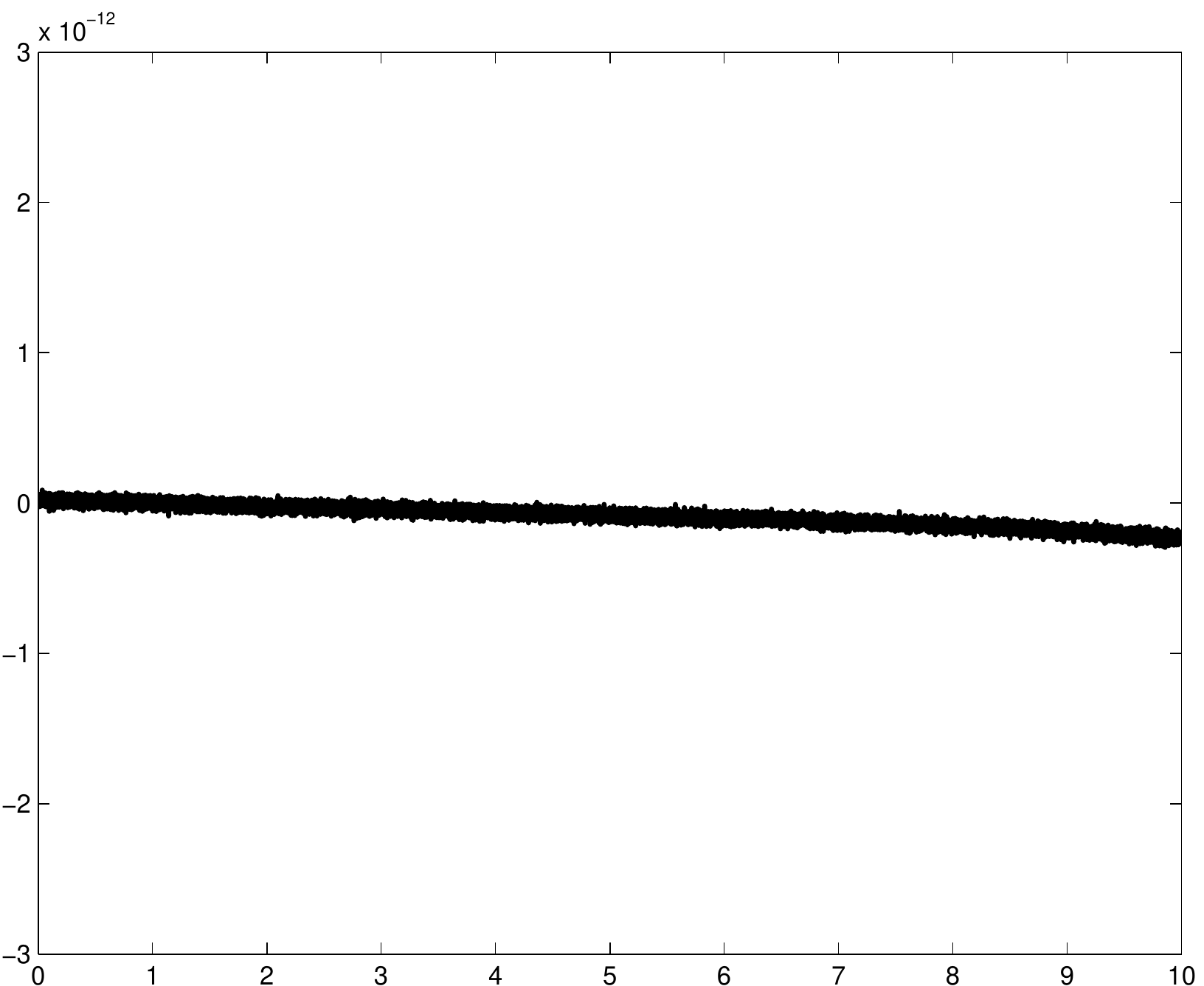}
  \end{minipage}
  \begin{minipage}{5cm}
  \includegraphics[width=1\textwidth]{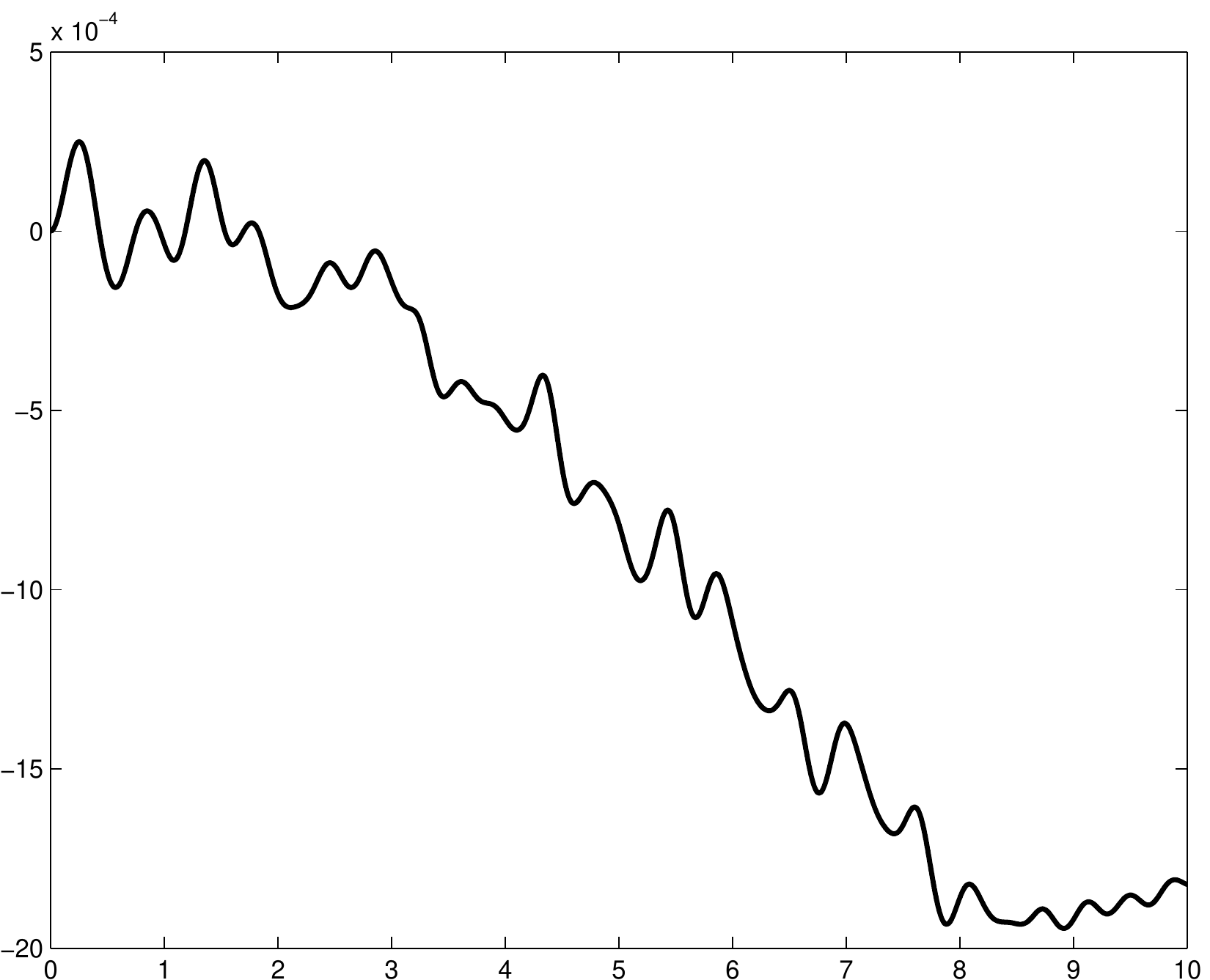}
  \end{minipage}
  \caption{\small Example \ref{polar}: {Same as Fig. \ref{fig:steadycon3} except for example \ref{polar} with $N=54$.}}
  \label{fig:polarcon}
\end{figure}

\begin{example}[Instable barotropic  jet flow]\label{baro}\rm
This instable barotropic jet flow  introduced in \cite{Galewsky: 2004}
 is similar to Williamson's test 2,  but more difficult and challenging  for the numerical methods due to the instable wave structure within a more narrow zonal region and  the dynamic balance in the solutions.
 %, thus  this test is very challenging for the numerical methods.
%due to numerical errors
Specially, it has a great challenge to the numerical methods on the
cubed-sphere grid in Fig. \ref{fig:sphere} (a)-{(c)},
see \cite{StCyr: 2008}, because
 the jet flow is driven by a relatively mild perturbation and then passes over the cubed-sphere edges several times in a long time range.

Initially, the zonal velocity fields are chosen as follows
\begin{align*}
u_{s}(\xi,\eta,0)=&
\begin{cases}
u_{\max}e_{n}^{-1}\exp\left[\frac{1}{\left(\eta-\eta_{0}\right)\left(\eta-\eta_{1}\right)}\right],&\mbox{if}\  \eta_{0}<\eta<\eta_{1},\\
0,&\mbox{otherwise},
\end{cases}\\
v_{s}(\xi,\eta,0)=&0,
\end{align*}
while the  balanced height $h$ is calculated by the following balance equation
\[
h\left(\xi,\eta,0\right)=h_{0}-g^{-1}\int_{-\frac{\pi}{2}}^{\eta}Ru_{s}\left(\eta'\right)\left[f+\frac{\tan\eta'}{R}u_{s}
\left(\eta'\right)\right]d\eta',
\]
where $h_{0}=1000$ m, $u_{\max}=80\mbox{ m/s}$, $\eta_{0}=\frac{\pi}{7}$, $\eta_{1}=\frac{\pi}{2}-\eta_{0}$, and
$e_{n}=\exp\left(-\frac{4}{\left(\eta_{1}-\eta_{0}\right)^2}\right)$.
In order to initiate the instability,   an initial perturbation
\[
h'(\xi,\eta)=\hat{h}\cos(\eta)\exp\left[-\left(\frac{\xi}{\alpha}\right)^{2}-\left(\frac{\eta_{2}-\eta}{\beta}\right)^{2}\right],
\]
is added to the above  balanced height $h\left(\xi,\eta,0\right)$,
where $\hat{h}=120$ m, $\alpha=\frac{1}{3}$, $\beta=\frac{1}{15}$, and $\eta_{2}=\frac{\pi}{4}$.
It implies that the initial height $h$ has a   large  gradient near the cubed-sphere edges in Fig. \ref{fig:sphere} (a)-{(c)}, and the initial perturbation is located on the edge shared
by the subregions $P_1$ and $P_5$ in Fig. \ref{fig:sphere} (a). %\cite{Chen:2008,Chen:2014}.

Fig.~\ref{fig:barocompare} compares
the  relative vorticities $\varsigma$ at $t = 6$ days obtained by using
${ \mathbb{P}^{3}}$-based RKDLEG method, ${ \mathbb{P}^{3}}$-based RKDG method with {Godunov's} flux, and FVLEG method with {fifth-order accurate weighted essentially non-oscillatory (WENO5) reconstruction} with $N=32$.
 It is obvious that the RKDG method with {Godunov's} flux is influenced by 4-wave errors within the longitude interval $[-260,-120]$, %\cite{Chen:2008}
and the FVLEG method with WENO5 reconstruction can not give the correct result in comparison to the reference solution in \cite{Galewsky: 2004} due to the fast growing error  as shown in Fig. \ref{fig:errorcompare}. %{\color{red}{example}} \ref{time}.
 Fig.~\ref{fig:barocon} gives the relative {{conservation}} errors
 at $t = 6$ days in the total mass, total energy and potential enstrophy
 obtained by using the ${ \mathbb{P}^{3}}$-based RKDLEG method with $N=32$.
 It is seen that the total mass is conservative, and
  other errors  are about  $O(10^{-6} )$ and $O(10^{-3} )$, respectively.
%
%  The long time evolution of the unstable barotropic jet flow is useful to demonstrate the conservation properties of an shallow water model.
Fig.~\ref{fig:baro}  further {investigates} the convergence of
 the ${\mathbb{P}^{3}}$-based RKDLEG method, where
 the relative vorticities are obtained on the finer grids of  $N=64$, $96$, and $128$, respectively.
%The result with $N=32$ has already not influenced by the 4-wave errors in relative vorticity field .
Those results look very similar to the reference solution in \cite{Galewsky: 2004} except {{for}}  very little oscillation.

\end{example}

\begin{figure}[htbp]
   \centering
  \includegraphics[width=1 \textwidth]{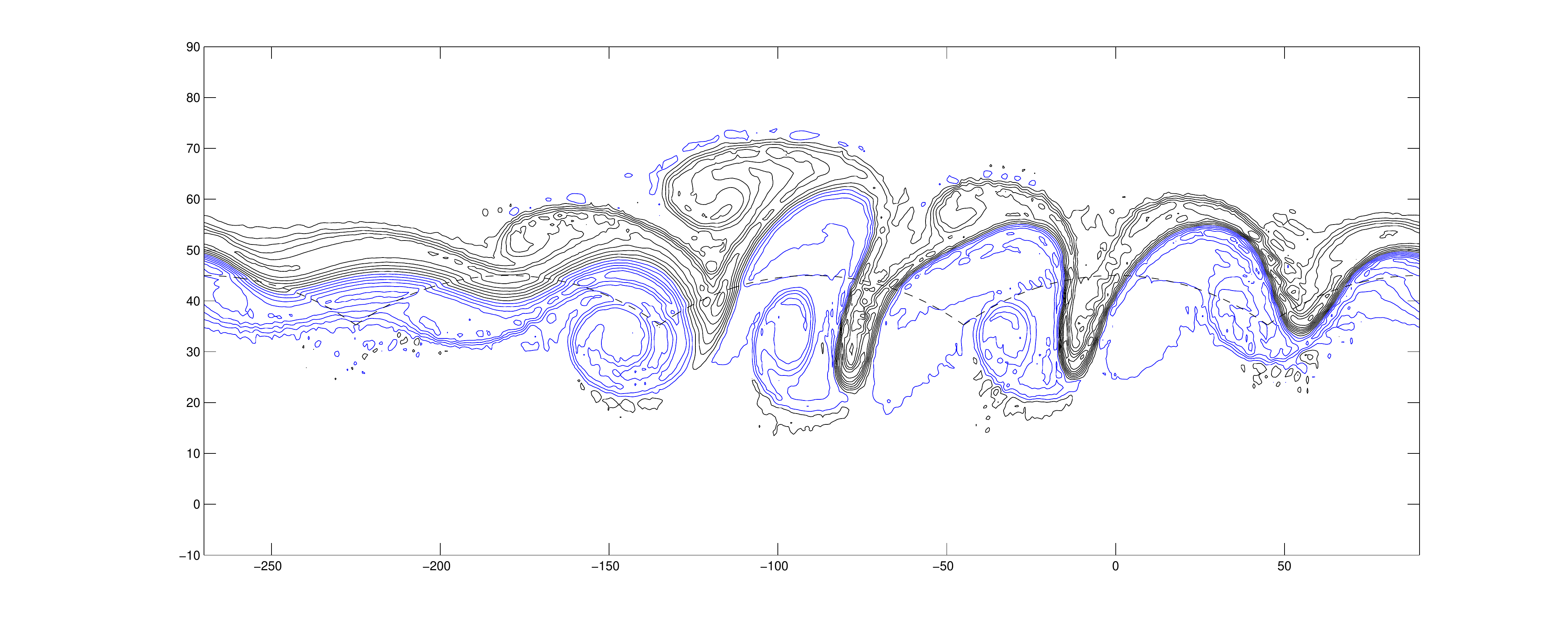}

  \centering
  \includegraphics[width=1 \textwidth]{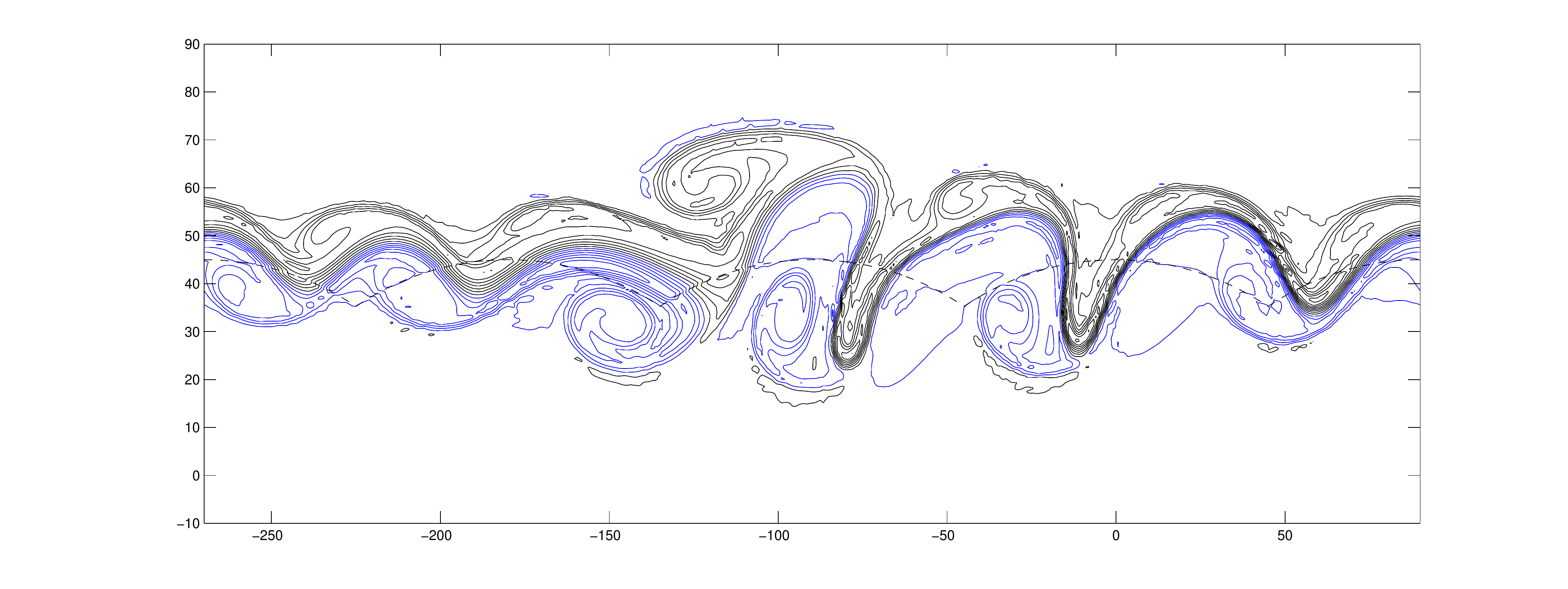}

  \centering
  \includegraphics[width=1 \textwidth]{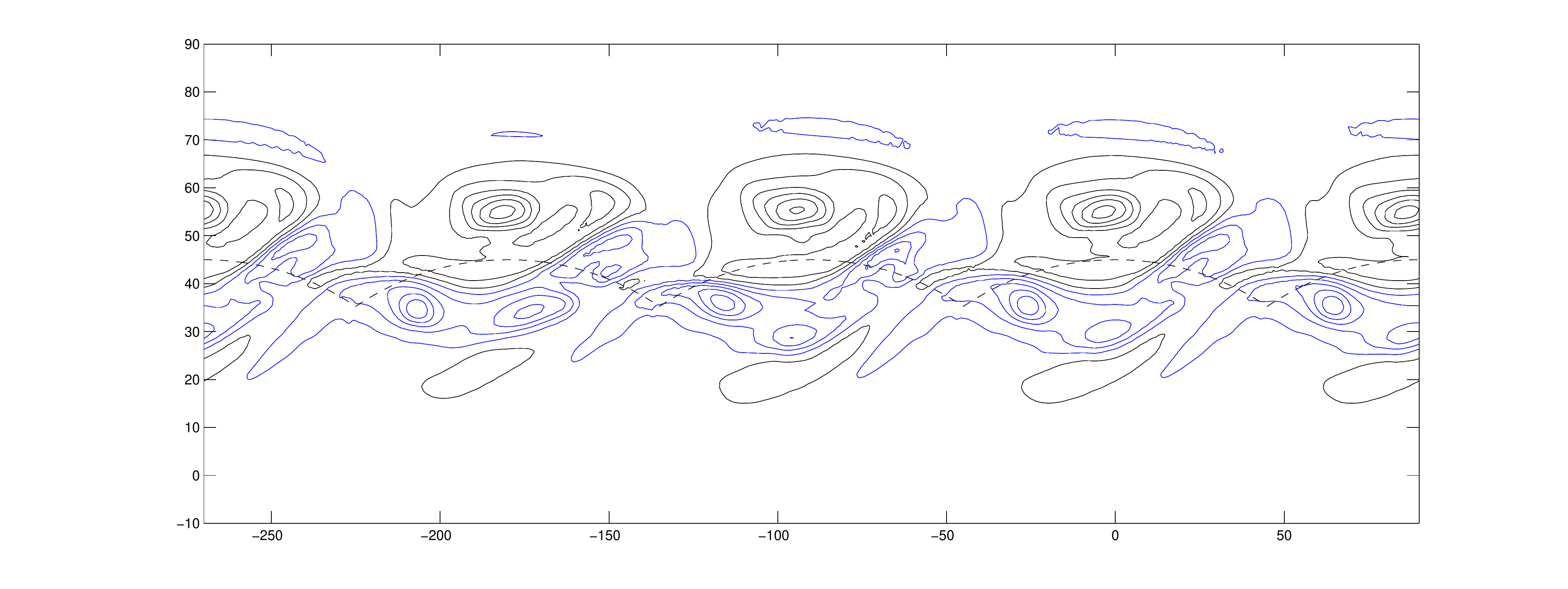}
  \caption{\small Example \ref{baro}: Relative vorticities at $t = 6$ days obtained by using by ${ \mathbb{P}^{3}}$-based RKLEG method, ${\mathbb{P}^{3}}$-based RKDG method with {Godunov's} flux and FVLEG method with WENO5 reconstruction (from top to bottom) with $N = 32$. Contour lines are equally spaced from $-1.1\times 10^{-4} \mbox{ s}^{-1}$ to $-1\times 10^{-5}\mbox{ s}^{-1}$ in dashed lines and from $1\times 10^{-5}\mbox{ s}^{-1}$ to $1.5\times 10^{-4}\mbox{ s}^{-1}$ in solid lines with an interval of $2\times 10^{-5}\mbox{ s}^{-1}$.}\label{fig:barocompare}
\end{figure}

\begin{figure}[htbp]
  \centering
  \begin{minipage}{5cm}
  \includegraphics[width=1\textwidth]{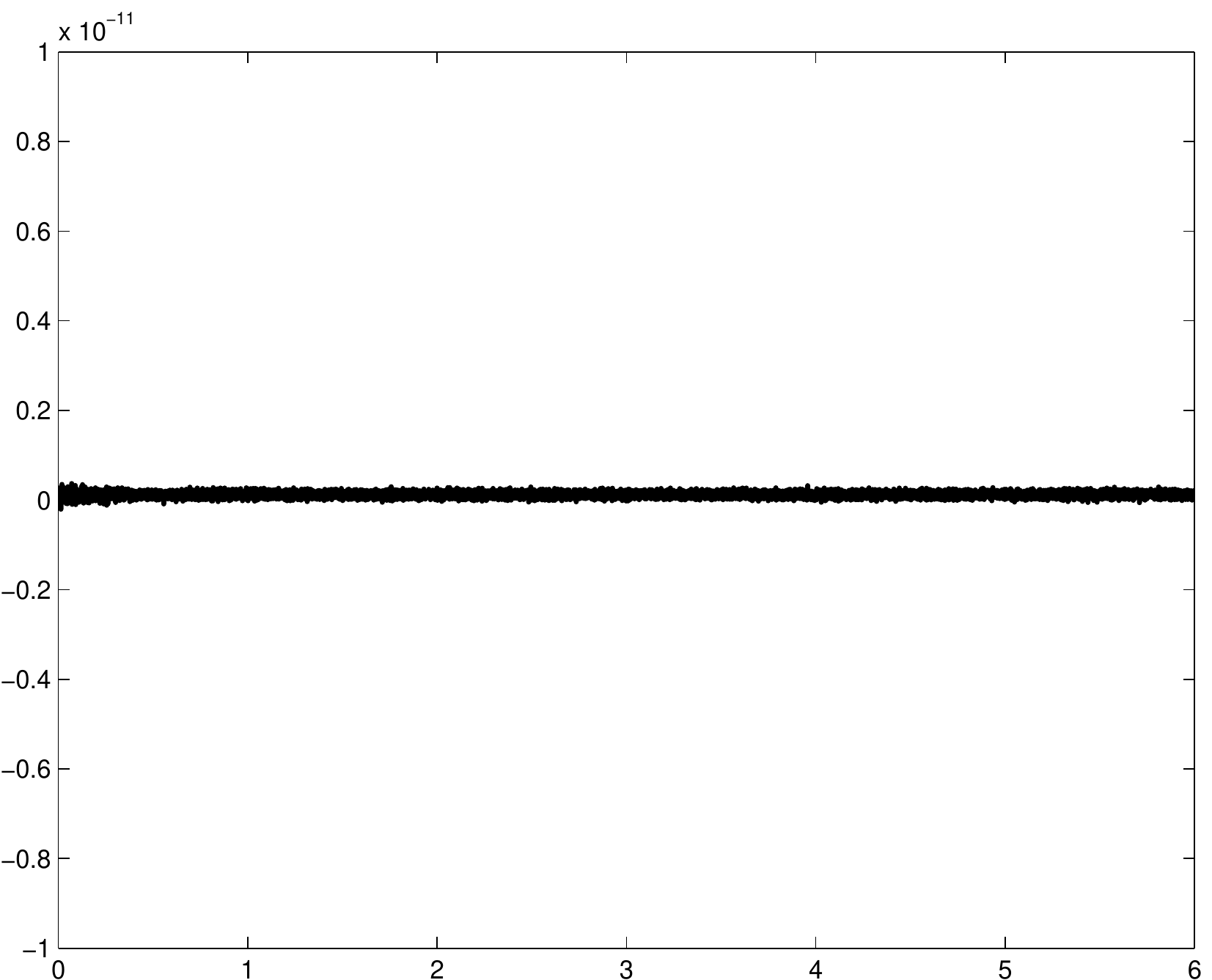}
  \end{minipage}
  \begin{minipage}{5cm}
  \includegraphics[width=1\textwidth]{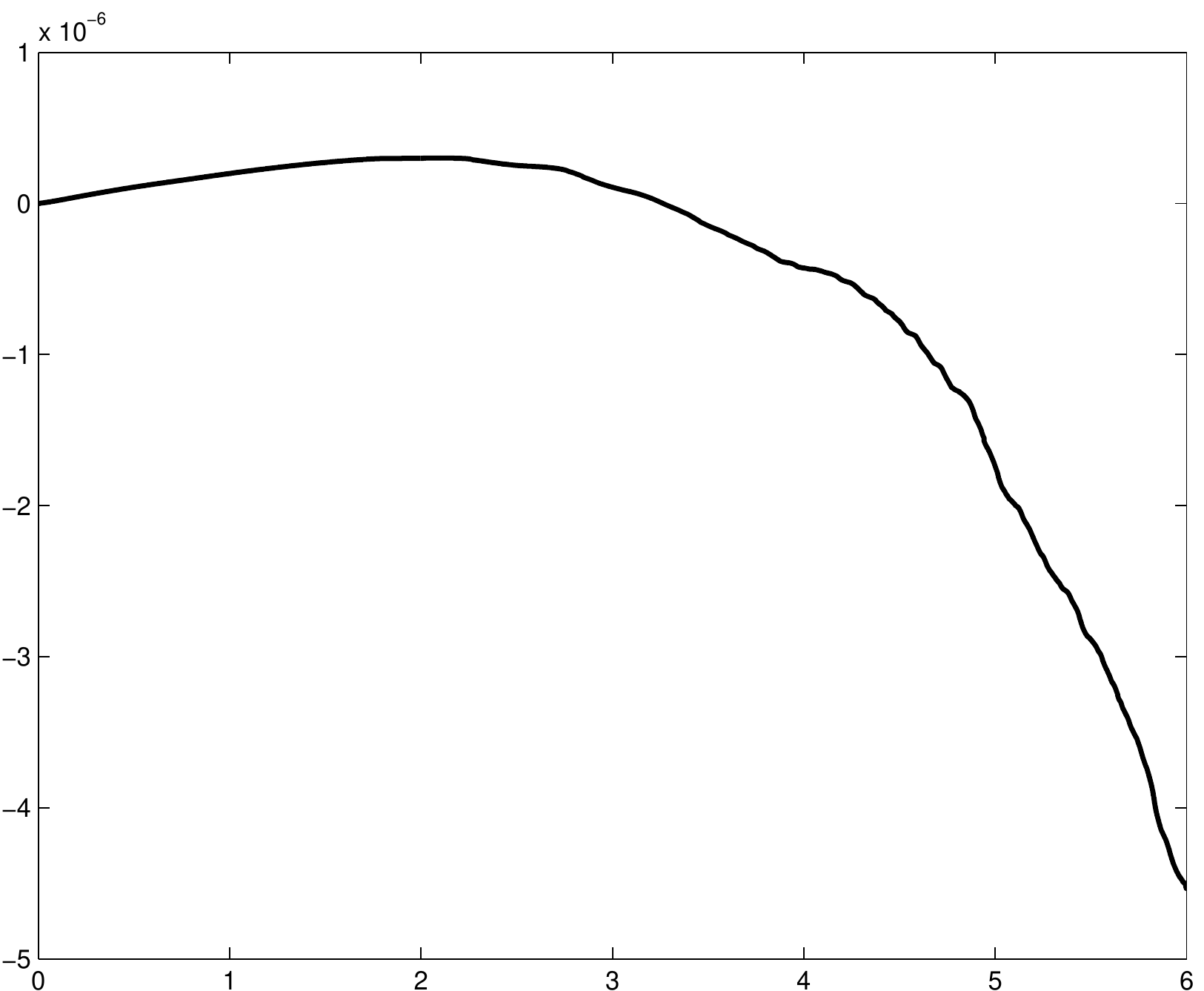}
  \end{minipage}
  \begin{minipage}{5cm}
  \includegraphics[width=1\textwidth]{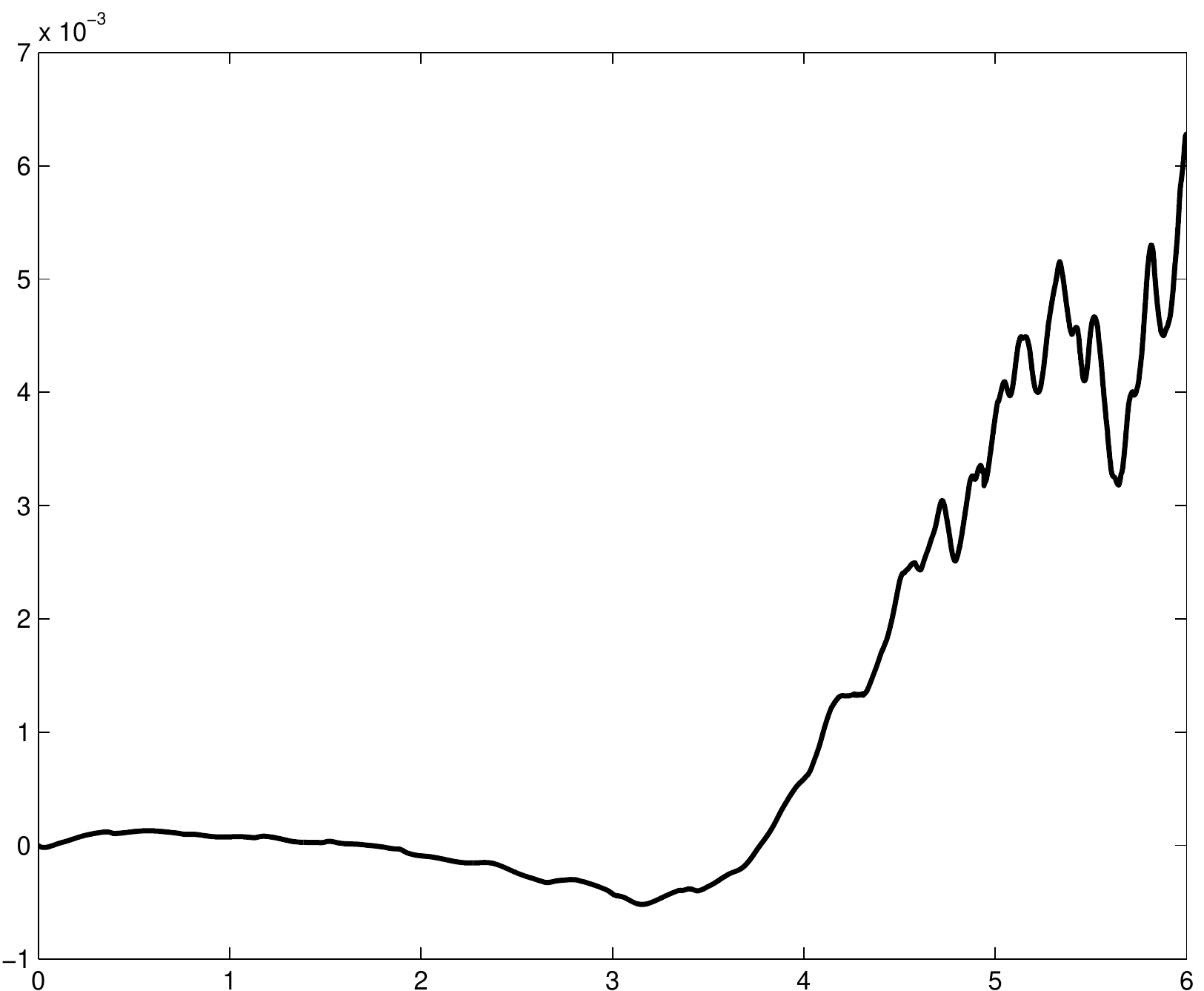}
  \end{minipage}
  \caption{\small Example \ref{baro}: {Same as Fig. \ref{fig:steadycon3} except for example \ref{baro} with $N=32$.}}
  \label{fig:barocon}
\end{figure}

\begin{figure}[htbp]
  \centering
  \includegraphics[width=1 \textwidth]{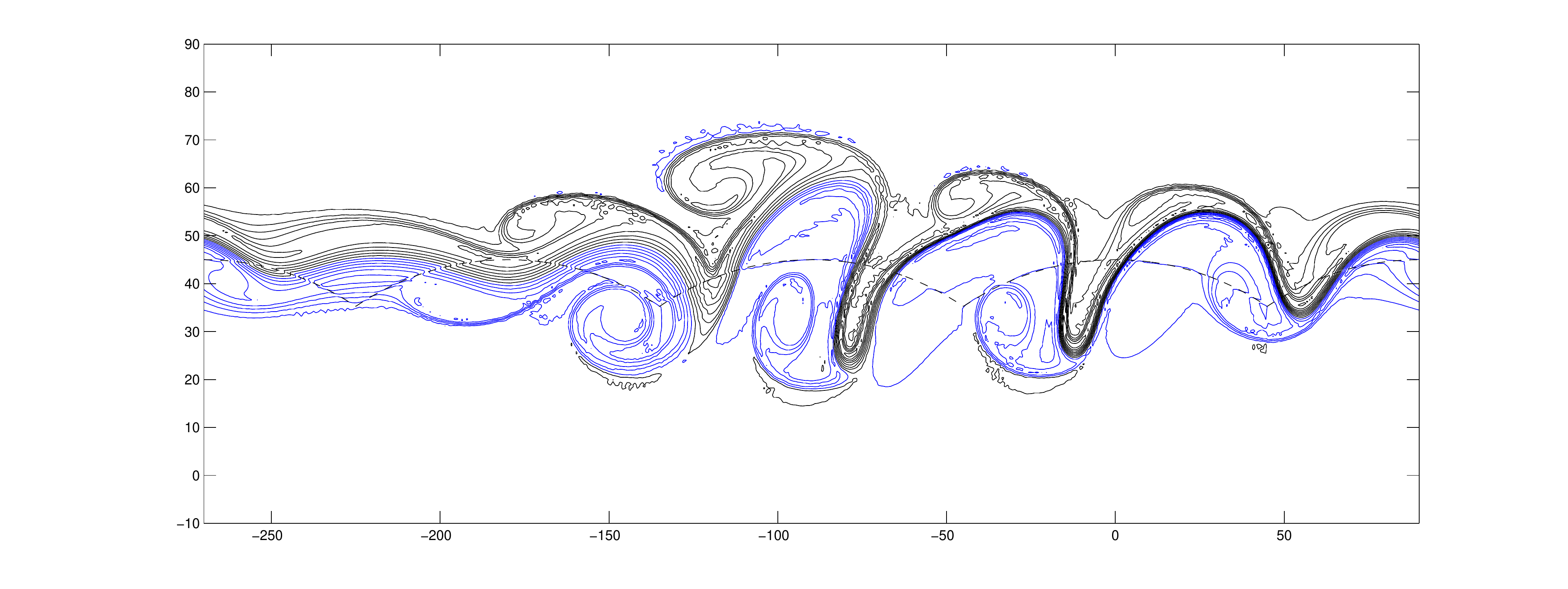}

   \centering
  \includegraphics[width=1 \textwidth]{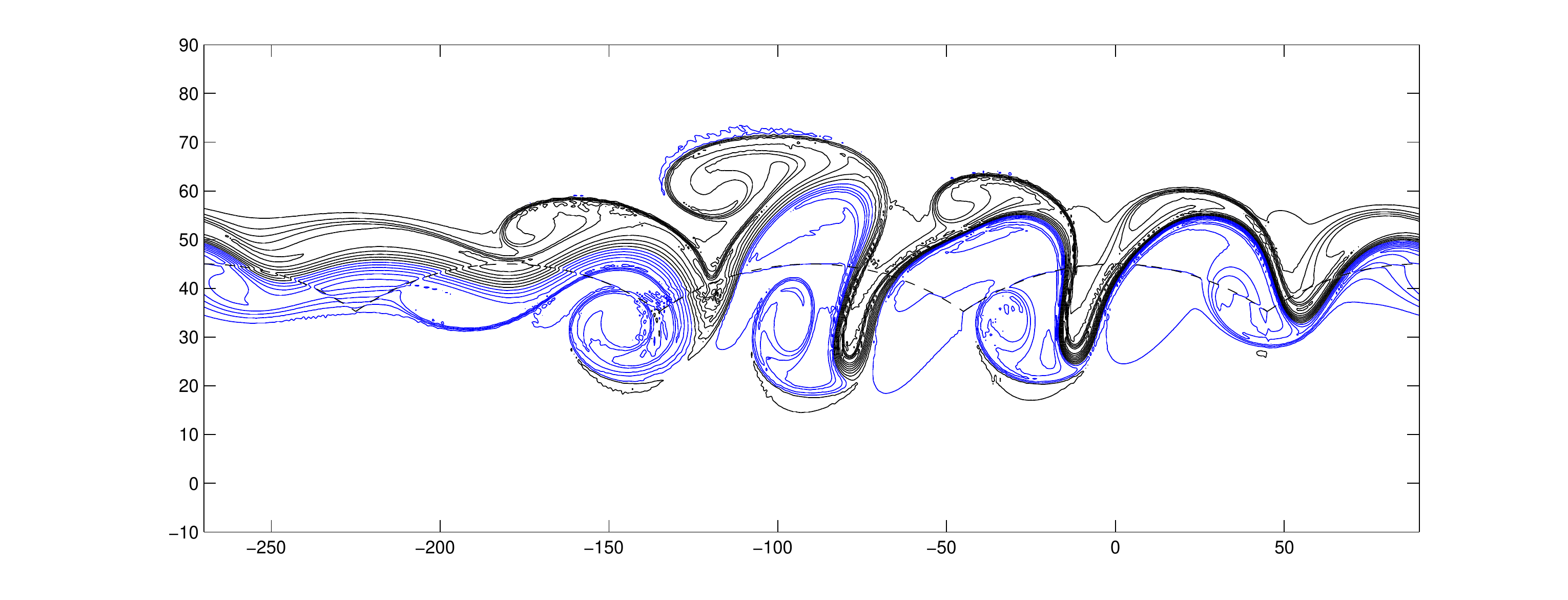}

  \centering
  \includegraphics[width=1 \textwidth]{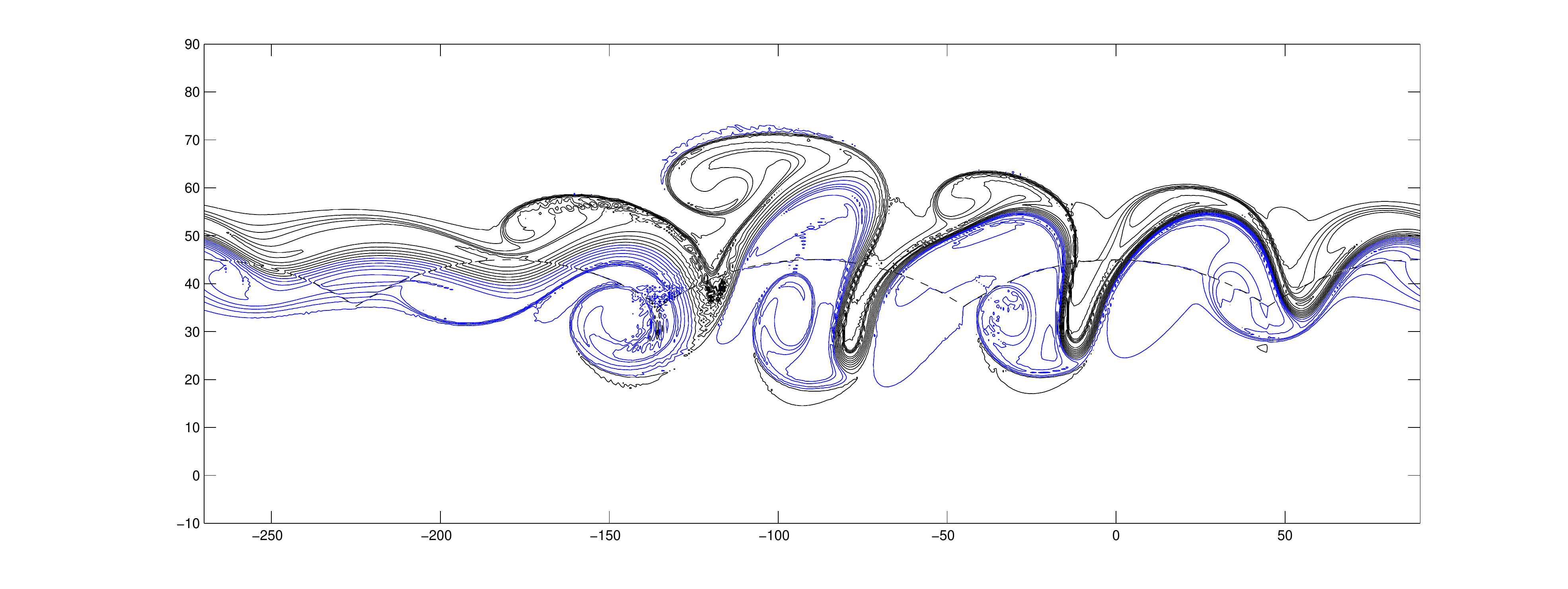}
  \caption{\small Example \ref{baro}: Relative vorticity obtained by using ${\mathbb{P}^{3}}$-based RKLEG method with $N = 64$, $96$, and $128$ (from top to bottom) at $t=6$ days.
  %Contour lines are equally spaced from $-1.1\times 10^{-4} \mbox{ s}^{-1}$ to $-0.1\times 10^{-4}\mbox{ s}^{-1}$ by dashed lines and from $0.1\times 10^{-4}\mbox{ s}^{-1}$ to $1.5\times 10^{-4}\mbox{ s}^{-1}$ by solid lines with an interval of $0.2\times 10^{-4}\mbox{ s}^{-1}$.
  }\label{fig:baro}
\end{figure}

\section{Conclusions}\label{sec:conclud}
The paper developed arbitrary high order accurate Runge-Kutta discontinuous local evolution Galerkin (RKDLEG)  methods on the cubed-sphere grid for the shallow water equations {{SWEs}}.
% with  the aid of the reference coordinates.
The exact and approximate  evolution operators of the locally linearized {SWEs} in the reference coordinates
were first derived based on the theory of bicharacteristics,
then the approximate  local evolution operator {was} combined
with the Runge-Kutta discontinuous  Galerkin (RKDG) methods for the SWEs in spherical geometry.
In other words,
 the proposed  RKDLEG methods were built on genuinely multi-dimensional approximate local evolution operator of the locally linearized {SWEs} in the spherical geometry by considering all bicharacteristic directions,
 instead of  the dimensional splitting method or  one-dimensional Riemann solver in the direction normal to the cell interface.
A special treatment on the edges of the cubed sphere face
was given, where
the approximate local evolution operator of the SWE{{s}} in the
 LAT/LON coordinates  was replaced with that of the SWE{{s}} in the
 reference coordinates  in order to
 ensure the conservation of numerical flux there.
%
% The approximate local evolution operator of the SWE{\color{red}{s}} in the
% reference coordinates at the point in the cubed sphere face,
%
 %
  Several benchmark problems were numerically solved to check the accuracy and performance of our RKDLEG methods, in comparison to the RKDG method with Godunov's flux etc.
  The results showed that
  in comparison to  the {RKDG} method with {Godunov's} flux  and {FVLEG} method ,
  the proposed RKDLEG methods were competitive to
  solve those standard tests
  of Williamson et al. in terms of accuracy, good multi-dimensional behavior, and long time simulation.

%first we evolving the solutions of corresponding locally linearized {\color{red}{SWEs}} on the cubed-sphere grid along all of the infinitely bicharacteristic directions to get the exact operator. Then using suitable approximating for the ``source'' terms, the approximate evolution operator was given, and to simplify the evaluation, approximate local evolution operator was given through taking limit of the approximate evolution operator at the time level $t_{n}$ as the time $t_{n}+\tau$ approached $t_{n}$. Due to the reference coordinates, we had special treatment with the proof of the conservation of the numerical fluxes on the patch boundaries.

\section*{Acknowledgements}
This work was partially supported by the National Natural Science Foundation of China (Nos.
91330205 \& 11421101).

% Special project ``High performance computing''  of the national key research and development program (No. 2016YFB0200603)

\begin{appendices}
\section{}
\label{sec:AppendixA}
This appendix presents the detailed procedure to evaluate $\{\theta_{i},i=1,2,\cdots,\hat{N}\} $ in the
definition of  approximate  evolution operator $\mathcal{E}_{h}(\tau) $ or $\mathcal{E}_{h,0} $
   in Section %\ref{subsec:Approximate}
\ref{subsubsec:nonlocal} or \ref{subsubsec:local}.
%for the inner points of each patch for $0<\tau\leq\Delta t_{n} $ which is presented
% in the $(x, y)$ space

\subsection{The inner points on the cell edge within the subregion} %  $ \partial C_{j+\frac{1}{2},k+\frac{1}{2}}$ }
\label{case1}
%Consider the evolution operator $\mathcal{E}_{h}(\tau) $ at the inner points
%on the cell edge  $ \partial C_{j+\frac{1}{2},k+\frac{1}{2}}$.
 Without loss of generality,    consider  the point $(x_{m}^{G},y_{k},t_{n}) $, denoted by ${\tt P_{0}}$, on the bottom edge
 of cell $ C_{j+\frac{1}{2},k+\frac{1}{2}} $.
 %while the case{\color{red}{s}} of other edges {\color{red}{are}} completely similar.
 Only the edge $ \mathcal{L}_{{\tt P}_{0}} =\left\{\left(x,y_{k}\right)|x_{j-1}\leq x\leq x_{j}\right\} $   intersects with the closed curve
  $\mathcal{C}_{{\tt P}}^{n}=\left\{\left(x_{j}-d_{1}^{(\ell)}(\theta)\tau,y_{k}-d_{2}^{(\ell)}(\theta)\tau, t_{n}\right) ={\tt Q}(\theta)|\ell=1,3,\theta\in\left[0,2\pi\right)\right\}
$ possibly,  under the restriction on $\Delta t_{n}$ in \eqref{delta t}. Obviously,   the number of  interaction points between the edge $ \mathcal{L}_{{\tt P_{0}}} $
with the closed curve $ \mathcal{C}_{{\tt P}}^{n}  $ is equal to the number of real solutions to the algebraic equation $d_{2}^{(1)}(\theta)=0 $.
\begin{lemma}
\label{lemma:d2theta}
If the inequality
\begin{equation}
\label{condition1}
|\tilde{v}|<\tilde{c}\sqrt{\tilde{g}^{22}\tilde{h}},
\end{equation}
holds, then $d_{2}^{(1)}(\theta)=0 $ has two real solutions, which are located in the interval $(-\frac{\pi}{2},\frac{\pi}{2})$ and $(\frac{\pi}{2},\frac{3\pi}{2})$, respectively; otherwise, it has less than two real solutions.
\end{lemma}
\begin{proof}
Since $d_{2}^{(1)}(\theta) $ is a $2\pi$-periodic function of $\theta$,  our attention may be restricted to the interval $[-\frac{\pi}{2},\frac{3\pi}{2}] $. Taking the derivative of $d_{2}^{(1)}(\theta) $ with respect to $\theta$ gives
\[
\frac{d}{d\theta}d_{2}^{(1)}(\theta)= -\frac{\tilde{c}}{\tilde{\Lambda}\tilde{K_{\theta}}^{3} } \cos\theta.
\]
Because $ -\frac{\tilde{c}}{\tilde{\Lambda}\tilde{K_{\theta}}^{3} }<0$, the function $d_{2}^{(1)}(\theta) $ decreases monotonically in the interval $[-\frac{\pi}{2},\frac{\pi}{2}] $ and increases monotonically in the interval $[\frac{\pi}{2},\frac{3\pi}{2}] $ so that the maximum and minimum values of $d_{2}^{(1)}(\theta) $ are
\begin{align*}
\max_{\theta\in[-\frac{\pi}{2},\frac{3\pi}{2}]} d_{2}^{(1)}(\theta) &= d_{2}^{(1)}(-\frac{\pi}{2})=\tilde{v}+\tilde{c}\sqrt{\tilde{g}^{22}\tilde{h}},\\
\min_{\theta\in[-\frac{\pi}{2},\frac{3\pi}{2}]} d_{2}^{(1)}(\theta) &= d_{2}^{(1)}(\frac{\pi}{2})=\tilde{v}-\tilde{c}\sqrt{\tilde{g}^{22}\tilde{h}}.
\end{align*}
Therefore, the sufficient and necessary condition for that $d_{2}^{(1)}(\theta)=0 $ has two
real solutions is
\[
\min_{\theta\in[-\frac{\pi}{2},\frac{3\pi}{2}]} d_{2}^{(1)}(\theta)<0<\max_{\theta\in[-\frac{\pi}{2},\frac{3\pi}{2}]} d_{2}^{(1)}(\theta),
\]
which is equivalent to \eqref{condition1}. It completes the proof.
\qed\end{proof}

With the aid of {{Lemma}} \ref{lemma:d2theta},  $\theta_{i} $  may be evaluated as follows:
\begin{itemize}
\item If the inequality \eqref{condition1} holds, then the nonlinear equation $d_{2}^{(1)}(\theta)=0 $ is iteratively solved by using Newton's method
\[
\theta^{(m+1)}=\theta^{(m)}-\frac{d_{2}^{(1)}(\theta^{(m)})}{ \frac{d}{d\theta}d_{2}^{(1)} (\theta^{(m)})},
\ m=0,1,2,3,\cdots,
\]
with the initial guesses $\theta^{(0)} =0$ and $\pi$ respectively to get two approximate solutions $\theta_{1}$ and $\theta_{2} $, and set $\hat{N}=2$ in the approximate evolution operator $\mathcal{E}_{h}(\tau) $ or $\mathcal{E}_{h,0} $.
\item If the inequality \eqref{condition1} does not hold, then set $\hat{N}=1 $ and $\theta_{1}=0$ in the approximate   evolution operator $\mathcal{E}_{h}(\tau) $ or $\mathcal{E}_{h,0}$.
\end{itemize}

\subsection{The end points on the cell edge  within the subregion} %  $ \partial C_{j+\frac{1}{2},k+\frac{1}{2}}$}
\label{case2}
%Consider the approximate (or approximate local) evolution operator $\mathcal{E}_{h}(\tau)$ (or $\mathcal{E}_{h,0}$) at the mesh point of the edge of cell $ C_{j+\frac{1}{2},k+\frac{1}{2} } $.
Use ${\tt P_{0}}$ and ${\tt P}$ to denote the grid points $\left(x_{j},y_{k},t_{n}\right)$ and  $\left(x_{j},y_{k},t_{n}+\tau\right) $  respectively, and
%As mentioned in Section \ref{subsubsec:nonlocal},
 consider the possible intersection points between the bicharacteristic cone past ${\tt P} $ and the four cell edges $\mathcal{L}_{{\tt P}_{0}}^{\ell},\ell=1,2,3,4$.
The angles $\theta_{i}$ corresponding to those intersection points  may be evaluated by solving the equations
$d_{1}^{(1)}(\theta)=0$ and $d_{2}^{(1)}(\theta)=0$,
respectively.
Similarly, for the first equation,  the following conclusion holds.
\begin{lemma}
\label{lemma:d1theta}
If the inequality
\begin{equation}
\label{condition2}
\left|\tilde{u}\right|<\tilde{c}\sqrt{\tilde{g}^{11}\tilde{h}},
\end{equation}
holds, then the equation $d_{1}^{(1)}(\theta)=0 $ has two real solutions, which are located in the interval $\left(0,\pi\right)$ and $\left(\pi,2\pi\right)$, respectively; otherwise, it has less than two real solutions.
\end{lemma}
 With the help of {{Lemmas}} \ref{lemma:d2theta} and \ref{lemma:d1theta},
 the value of $\theta_{i}$ may be  evaluated as follows:
\begin{itemize}
\item If both \eqref{condition1} and \eqref{condition2} do not hold, then set $\hat{N}=1 $, $\theta_{1}=0 $.

\item If \eqref{condition1} holds while \eqref{condition2} does not hold, then iteratively solve $d_{2}^{(1)}(\theta)=0 $ with Newton's method to get its two solutions denoted by $\theta_{1}$ and $\theta_{2} $ with the initial guesses $0$ and $\pi$ respectively, and set $\hat{N}=2 $.

\item If \eqref{condition2} holds while \eqref{condition1} does not hold, then iteratively solve $d_{1}^{(1)}(\theta)=0 $ with Newton's method to get its two solutions denoted by $\theta_{1}$ and $\theta_{2} $ with the initial guesses $\frac{\pi}{2}$ and $\frac{3\pi}{2}$ respectively, and set $\hat{N}=2 $.

\item If both \eqref{condition1} and \eqref{condition2} hold,
then iteratively solve  $d_{2}^{(1)}(\theta)=0 $ and $d_{1}^{(1)}(\theta)=0 $ with Newton's method to obtain four angles,
which are labeled in anticlockwise direction
as $\theta_{1},\theta_{2},\theta_{3},\theta_{4} $, and set $\hat{N}=4$.
\end{itemize}

% % % % % % % % % % % % % % % % % % % % % % % % % % % % % % % % % % % % % % % % % % % %5
\section{}
\label{sec:AppendixC}
This appendix presents the detailed procedure to evaluate $\left\{\theta_{i},i=1,2,\cdots,\hat{N}\right\} $ in the approximate  evolution operator $\mathcal{E}_{h}(\tau) $ or $\mathcal{E}_{h,0}$
for   the subregion boundaries  in the  LAT/LON $(\xi,\eta )$ plane,
 %for $0<\tau\leq\Delta t_{n} $ which is presented
 see Section \ref{subsubsection:edges}.
%
%Since we want {\color{red}{to}} obtain the the approximate (or approximate local) evolution operator $\mathcal{E}_{h}(\tau) $(or $\mathcal{E}_{h,0} $) for SWE{\color{red}{s}} on LAT/LON coordinates, corresponding angle of the intersection points between the curves which are the projection of the cell edges in $(x,y)$ and the bottom of the bicharacteristic cone should be evaluated.
%there are 12 subregion boundaries,
For the sake of convenience,  %On the cubed-sphere grid,  see Fig. \ref{fig:sphere} (a) or (b),
use $\mathcal{L}_{i}$ to denote the  subregion boundary between the $i$th and   $\big((i+1)\mod4\big)$th subregions,
  $i=1,2,3,4$. Also use $\mathcal{L}_{i}^{up} $ (resp. $\mathcal{L}_{i}^{down} $)
to  denote the subregion boundary between $i$th and  $5$th (resp. $6$th) subregions,
 $i=1,2,3,4$.

The subregion boundaries $\mathcal{L}_{1}$ and $\mathcal{L}_{1}^{up} $ are only discussed in the following,
because other subregion boundaries may be similarly treated by using the  translation and reflection transformations.

% % % % % % % % % % % % % % % % % % % % % % % % % % % % % % % % % % % %
\subsection{The inner points of the cell edge on the subregion boundary}

Since the cell edges  $ \partial C_{j+\frac{1}{2},k+\frac{1}{2}}$ on  $\mathcal{L}_{1}$
in  the LAT/LON $(\xi,\eta) $ plane are just part of the longitude lines (i.e. $\xi=$const),
the number of  interaction points between the edge $ \mathcal{L}_{{1}} $
and the bottom of the bicharacteristic cone is equal to the number of real solutions to the algebraic equation $d_{2}^{(1)}(\theta)=0 $
%  the intersection points between the curves which are the projection{\color{red}{s}} of the cell edges and the bottom of the bicharacteristic cone {\color{red}{are}} still contingent on whether or not the algebraic equation $d_{1}^{(1)}(\theta)=0 $ has solution.
such that  the method in Appendix \ref{case1} may be directly used to evaluate the value of $\theta_i$.

The  cell edge  $ \partial C_{j+\frac{1}{2},k+\frac{1}{2}}$ on
$ \mathcal{L}_{1}^{up}$ in  the LAT/LON $(\xi,\eta) $ plane
 consists of those  curves satisfying the equation
 \begin{equation}\label{EQ-App-B-01}
 \tan\eta=\cos\xi\tan y,
 \end{equation}
 where $y=y_{\hat{N}}$ and $\xi\in[-\frac{\pi}{4},\frac{\pi}{4}]$.
 %, and $\left(\xi,\eta\right)$
 %is the coordinate of  mapping point on the cell edges  $ \partial C_{j+\frac{1}{2},k+\frac{1}{2}}$
 %to the LAT/LON space.
 Use $(\xi_{0},\eta_{0}) $ to
 denote the inner point on the cell edges $(x,y_{N}) $  in
 the LAT/LON $(\xi,\eta)$ plane.
The task is  to get the
intersection points between the curve \eqref{EQ-App-B-01} and
\[
\mathcal{C}_{{\tt P}}^{n}=\left\{\left(x-d_{1}^{(\ell)}(\theta)\tau,y_{N}-d_{2}^{(\ell)}(\theta)\tau,t_{n}\right)=
{\tt Q}(\theta)|\ell=1,3,\theta\in[0,2\pi)\right\}.
\]
Substituting point $\left(\xi_{0}-d_{1}^{(1)}(\theta)\tau,\eta_{0}-d_{2}^{(1)}(\theta)\tau\right) $ into the equation \eqref{EQ-App-B-01} gives
\[
\tan(\eta_{0}-d_{2}^{(1)}(\theta)\tau)
=\cos(\xi_{0}-d_{1}^{(1)}(\theta)\tau)\tan y_{N},
 \]
thus one has
\[
\tan\eta_{0}-\tan(\eta_{0}-d_{2}^{(1)}(\theta)\tau)
=\cos\xi_{0}\tan y_{N}-\cos(\xi_{0}-d_{1}^{(1)}(\theta)\tau) \tan y_{N}.
\]
Using   Lagrange's mean value theorem, and letting $\tau\to 0 $,  one yields
\begin{equation}
\label{curve1}
\sec^{2}\eta_{0} d_{2}^{(1)}(\theta)=-\sin\xi_{0} d_{1}^{(1)}(\theta).
\end{equation}
%Similar to {{Lemmas}} \ref{lemma:d2theta} and \ref{lemma:d1theta},
%Thus we have the following conclusion.

\begin{lemma}
\label{lemma:l1}
If the inequality
\begin{equation}
\label{conditionedge1}
\mathcal{F}_{c}\left(\theta_{\min},\xi_{0},\eta_{0},y_{N}\right) <0< \mathcal{F}_{c}\left(\theta_{\max},\xi_{0},\eta_{0},y_{N}\right),
\end{equation}
holds, where
$\mathcal{F}_{c}(\theta,\xi_{0},\eta_{0},y): =\sec^{2}\eta_{0} d_{2}^{(1)}(\theta)+\sin\xi_{0} d_{1}^{(1)}(\theta)\tan y $,
 then \eqref{curve1} has two real solutions, which are located in the interval $(\theta_{\min},\theta_{\max})$ and $(\theta_{\max},\theta_{\min}+2\pi) $, respectively; otherwise, it has less than two real solutions, where $\theta_{\min} $ satisfies
\[
\sin\xi_{0}\sin\theta\tan y_{N}-\sec^{2}\eta_{0}\cos\theta=\sqrt{\sin^{2}\xi_{0}\tan^{2}y_{N}+\sec^{4}\eta_{0}}\sin(\theta-\theta_{\min}) ,\]
and $\theta_{\max}=\theta_{\min}+\pi$.
\end{lemma}

With  the aid of Lemma \ref{lemma:l1},
the value of $\theta_{i}$ may be  evaluated as follows:
\begin{itemize}
\item If \eqref{conditionedge1} holds, then iteratively solve $\mathcal{F}_{c}\left(\theta,\xi_{0},\eta_{0},y_{N}\right)=0 $ with Newton's method to get its two solutions denoted by $\theta_{1}$ and $\theta_{2} $ with the initial guesses $ \frac{1}{2}(\theta_{\min}+\theta_{\max})$ and $\frac{1}{2}(\theta_{\min}+\theta_{\max}) +\pi$, respectively, and set $\hat{N}=2$.
\item If \eqref{conditionedge1} does not hold, then set $\hat{N}=1$, $\theta_{1}=0$.
\end{itemize}

% % % % % % % % % % % % % % % % % % % % % % % % % % % % % % % % % % % %
\subsection{The end points  of the cell edge on the subregion boundary}
Use ${\tt P_{0}}$, and ${\tt P}$
to denote the grid points {$(x_{N},y_{k},t_{n})$ and $(x_{N},y_{k},t_{n}+\tau) $} on  $\mathcal{L}_{1}$.
% As mentioned in Section \ref{subsubsection:edges}, we only need to consider the possible intersection points between the bottom of the bicharacteristic cone and projection{{s}} of the four $\mathcal{L}_{{\tt P}_{0}}^{\ell},\ell=1,2,3,4 $cell edges on LAT/LON coordinates.
In this case,   the cell edges on $\mathcal{L}_{1}$ mapping to the LAT/LON $(\xi,\eta) $ plane are part of the longitude lines (i.e. $\xi$=const), and the curves  satisfying
\eqref{EQ-App-B-01} and $\tan\eta=\cos(\xi-\frac{\pi}{2})\tan y $ respectively,
where  $y=y_{k} $, {see schematic diagram in Fig \ref{fig:boundary} (a)}.  %, $(\xi,\eta) $ is the projection of a point on the cell edges to LAT/LON coordinates.
Use $(\xi_{0},\eta_{0})$  to denote the  LAT/LON coordinates
of  the end points of the cell edges $(x_{N},y_{k}) $.
Corresponding angles $\theta_i$
of two possible intersection points  between the edge
 (i.e. $\xi$=const) and the bottom of the bicharacteristic cone
may  be  evaluated by solving $d_{1}^{(1)}(\theta)=0$,
while the  angles $\theta_i$ corresponding to other interaction points
are gotten by solving
the following  equations
\[\begin{cases}
 \mathcal{F}_{c}\left(\theta,\xi_{0},\eta_{0},y_{k}\right)=0, &
 d_{1}^{(1)}(\theta)>0, \\
   \mathcal{F}_{c}\left(\theta,\xi_{0}-\frac{\pi}{2},\eta_{0},y_{k}\right)=0, & d_{1}^{(1)}(\theta)<0.
 \end{cases}
\]
%Use {{Lemma}} \ref{lemma:l1} to find out the possible intersection points, and judge the sign of $d_{1}^{(1)}(\theta)$ at the intersection points.
Using  the procedure in   Appendix \ref{case2} gives %\ref{sec:AppendixA},
$\hat{N}$ and $\theta_{i},i=1,\cdots,\hat{N} $.

{{The}}  cell edges on $ \mathcal{L}_{1}^{up}$  mapping to the LAT/LON
plane are the longitude line (i.e. $\xi$=const),
or the curve satisfying  \eqref{EQ-App-B-01} or $\tan\xi\tan x=\sin\lambda $, here $x=x_{j},y=y_{N} $, $\xi\in[-\frac{\pi}{4},\frac{\pi}{4}] $,
and $(\xi,\eta) $ denotes the point on the cell edges in the LAT/LON plane, {see schematic diagram in Fig \ref{fig:boundary} (b)}.
Use $(\xi_{0},\eta_{0})$
 to denote the end point of the cell edges $(x_{j},y_{N}) $ in the LAT/LON space.
 The angles relating  to the intersection points between the cell edges and the bottom of
 the bicharacteristic cone may  be obtained by solving  the equation
 $\mathcal{F}_{c}\left(\theta,\xi_{0},\eta_{0},y_{N}\right)=0$, and
\begin{align*}
\begin{cases}
 d_{2}^{(1)}(\theta)=0, & \mathcal{F}_{c}\left(\theta,\xi_{0},\eta_{0},y_{N}\right)>0,\\
\mathcal{F}_{s}\left(\theta,\xi_{0},\eta_{0},x_{j}\right)=0, & \mathcal{F}_{c}\left(\theta,\xi_{0},\eta_{0},y_{N}\right)<0,
\end{cases}\end{align*}
respectively, where
\begin{equation}
\label{curve2}
\mathcal{F}_{s}\left(\theta,\xi_{0},\eta_{0},x_{j}\right):= \sec^{2}\eta_{0}\tan x_{j}d_{2}^{(1)}(\theta)-\cos\xi_{0} d_{1}^{(1)}(\theta).
\end{equation}

For the equation  $\mathcal{F}_{s}\left(\theta,\xi_{0},\eta_{0},x_{j}\right)=0 $,
 the following conclusion holds.
\begin{lemma}
\label{lemma:l2}
If the inequality
\begin{equation}
\label{conditionedge2}
\mathcal{F}_{s}\left(\theta_{\min},\xi_{0},\eta_{0},x_{j}\right)<0<\mathcal{F}_{s}\left(\theta_{\max},\xi_{0},\eta_{0},x_{j}\right),
\end{equation}
holds, then the equation $\mathcal{F}_{s}\left(\theta,\xi_{0},\eta_{0},x_{j}\right)=0 $ has two real solutions, which are in the interval $(\theta_{\min},\theta_{\max})$ and $(\theta_{\max},\theta_{\min}+2\pi) $, respectively; otherwise, it has less than two real solutions, where $\theta_{\min} $ satisfies
\[
-\sec^{2}\eta_{0}\tan x_{j}\cos\theta-\cos\xi_{0}\sin\theta=\sqrt{\sec^{4}\eta_{0}\tan^{2}x_{j}+\cos^{2}\xi_{0}}\sin(\theta-\theta_{\min}),
\]
and $\theta_{\max}=\theta_{\min}+\pi $.
\end{lemma}
It means that  two solutions of the equation $\mathcal{F}_{s}\left(\theta,\xi_{0},\eta_{0},x_{j}\right)=0 $ may be iteratively gotten
by using Newton's method  with the initial guesses $ \frac{1}{2}(\theta_{\min}+\theta_{\max})$ and $\frac{1}{2}(\theta_{\min}+\theta_{\max}) +\pi$, respectively,
if \eqref{conditionedge2} holds.

% Use {\color{red}{lemma}} \ref{lemma:l2} to find out the possible intersection points, and judge the sign of $ \mathcal{F}_{c}\left(\theta,\xi_{0},\eta_{0},y_{\color{red}N}\right)$ at the intersection points. Similar to Appendix \ref{case2}, $\hat{N}$ and $\theta_{i},i=1,\cdots,\hat{N} $ can be obtained.
Following  the procedure in   Appendix \ref{case2} gives %\ref{sec:AppendixA},
$\hat{N}$ and $\theta_{i},i=1,\cdots,\hat{N} $.
\begin{figure}[htbp]
  \centering
    \subfigure[{\small $\mathcal{L}_{1}$}]{
    	\begin{minipage}[c]{7cm}
    		\includegraphics[width=7cm,height=5cm]{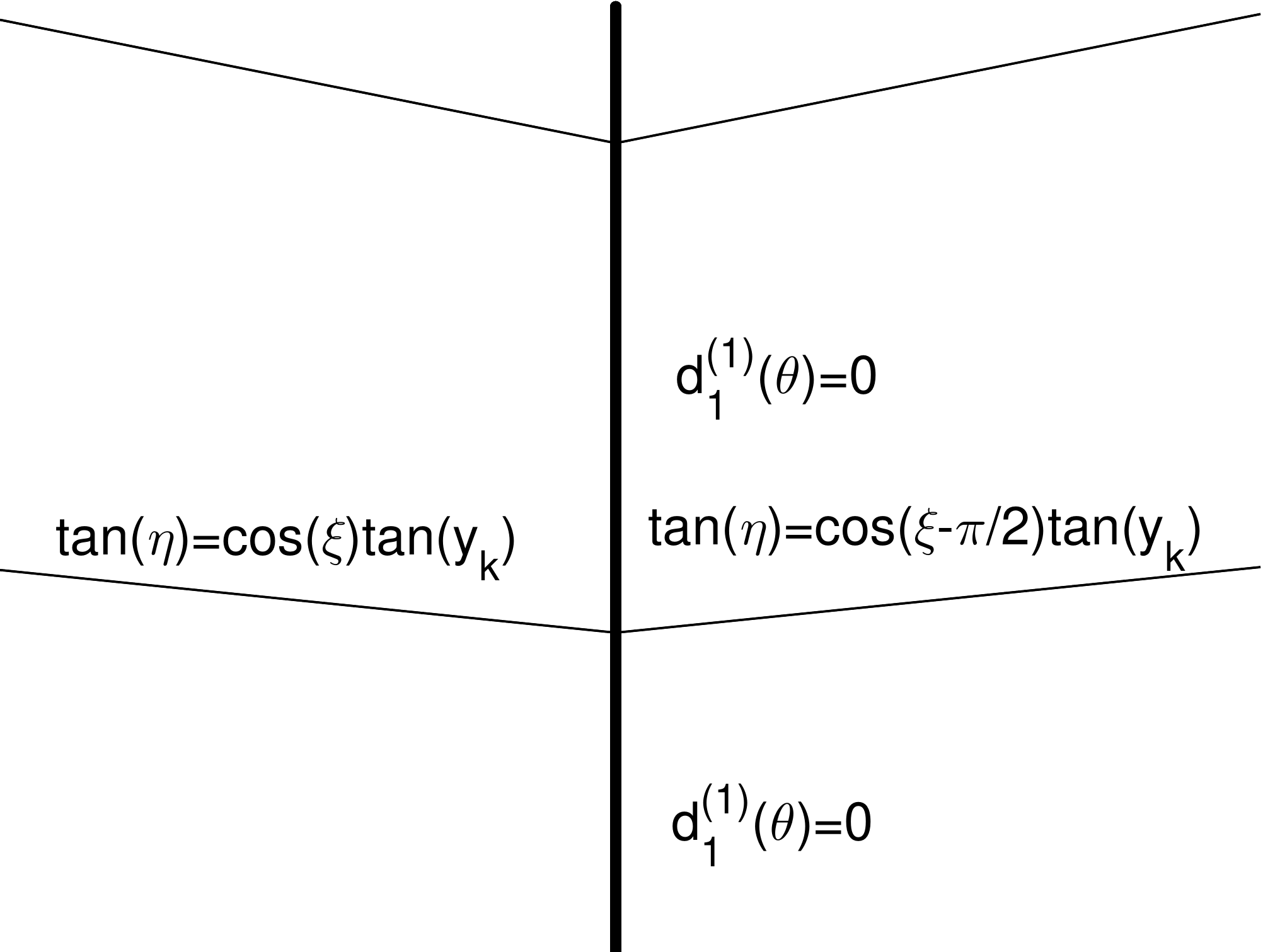}
    	\end{minipage}}\qquad
  \subfigure[{\small $\mathcal{L}_{1}^{up}$}]{
  \begin{minipage}{7cm}
  \includegraphics[width=7cm,height=5cm]{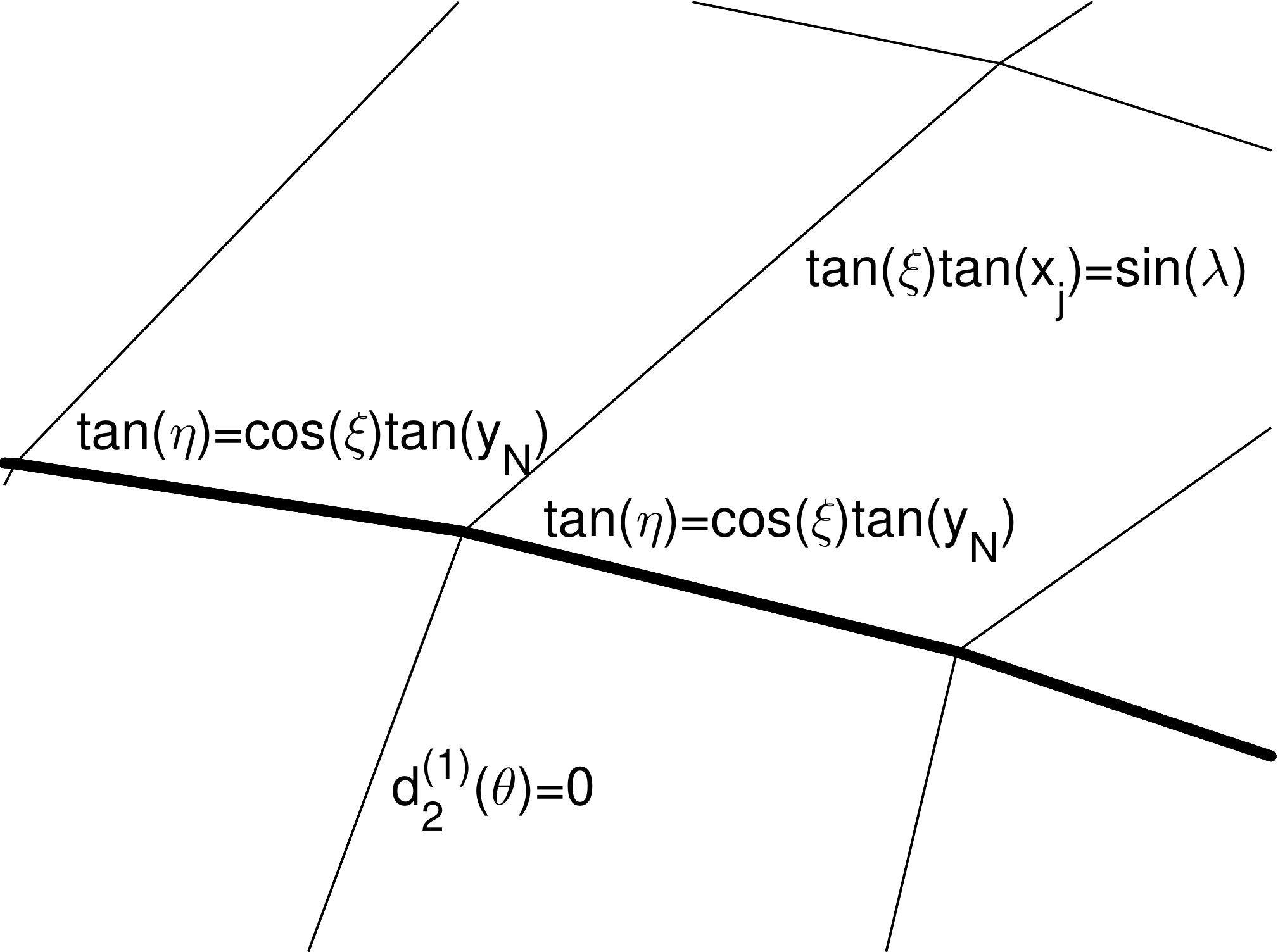}
  \end{minipage}}
  \caption{Schematic diagram of the subregion boundary.}
  \label{fig:boundary}
\end{figure}

\end{appendices}

\end{document}